\documentclass[11pt]{amsart}
\usepackage[utf8]{inputenc}
\usepackage[T1]{fontenc}
\usepackage{amssymb,amsmath,amsthm,mathtools}
\usepackage[a4paper,margin=1in]{geometry}
\usepackage{xcolor}
\usepackage{pgfplots}
\usepackage{booktabs}
\pgfplotsset{compat=1.18}
\usepackage{hyperref}
\hypersetup{colorlinks=true,linkcolor=blue!45!black,citecolor=blue!45!black,urlcolor=blue!45!black}

\theoremstyle{plain}
\newtheorem{theorem}{Theorem}[section]
\newtheorem{proposition}[theorem]{Proposition}
\newtheorem{corollary}[theorem]{Corollary}
\newtheorem{lemma}[theorem]{Lemma}
\theoremstyle{remark}
\newtheorem{remark}[theorem]{Remark}

\newcommand{\R}{\mathbb{R}}
\newcommand{\norm}[1]{\left\lVert #1\right\rVert}
\newcommand{\abs}[1]{\left\lvert #1\right\rvert}
\newcommand{\spec}{\sigma}
\newcommand{\Emin}{E_{\min}}
\newcommand{\Emax}{E_{\max}}
\DeclareMathOperator{\arccosh}{arccosh}
\DeclareMathOperator{\dist}{dist}
\DeclareMathOperator{\diag}{diag}
\newcommand{\Ecal}{\mathcal{E}}
\newcommand{\cpc}{\operatorname{cap}}
\begin{document}

\title[The remote-stabilization cost rate as a reflected off-spectral exponent]
{The cost rate of nonlinear remote stabilization on the Aubry--Andr\'e lattice:
a reflected off-spectral exponent and the sharp identity for almost every phase}
\author{Nassim Athmouni}
\address{Faculty of Sciences, University of Gafsa, BP 2100, Gafsa, Tunisia}
\email{nassim.athmouni@fsgf.u-gafsa.tn}
\author{Nejib Brahmia}
\address{University of Sfax, Faculty of Sciences of Sfax, Department of Mathematics, BP 1171 Sfax 3000, Tunisia}\email{brahmianajib@gmail.com}
\author{Ahmed Hachani}
\address{University of Gafsa, Faculty of Sciences of Gafsa, Department of Mathematics, Zarroug Gafsa 2112, Tunisia}\email{ahmedhachnai@gmail.com}
\begin{abstract}
We study the exponential rate $r(\alpha,\lambda)$ at which the least energy needed to steer a far site of an Aubry--Andr\'e chain, via a single boundary actuator with closed-loop margin $\alpha$, grows with the distance $N$. An exact eigenbasis reduction writes this energy as a Cauchy quadratic form; its rate equals the off-spectral Lyapunov exponent $\gamma_\lambda(z^\star)$ of the transfer cocycle at a reflected band edge $z^\star$, giving the identity $r(\alpha,\lambda)=\gamma_\lambda(z^\star)$.

We prove this identity unconditionally, for every phase, throughout the metallic and critical range $0<\lambda\le2$. In the localized range $\lambda>2$, an inverse-free cocycle form and a Christoffel--Darboux identity reduce the remaining lower bound to localization as its sole external input, giving the identity at every Diophantine frequency and almost every phase under an explicit, numerically supported localization hypothesis, whose band-edge gap component we prove at strong coupling.
\end{abstract}


\maketitle

\section{Introduction}\label{sec:intro}

How much energy does it cost to steer a distant site of a quasiperiodic chain from its boundary, and which
spectral quantity sets the exponential rate of that cost in the distance? On a one-dimensional quasiperiodic
chain the question has a sharp answer, and this paper gives it.

The chain carries the Aubry--Andr\'e (almost Mathieu) operator $H_\lambda$, with quasiperiodic potential
$\lambda\cos(2\pi\beta j+\phi)$ at a Diophantine frequency $\beta$ (the golden mean $(\sqrt5-1)/2$ in the numerics). It is a model with a fully understood
metal--insulator transition: for $\lambda<2$ the spectrum is purely absolutely continuous, for $\lambda>2$
it is pure point with eigenfunctions localized at rate $\log(\lambda/2)$, and $\lambda=2$ is the self-dual
critical coupling \cite{GJLS1997,AvronSimon,Jito1999,Avila2015}. Throughout the spectrum the Lyapunov
exponent of the transfer cocycle equals $\log_+(\lambda/2)$. The control quantity is the least energy
$\Ecal_N=e_N^\top W_\alpha^{-1}e_N$ steering the far unit at distance $N$, where $W_\alpha$ is the
controllability Gramian of the chain shifted by a closed-loop margin $\alpha$, and its exponential rate is
$r(\alpha,\lambda)=\liminf_N\frac1{2N}\log\Ecal_N$. This control system is linear, and it is the local model for
remote stabilization of a \emph{nonlinear} chain: at an operating point where a nonlinear lattice linearizes to the
shifted operator $A_\alpha=E_0I-H_\lambda$ (a hyperbolic equilibrium once the margin $\alpha>0$ is imposed), the
least energy of a small boundary maneuver steering the far site is, to leading order in the maneuver amplitude, the
linear cost $\Ecal_N$, so $r(\alpha,\lambda)$ is the cost rate of local nonlinear remote stabilization. We work with
the linear cost throughout.

A natural guess takes $r$ to be the localization rate, or the off-spectral Lyapunov exponent
$\gamma_\lambda(\Emin-\alpha)$ at the band margin. Neither is correct: already the free chain $\lambda=0$, where
there is no localization at all, has cost rate $\arccosh(3+\alpha)$, several times larger than that margin
exponent. The mechanism is different: $\Ecal_N$ is governed by the smallest eigenvalues of an exponentially
ill-conditioned controllability Gramian, not by a single resolvent column, and the resolvent decay does not see
those eigenvalues.

This paper identifies what does set the rate and, at every Diophantine frequency and almost every phase, establishes the identity throughout
the localized regime. Lemma~\ref{lem:reduction} expresses $\Ecal_N$ as the Cauchy quadratic form
$\tilde b^\top C^{-1}\tilde b$ in the boundary-amplitude ratios, a Zolotarev-type quantity for the shifted
spectrum \cite{BeckermannTownsend}. Its rate is the off-spectral exponent of the same cocycle at the
\emph{reflected band edge} $z^\star=2\Emin-2\alpha-\Emax$: Theorem~\ref{thm:exact} pins $r$ in the bracket
$g_{\mathbb C\setminus\Sigma_\lambda}(z^\star)\le r\le\gamma_\lambda(z^\star)$, of width $\log_+(\lambda/2)$,
the two ends coinciding for $\lambda\le2$ because the spectrum has logarithmic capacity $\max(1,\lambda/2)$.
The identity $r=\gamma_\lambda(z^\star)$ is thus rigorous and unconditional for $0<\lambda<2$, at every phase and every
Diophantine frequency (Theorem~\ref{thm:metclosed}, through the vanishing of the Lyapunov exponent on the subcritical
spectrum and a Bernstein--Walsh estimate), and in closed form for the free chain (Proposition~\ref{prop:freerate}); for $\lambda>2$
it is the upper end of the bracket, established below for almost every phase under an explicit localization
hypothesis (Theorem~\ref{thm:uncond}). At the self-dual coupling $\lambda=2$
the identity holds unconditionally as well (Theorem~\ref{thm:critical}): there the Green's function coincides with the
Lyapunov exponent, so the band edge is automatically regular and the metallic mechanism closes the critical case.

Three results, and then a cocycle form of the cost, close the identity. A
one-line extremal argument yields the mode-dominant lower bound
\[
r\ge\max\{g_{\mathbb C\setminus\Sigma_\lambda}(z^\star),\log_+(\lambda/2)\}
\]
(Proposition~\ref{prop:modelower}); the relative gap to $\gamma_\lambda(z^\star)$ vanishes at both ends of
the localized phase, as $\lambda\downarrow2$ and as $\lambda\to\infty$ (Proposition~\ref{prop:edge}); and an
inverse-free cocycle form of the cost reduces the identity to the single statement
that the inverse Gram matrix amplifies the boundary-data direction at the full exterior-Green rate
(Proposition~\ref{prop:reduction}).

This form has a definite sign structure: the coefficients of the cost are of one sign
(Proposition~\ref{prop:possum}), so the cost is a sum of nonnegative terms with no cancellation, and the rate is
the single-mode maximum $r=\liminf_N\max_k\frac1N\log\abs{c_k}$. Corollary~\ref{cor:singlemode} identifies the identity with an extremal statement: a band-edge mode
that is asymptotically right-boundary localized. A Christoffel--Darboux identity (Lemma~\ref{lem:cd}) collapses the
cost coefficients to $\abs{c_k}=Q(\delta_k)\,(\hat\psi^{(k)}_N)^2$, a regular boundary-weighted form; with the three-distance lemma this establishes the extremal statement,
hence the identity at every Diophantine frequency and almost every phase (Theorem~\ref{thm:uncond}), under the
polynomial-prefactor localization hypothesis \textup{(PL)}
(Lemma~\ref{lem:semiloc}, Remark~\ref{rem:plstatus}), which strengthens the semi-uniform localization of
\cite{Jito1999,Avila2015}, is confirmed numerically down to the transition, and whose window-gap component is
proved at strong coupling (Lemma~\ref{lem:edgeexp}); Remark~\ref{rem:locinput} records why
localization is the irreducible analytic input.

A certified finite-scale enclosure (\S\ref{sec:certified}) places
the rate at the upper end to high precision. What remains is the phase-uniform statement, over all phases (the
frequency now being any Diophantine, with a type-dependent rate); the route reducing that to a single overlap bound is set out in
\S\ref{sec:outlook}.

\section{Setup}\label{sec:setup}

On the finite path $\{1,\dots,N\}$ with Dirichlet boundary conditions, the Aubry--Andr\'e (almost
Mathieu) operator $H_\lambda$ acts by
\begin{equation}\label{eq:HAA}
(H_\lambda x)_j=-(x_{j+1}+x_{j-1})+\lambda\cos(2\pi\beta j+\phi)\,x_j,\qquad x_0=x_{N+1}=0,
\end{equation}
with $\beta$ a Diophantine frequency, fixed to the golden mean $(\sqrt5-1)/2$ for the explicit computations; the phase $\phi=0$ is fixed for the explicit computations and the certified enclosure
of \S\ref{sec:certified}, while Theorem~\ref{thm:uncond} holds for almost every $\phi$. Its Dirichlet eigenpairs are $(E_k,\psi^{(k)})_{k=1}^N$
with $E_1<\dots<E_N$; we write $a_k=\psi^{(k)}_1$, $b_k=\psi^{(k)}_N$ and $\tilde b_k=b_k/a_k$ for the
boundary-amplitude ratios. The stabilization margin is $\alpha>0$, the shifted reference level is
$E_0=\Emin-\alpha$ with $\Emin=\inf\spec(H_\lambda)$ (infinite volume), and $\delta_k=E_k-E_0>0$. We set
$\Sigma_\lambda=\spec(H_\lambda)$ and, using the spectral symmetry $\Emax=-\Emin$ recalled below, write
\begin{equation}\label{eq:zstar}
z^\star=2\Emin-2\alpha-\Emax=2E_0-\Emax=3\Emin-2\alpha
\end{equation}
for the reflected band edge. We write $g_{\mathbb C\setminus K}$ for the Green's function of
$\mathbb C\setminus K$ with pole at infinity.

\emph{The cost.} The quantity of interest (Figure~\ref{fig:setup}) is the
least energy of a boundary input that steers the far unit while the margin-$\alpha$ shifted dynamics
decay. With $A_\alpha:=E_0I-H_\lambda$ (symmetric, Hurwitz) and boundary actuator $B=e_1$,
\begin{equation}\label{eq:gram}
W_\alpha:=\int_0^\infty e^{A_\alpha s}BB^\top e^{A_\alpha s}\,ds,\qquad \Ecal_N:=e_N^\top W_\alpha^{-1}e_N,
\end{equation}
and the object of study is the exponential rate
\begin{equation}\label{eq:rate}
r(\alpha,\lambda):=\liminf_{N\to\infty}\frac1{2N}\log\Ecal_N
\end{equation}
(the $\liminf$ is used because existence of the limit is not needed).

\begin{figure}[t]
\centering
\begin{tikzpicture}[x=1.15cm,y=1cm,>=latex]
  \draw[gray!55,thick,smooth,samples=140,domain=-0.3:6.2,variable=\t]
        plot ({\t},{0.92+0.27*cos(deg(2*pi*0.618*\t+0.7))});
  \node[gray!78,anchor=south,font=\footnotesize] at (2.0,1.24)
        {on-site potential $\lambda\cos(2\pi\beta j+\phi)$};
  \draw[thick] (0,0)--(1,0)--(2,0)--(3,0)--(3.7,0);
  \draw[thick,dotted] (3.7,0)--(4.6,0);
  \draw[thick] (4.6,0)--(5.3,0)--(6,0);
  \foreach \x/\lab in {0/1,1/2,2/3,3/4,5.3/{N\!-\!1}}
     {\fill (\x,0) circle (2.2pt); \node[below=2.5pt,font=\footnotesize] at (\x,0) {$\lab$};}
  \draw[thick] (6,0) circle (3.8pt); \fill (6,0) circle (2.3pt);
  \node[below=2.5pt,font=\footnotesize] at (6,0) {$N$};
  \draw[->,thick] (-1.35,0)--(-0.12,0);
  \node[anchor=east,font=\footnotesize,align=center] at (-1.38,0)
        {boundary\\ actuator $B=e_1$};
  \draw[->,thick] (6.7,0)--(6.2,0);
  \node[anchor=west,font=\footnotesize,align=center] at (6.62,0)
        {far site:\\ steer to $e_N$};
  \draw[|-|] (0,-0.62)--(6,-0.62);
  \node[below,font=\footnotesize] at (3,-0.62) {distance $N$};
\end{tikzpicture}
\caption{The remote-stabilization setup. The almost Mathieu operator $H_\lambda$ acts on the Dirichlet path
$\{1,\dots,N\}$ with unit nearest-neighbour hopping and quasiperiodic on-site potential
$\lambda\cos(2\pi\beta j+\phi)$. A single boundary actuator $B=e_1$ drives the margin-$\alpha$ stable dynamics
$A_\alpha=E_0I-H_\lambda$, and $\Ecal_N=e_N^\top W_\alpha^{-1}e_N$ \eqref{eq:gram} is the least input energy that
steers the far site $N$. The object of study is the rate $r(\alpha,\lambda)$ \eqref{eq:rate} at which this energy
grows with the distance $N$.}
\label{fig:setup}
\end{figure}
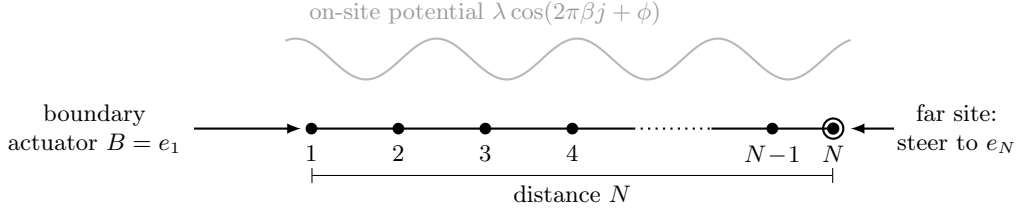

\emph{The cocycle.} The one-step transfer matrices of $H_\lambda$, from the recurrence
$x_{k+1}=(\lambda\cos(2\pi\beta k+\phi)-z)x_k-x_{k-1}$, are
\[
T_k(z)=\begin{pmatrix}\lambda\cos(2\pi\beta k+\phi)-z&-1\\1&0\end{pmatrix},\qquad
\gamma_\lambda(z)=\lim_{n\to\infty}\frac1n\log\norm{T_n(z)\cdots T_1(z)},
\]
the latter the Lyapunov exponent, defined for every $z\in\mathbb C$. We use the following standard input.

\begin{theorem}[Aubry--Andr\'e; Bourgain--Jitomirskaya; Jitomirskaya; Avron--Simon; Craig--Simon]\label{thm:AA}
For \eqref{eq:HAA} with $\beta$ Diophantine (e.g.\ the golden mean $(\sqrt5-1)/2$),
$\gamma_\lambda(E)=\log_+(\lambda/2):=\max\{0,\log(\lambda/2)\}$ for every $E\in\Sigma_\lambda$ \cite{BourgainJito2002,Avila2015}, and
$\gamma_\lambda(z)>\log_+(\lambda/2)$ off the spectrum; $z\mapsto\gamma_\lambda(z)$ is continuous \cite{BourgainJito2002} and
subharmonic on $\mathbb C$ \cite{CraigSimon}. For $\lambda>2$ and almost every $\phi$ the spectrum is pure point with
eigenfunctions localized at rate $\log(\lambda/2)$ \cite{Jito1999}; for $\lambda<2$ it is purely absolutely continuous \cite{Avila2015};
$\lambda=2$ is the self-dual critical coupling \cite{GJLS1997}. The integrated density of states $N_\lambda$ \cite[\S2]{AvronSimon} obeys the
Thouless formula \cite{Thouless1972}
\begin{equation}\label{eq:thouless}
\gamma_\lambda(E)=\int\log\abs{E-E'}\,dN_\lambda(E'),\qquad E\in\R.
\end{equation}
\end{theorem}

In \eqref{eq:thouless} the additive constant $-\log\abs{V}$ of the general Thouless formula vanishes here, the
hopping being unit; the identity is rigorously due to Avron--Simon \cite[Thm~4.4]{AvronSimon} for almost
every $E$, and for the almost Mathieu operator $\gamma_\lambda$ is continuous \cite{BourgainJito2002}, so it holds for
every $E\in\R$, and we invoke it only off $\Sigma_\lambda$.

Two elementary spectral facts complete the setup. The spectrum is phase-independent, and the conjugation
$(Ux)_j=(-1)^jx_j$ maps $H_\phi$ to $-H_{\phi+\pi}$, so $\Sigma_\lambda=-\Sigma_\lambda$ and
\begin{equation}\label{eq:sym}
\Emax=-\Emin>0,\qquad\Sigma_\lambda\subset[-\Emax,\Emax].
\end{equation}
By Weyl's inequality, $H_\lambda$ being a hopping part with spectrum in $[-2,2]$ plus
$\diag(\lambda\cos(\cdot))\in[-\lambda,\lambda]$,
\begin{equation}\label{eq:edge}
\Emin\ge-\lambda-2,\qquad\text{so}\qquad m=\abs{\Emin}\le\lambda+2.
\end{equation}

\section{The cost as a Cauchy quadratic form}\label{sec:cauchy}

The cost \eqref{eq:gram} reduces, in the eigenbasis of $H_\lambda$, to an exponentially ill-conditioned
Cauchy form.

\begin{lemma}[Exact eigenbasis reduction]\label{lem:reduction}
With $C_{kl}=(\delta_k+\delta_l)^{-1}$ the symmetric positive definite Cauchy matrix of the shifted
eigenvalues, $a_k\ne0$ for all $k$ and
\begin{equation}\label{eq:reduction}
\Ecal_N=\tilde b^\top C^{-1}\tilde b,\qquad\tilde b_k=\frac{b_k}{a_k}=\frac{\psi^{(k)}_N}{\psi^{(k)}_1},
\qquad r(\alpha,\lambda)=\liminf_{N\to\infty}\frac1{2N}\log\Ecal_N.
\end{equation}
Moreover $\tilde b_k=\mathcal T_{21}(E_k)$, where $\mathcal T(E)=T_N(E)\cdots T_1(E)$ and $\det\mathcal T\equiv1$.
\end{lemma}

\begin{proof}
Since $A_\alpha=E_0I-H_\lambda$, the eigenbasis gives $A_\alpha\psi^{(k)}=-\delta_k\psi^{(k)}$ with
$\delta_k=E_k-E_0\ge\alpha>0$, hence $e^{A_\alpha s}e_1=\sum_k a_k e^{-\delta_k s}\psi^{(k)}$ and
\[
W_\alpha=\int_0^\infty e^{A_\alpha s}e_1e_1^\top e^{A_\alpha s}\,ds
=\sum_{k,l}\frac{a_ka_l}{\delta_k+\delta_l}\,\psi^{(k)}\psi^{(l)\top}.
\]
In eigencoordinates $W_\alpha=D_aCD_a$ with $D_a=\diag(a_k)$ and $C$ the Cauchy matrix, positive definite
(distinct $\delta_k>0$). An eigenfunction with $\psi_0=\psi_1=0$ vanishes identically by the three-term
recurrence, so $a_k\ne0$ and $W_\alpha\succ0$. Writing $e_N=\sum_k b_k\psi^{(k)}$,
\[
\Ecal_N=e_N^\top W_\alpha^{-1}e_N=b^\top D_a^{-1}C^{-1}D_a^{-1}b=\tilde b^\top C^{-1}\tilde b.
\]
For the cocycle form, a Dirichlet eigenfunction satisfies
$\binom{\psi_{N+1}}{\psi_N}=\mathcal T(E_k)\binom{\psi_1}{0}$, so $\psi_{N+1}=a_k\mathcal T_{11}(E_k)=0$ and
$\psi_N=a_k\mathcal T_{21}(E_k)$, i.e.\ $\tilde b_k=\mathcal T_{21}(E_k)$.
\end{proof}

\section{The cost rate is a reflected off-spectral exponent}\label{sec:identify}

We now identify the rate. A Zolotarev/condenser estimate for the Cauchy form has its extremal energy, by
reflection through $E_0$, land at $z^\star$.

\begin{lemma}[Monotonicity of the exterior Green's function]\label{lem:greenmono}
Let $K\subset\R$ be compact of positive logarithmic capacity. For real $w<\min K$ the map
$w\mapsto g_{\mathbb C\setminus K}(w)$ is strictly decreasing. Consequently, for $J\subset(0,\infty)$ the
function $s\mapsto g_{\mathbb C\setminus J}(-s)$ attains its maximum over $s\in J$ at $s=\max J$.
\end{lemma}

\begin{proof}
With $\mu_{\mathrm{eq}}$ the equilibrium measure of $K$, the Frostman representation
\cite{Ransford}
\[
g_{\mathbb C\setminus K}(w)=\int_K\log\abs{w-t}\,d\mu_{\mathrm{eq}}(t)-\log\cpc(K)
\]
holds on the unbounded complement. For $w<\min K$ every $t\in K$ has $t-w>0$, so
$\partial_w g_{\mathbb C\setminus K}(w)=-\int_K(t-w)^{-1}\,d\mu_{\mathrm{eq}}(t)<0$; no regularity of $K$
is used, so this holds for Cantor spectra. As $s$ ranges over $J$, $-s$ ranges to the left of $J$, where
$g_{\mathbb C\setminus J}$ decreases, so $g_{\mathbb C\setminus J}(-s)$ is largest at $s=\max J$.
\end{proof}

\begin{theorem}[The cost rate is a reflected off-spectral exponent]\label{thm:exact}
The density of states of $H_\lambda$ is the equilibrium measure of $\Sigma_\lambda$, and
\begin{equation}\label{eq:capval}
\cpc(\Sigma_\lambda)=\max(1,\lambda/2),\qquad
g_{\mathbb C\setminus\Sigma_\lambda}(w)=\gamma_\lambda(w)-\log_+(\lambda/2)\quad\text{for }w\notin\Sigma_\lambda
\end{equation}
(in particular at $w=z^\star$).
The cost rate is pinned in the bracket
\begin{equation}\label{eq:bracket}
g_{\mathbb C\setminus\Sigma_\lambda}(z^\star)\ \le\ r(\alpha,\lambda)\ \le\ \gamma_\lambda(z^\star),
\qquad\gamma_\lambda(z^\star)=g_{\mathbb C\setminus\Sigma_\lambda}(z^\star)+\log_+(\lambda/2),
\end{equation}
of width $\log_+(\lambda/2)$ --- the upper end at every phase and coupling, the lower end at every phase for
$0<\lambda\le2$ and, for $\lambda>2$, for almost every phase under \textup{(PL)} (Theorem~\ref{thm:uncond}; see the
summary table after Corollary~\ref{cor:summary}). The two ends coincide for $\lambda\le2$, giving
$r(\alpha,\lambda)=\gamma_\lambda(z^\star)$ there --- unconditionally, at every phase, for $0<\lambda<2$
(Theorem~\ref{thm:metclosed}) and at $\lambda=2$ (Theorem~\ref{thm:critical}),
and for the free chain (Proposition~\ref{prop:freerate}), with the conditioning route of
Proposition~\ref{prop:cond} an alternative under its edge-regularity input. For $\lambda>2$ the boundary amplitudes
grow at least at the localization rate, in the almost-sure $\liminf$ sense of Proposition~\ref{prop:amprate}.
\end{theorem}

\begin{proof}
\emph{Logical structure.} Of the two ends of \eqref{eq:bracket}, the upper bound $r\le\gamma_\lambda(z^\star)$ is
proved \emph{unconditionally} in part~(iii) below, from $\abs{c_k}\le Q(\delta_k)$ and Lemma~\ref{lem:Qedge}
(the same estimate that closes Theorem~\ref{thm:uncond}); the lower bound
$r\ge g_{\mathbb C\setminus\Sigma_\lambda}(z^\star)$ holds at every phase from Theorem~\ref{thm:metclosed} for
$0<\lambda<2$ (and at $\lambda=2$ from Theorem~\ref{thm:critical}), and for $\lambda>2$ from
Theorem~\ref{thm:uncond}, for almost every phase under \textup{(PL)}. Parts~(i)--(ii) only \emph{identify} the bracket
value: (ii) is potential theory, and the condenser/Zolotarev computation of~(i) together with
Proposition~\ref{prop:cond} is an alternative \emph{every-phase} route to the lower end, conditional on the
edge-regularity of the equilibrium measure and \emph{used for no result stated in this paper}. A reader interested
only in the identity may read~(i) as motivation and~(ii)--(iii) as proof.

\emph{(i) Zolotarev rate (conditional; not used below).} By Lemma~\ref{lem:reduction}, $\Ecal_N=\tilde b^\top C^{-1}\tilde b$ for the
Cauchy matrix with nodes $\delta_k$ filling $J:=\Sigma_\lambda-E_0$. The Beckermann--Townsend
bound \cite{BeckermannTownsend} controls the singular values of $C$ by the Zolotarev number of the spectra,
$\sigma_{k+1}(C)\le Z_k(\{\delta_j\},\{-\delta_j\})\,\sigma_1(C)$, which bounds $\sigma_{\min}(C)$ from above and
so gives the lower estimate
\[
\liminf_N\tfrac1{2N}\log\sigma_{\min}(C)^{-1}\ \ge\ \max_{s\in J}U(s),\qquad\nu=\mu_{\mathrm{eq}}^J,
\]
where $U(s)=\int_J\log(s+t)\,d\nu(t)-\int_J\log\abs{s-t}\,d\nu(t)$. This is the rate of the \emph{discrete}
Zolotarev number $Z_k(\{\delta_j\},\{-\delta_j\})$, governed by the pointwise
potential once $\nu$ vanishes like a square root at $\partial J$ \cite{BeckermannGryson,Goncar}; the continuous
bound $Z_k([-b,-a],[a,b])\le4\rho^{-2k}$ of \cite{BeckermannTownsend} alone gives only the strictly smaller
condenser rate $\tfrac1{2\cpc(J,-J)}$, the discrete number being the sharper object that \cite{BeckermannTownsend}
explicitly leaves out of scope. The matching upper estimate, hence the limit
$\frac1{2N}\log\sigma_{\min}(C)^{-1}\to\max_{s\in J}U(s)$, is Proposition~\ref{prop:cond}
(modulo edge-regularity), the continuous Zolotarev bound being one-sided on the discrete singular values. The
equilibrium-potential term equals $\log\cpc(J)$ on $J$ and cancels, leaving
$U(s)=g_{\mathbb C\setminus J}(-s)$; by Lemma~\ref{lem:greenmono} the maximum is at $s=\max J=\Emax-E_0$,
so $\max_{s\in J}U(s)=g_{\mathbb C\setminus\Sigma_\lambda}(2E_0-\Emax)=g_{\mathbb C\setminus\Sigma_\lambda}(z^\star)$
by translation covariance. The upper bracket would follow: $\Ecal_N\le\norm{\tilde b}^2\sigma_{\min}(C)^{-1}$ carries the rate
$g_{\mathbb C\setminus\Sigma_\lambda}(z^\star)+\log_+(\lambda/2)=\gamma_\lambda(z^\star)$ only along subsequences on
which $\frac1{2N}\log\norm{\tilde b}^2\to\log_+(\lambda/2)$; by Remark~\ref{rem:gapstates}, gap-edge states can raise
the amplitude rate along other subsequences --- one more reason this route is not used.

The same upper rate $\gamma_\lambda(z^\star)$ is obtained \emph{unconditionally}, with no conditioning input, from
the boundary-weight bound $\abs{c_k}\le Q(\delta_k)$ and Lemma~\ref{lem:Qedge} (Theorem~\ref{thm:uncond}); the
conditioning route above is thus an alternative for the upper bracket. Proposition~\ref{prop:cond}, with its
edge-regularity input, is needed for no result at $\lambda>2$, where the matching lower bound is
Theorem~\ref{thm:uncond}; on the metallic side $0<\lambda<2$ the lower bound $r\ge g_{\mathbb C\setminus\Sigma_\lambda}(z^\star)$
is now supplied unconditionally by Theorem~\ref{thm:metclosed} (\S\ref{sec:metallic}), so Proposition~\ref{prop:cond}
stands as the every-phase condenser-rate statement of independent interest rather than as a needed input.

The lower bracket $r\ge g_{\mathbb C\setminus\Sigma_\lambda}(z^\star)$ holds by either of two routes,
the first unconditional and the one this paper uses. \emph{(a) For almost every phase}, it is a corollary of
Theorem~\ref{thm:uncond}, which gives $r=\gamma_\lambda(z^\star)\ge g_{\mathbb C\setminus\Sigma_\lambda}(z^\star)$
with no conditioning or edge-regularity input; that theorem is proved independently of this bracket, so the
implication is not circular. \emph{(b) For every phase}, it is the condenser (Zolotarev) rate of
$\sigma_{\min}(C)^{-1}$ computed above (made rigorous in Proposition~\ref{prop:cond} below, modulo the
edge-regularity of the equilibrium measure), together with the dual extremal-polynomial certificate of
\S\ref{sec:outlook}(i) (numerator rate $0$) that aligns $\tilde b$ with the worst-conditioned direction at the
full rate. Route (b) does \emph{not} reduce to the naive
$\Ecal_N\ge\langle\tilde b,u_N\rangle^2\sigma_{\min}(C)^{-1}$. The vectors $\tilde b$ and the smallest-eigenvalue
eigenvector $u_N$ of $C$ share signs (Gantmacher--Krein \cite{GantmacherKrein}), so the overlap carries no
cancellation; but that excludes cancellation only, not the exponential smallness of
$\langle\tilde b,u_N\rangle$ itself. Bounding that overlap is the phase-uniform route of \S\ref{sec:condoverlap}.

\emph{(ii) Density of states and capacity.} By Theorem~\ref{thm:AA}, $\gamma_\lambda\equiv\log_+(\lambda/2)$
on $\Sigma_\lambda$; by the Thouless formula \eqref{eq:thouless} the logarithmic potential of $N_\lambda$ is
the constant $\log_+(\lambda/2)$ q.e.\ on $\Sigma_\lambda$. A probability measure whose logarithmic
potential is constant q.e.\ on its support is the equilibrium measure of that support, the constant being
$\log\cpc$; hence $N_\lambda=\mu_{\mathrm{eq}}$ and $\cpc(\Sigma_\lambda)=\max(1,\lambda/2)$. (The rational
approximants $\beta\approx p/q$ each carry $\cpc=1$, but capacity is discontinuous through this limit, so
they do not transmit it: this is the gap in the naive ``unit-hopping $\Rightarrow\cpc=1$'' reading.) For every
$w\notin\Sigma_\lambda$ the Frostman representation $g_{\mathbb C\setminus\Sigma_\lambda}(w)=\int\log\abs{w-t}\,dN_\lambda(t)-\log\cpc$
combined with the Thouless formula \eqref{eq:thouless} gives
$g_{\mathbb C\setminus\Sigma_\lambda}(w)=\gamma_\lambda(w)-\log_+(\lambda/2)$, i.e.\
\eqref{eq:capval}; for $\lambda\le2$ the subtracted term is $0$, and at $\lambda=2$ both sides extend continuously to
$\Sigma_2$, where they vanish.

\emph{(iii) Supercritical amplitudes.} For $\lambda>2$ and almost every phase, every eigenfunction is localized at
rate $\log(\lambda/2)$ (Theorem~\ref{thm:AA}). We stress that the individual ratios $\tilde b_k=\mathcal T_{21}(E_k)=\psi^{(k)}_N/\psi^{(k)}_1$
are \emph{not} uniformly of size $e^{\log(\lambda/2)N}$: a mode localized at the left end
($\psi^{(k)}_1=O(1)$, $\psi^{(k)}_N=O(e^{-\log(\lambda/2)N})$) has $\abs{\tilde b_k}$ exponentially
\emph{small}, a centre mode has $\abs{\tilde b_k}=e^{o(N)}$, and only a right-end mode
($\psi^{(k)}_1=O(e^{-\log(\lambda/2)N})$, $\psi^{(k)}_N=O(1)$) has $\abs{\tilde b_k}=e^{\log(\lambda/2)N(1+o(1))}$;
this is quantified mode-by-mode in Corollary~\ref{cor:singlemode}. Accordingly \eqref{eq:amprate} is a
statement about the \emph{maximum}, equivalently the rate of the \emph{norm}, not about each coefficient.
The almost-sure lower bound at this rate, \eqref{eq:amprate}, is Proposition~\ref{prop:amprate}; a matching upper
bound fails in general, because finite-volume Dirichlet eigenvalues need not lie in $\Sigma_\lambda$ --- they may
fall in the spectral gaps, where a right-boundary state carries the strictly larger exponent
$\gamma_\lambda(E_k)>\log(\lambda/2)$ (Remark~\ref{rem:gapstates}). With the singular-value rate of~(i)---equivalently, and unconditionally, the
boundary-weight bound $\abs{c_k}\le Q(\delta_k)$ of Lemmas~\ref{lem:cd}--\ref{lem:Qedge} that closes
Theorem~\ref{thm:uncond}---this gives
$r\le g_{\mathbb C\setminus\Sigma_\lambda}(z^\star)+\log(\lambda/2)=\gamma_\lambda(z^\star)$, completing
\eqref{eq:bracket}. The high-precision evaluation of \S\ref{sec:certified} places $r$ at the upper end for
$\lambda>0$, but a proof for the quasiperiodic spectrum is the problem this paper resolves for almost every phase (Theorem~\ref{thm:uncond}).
\end{proof}

\begin{proposition}[Exact free-chain cost rate]\label{prop:freerate}
For $\lambda=0$ and $\alpha>0$, with $\delta_k=\alpha+4\sin^2\frac{k\pi}{2(N+1)}$ and
$\tilde b_k=(-1)^{k+1}$,
\[
r(\alpha,0)=\lim_{N\to\infty}\frac1{2N}\log\Ecal_N=\arccosh(3+\alpha)=\gamma_0(z^\star),
\]
strictly above the margin exponent $\gamma_0(\Emin-\alpha)=\arccosh(1+\tfrac\alpha2)$ (e.g.\ $1.7975$
versus $0.3149$ at $\alpha=0.1$): the resolvent rate is not the cost rate even without disorder.
\end{proposition}

\begin{proof}
The nodes $\delta_k=\alpha+4\sin^2\theta_k$, $\theta_k=\frac{k\pi}{2(N+1)}$, equidistribute to the arcsine
equilibrium measure $\nu$ of $J=[\alpha,\alpha+4]$, of capacity $\cpc(J)=1$. As a totally positive Cauchy
matrix $C$ has (Gantmacher--Krein \cite{GantmacherKrein}) a fully alternating smallest-eigenvalue eigenvector $u_N$; since
$\tilde b_k=(-1)^{k+1}$ shares that pattern, $\langle\tilde b,u_N\rangle^2=\norm{u_N}_1^2\ge\norm{u_N}_2^2=1$,
and with $C^{-1}\preceq\sigma_{\min}(C)^{-1}I$,
\[
\sigma_{\min}(C)^{-1}\ \le\ \Ecal_N\ \le\ N\,\sigma_{\min}(C)^{-1},
\]
both carrying the rate $\frac1{2N}\log\sigma_{\min}(C)^{-1}\to\max_{s\in J}U(s)$. As $\cpc(J)=1$ the
subtracted potential vanishes and $U(s)=g_{\mathbb C\setminus J}(-s)=\arccosh\frac{s+\alpha+2}{2}$,
increasing on $J$, maximal at $s=\alpha+4$, giving $r(\alpha,0)=\arccosh(\alpha+3)$. As
$\cpc(\Sigma_0)=1$, this equals $\gamma_0(z^\star)$ by \eqref{eq:capval}.
\end{proof}

\section{The mode-dominant lower bound}\label{sec:lower}

A single coefficient of the Cauchy form already carries the localization rate: a diagonal
Cauchy--Schwarz bound on $\Ecal_N=\tilde b^\top C^{-1}\tilde b$ isolates one $\tilde b_k$, whose
boundary-amplitude growth gives the lower end of the bracket.

\begin{proposition}[Boundary-amplitude rate]\label{prop:amprate}
Let $\lambda>2$, let $\beta$ be Diophantine, and let $\phi$ lie in the full-measure set on which the semi-uniform
localization bound \eqref{eq:sule} of Lemma~\ref{lem:semiloc} holds. Then
\begin{equation}\label{eq:amprate}
\liminf_{N\to\infty}\frac1{2N}\log\norm{\tilde b}^2\ \ge\ \log(\lambda/2).
\end{equation}
No matching upper bound is asserted; the corresponding limit fails in general (Remark~\ref{rem:gapstates}).
\end{proposition}

\begin{proof}
Fix $\varepsilon>0$ and let $C_\varepsilon$ be the constant of \eqref{eq:sule}, $\gamma=\log(\lambda/2)$. By
completeness, $\sum_k(\hat\psi^{(k)}_N)^2=1$, so some mode $k=k(N)$ has $\abs{\hat\psi^{(k)}_N}\ge N^{-1/2}$; for
that mode \eqref{eq:sule} at $n=N$ gives $N^{-1/2}\le C_\varepsilon e^{\varepsilon N}e^{-\gamma(N-j_k)}$, i.e.\
$\gamma(N-j_k)\le\varepsilon N+\log C_\varepsilon+\tfrac12\log N$, while \eqref{eq:sule} at $n=1$ gives
$\abs{\hat\psi^{(k)}_1}\le C_\varepsilon e^{\varepsilon N}e^{-\gamma(j_k-1)}$. Hence
\[
\abs{\tilde b_{k(N)}}=\frac{\abs{\hat\psi^{(k)}_N}}{\abs{\hat\psi^{(k)}_1}}
\ \ge\ \frac{N^{-1/2}}{C_\varepsilon e^{\varepsilon N}}\,e^{\gamma(j_k-1)}
\ \ge\ C_\varepsilon^{-2}\,N^{-1}\,e^{-2\varepsilon N}\,e^{\gamma(N-1)},
\]
so $\liminf_N\frac1{2N}\log\norm{\tilde b}^2\ge\liminf_N\frac1N\log\abs{\tilde b_{k(N)}}\ge\gamma-2\varepsilon$;
let $\varepsilon\downarrow0$.
\end{proof}

\begin{remark}[Gap-edge states: the amplitude limit fails]\label{rem:gapstates}
Finite-volume Dirichlet eigenvalues need not lie in $\Sigma_\lambda$: they may fall in the spectral gaps, where
$\gamma_\lambda(E)>\log(\lambda/2)$. A mode localized at the \emph{right} boundary at a gap energy $E_k$ has
$\abs{\hat\psi^{(k)}_1}\asymp e^{-\gamma_\lambda(E_k)N}\abs{\hat\psi^{(k)}_N}$, hence
$\abs{\tilde b_k}=e^{\gamma_\lambda(E_k)N(1+o(1))}$, and $\frac1{2N}\log\norm{\tilde b}^2$ exceeds
$\log(\lambda/2)$ along the scales carrying such a state. High-precision data show exactly this: at $\lambda=3$,
$\phi=0$, the maximizing rate $\max_k\frac1N\log\abs{\tilde b_k}$ over $N=30,45,\dots,120$ takes the values
$0.466,\,0.572,\,0.496,\,0.409,\,0.479,\,0.517,\,0.736$ against $\log(3/2)=0.4055$, and at each $N$ it matches, to
finite-size accuracy, the Lyapunov exponent recomputed independently at the maximizing energy ($\gamma_3=0.7355$
at $N=120$, $E\approx1.73$, deep in a gap); generic phases behave alike. The excess is confined to the boundary
amplitudes and cancels from the cost: the identity \eqref{eq:cdform} divides out $\mathcal T_{11}'(E_k)$, which
carries the same gap exponent, so no statement in this paper relies on an upper bound for $\norm{\tilde b}$.
\end{remark}

\begin{proposition}[Mode-dominant lower bound]\label{prop:modelower}
For every $\lambda\ge0$ and every $\alpha>0$ --- for $\lambda>2$, for almost every phase ---
\[
r(\alpha,\lambda)\ \ge\ \log_+(\lambda/2);
\]
combined with the lower bracket of \eqref{eq:bracket} where it is available (every phase for $\lambda\le2$; almost
every phase under \textup{(PL)} for $\lambda>2$), this gives
$r(\alpha,\lambda)\ge\max\{g_{\mathbb C\setminus\Sigma_\lambda}(z^\star),\log_+(\lambda/2)\}$.
\end{proposition}

\begin{proof}
The combined bound follows from the first assertion together with the lower bracket
$r\ge g_{\mathbb C\setminus\Sigma_\lambda}(z^\star)$ in \eqref{eq:bracket}, on its domain of validity. For $\lambda\le2$ the first
assertion is vacuous, since $\log_+(\lambda/2)=0$ and $r=g_{\mathbb C\setminus\Sigma_\lambda}(z^\star)\ge0$.

Let $\lambda>2$. Because $C$ is symmetric positive definite, the Cauchy--Schwarz inequality in the
$C$-inner product gives, for the standard basis vector $e_k$,
\[
(e_k^\top\tilde b)^2=\bigl(e_k^\top C^{1/2}\cdot C^{-1/2}\tilde b\bigr)^2
\le (e_k^\top C\,e_k)\,(\tilde b^\top C^{-1}\tilde b),
\qquad\text{i.e.}\qquad
\Ecal_N(\alpha)=\tilde b^\top C^{-1}\tilde b\ \ge\ \frac{\tilde b_k^2}{C_{kk}} .
\]
Since $C_{kk}=(2\delta_k)^{-1}$ and $\delta_k=E_k-E_0\ge \Emin-(\Emin-\alpha)=\alpha$, this reads
$\Ecal_N(\alpha)\ge 2\delta_k\,\tilde b_k^2\ge 2\alpha\,\tilde b_k^2$ for every $k$. Maximizing over $k$ and
using $\max_k\tilde b_k^2\ge \norm{\tilde b}^2/N$,
\[
\Ecal_N(\alpha)\ \ge\ 2\alpha\,\max_k\tilde b_k^2\ \ge\ \frac{2\alpha}{N}\,\norm{\tilde b}^2 .
\]
Taking logarithms,
\[
\frac1{2N}\log\Ecal_N(\alpha)\ \ge\ \frac{1}{2N}\log\frac{2\alpha}{N}+\frac{1}{2N}\log\norm{\tilde b}^2 .
\]
The first term tends to $0$; taking $\liminf$, applying the almost-sure amplitude bound \eqref{eq:amprate}
(Proposition~\ref{prop:amprate}) and invoking \eqref{eq:reduction} gives
$r(\alpha,\lambda)\ge\log(\lambda/2)=\log_+(\lambda/2)$.
\end{proof}

\begin{remark}\label{rem:improves}
The bound improves the lower bracket of \eqref{eq:bracket} precisely when localization dominates
conditioning, $\log_+(\lambda/2)>g_{\mathbb C\setminus\Sigma_\lambda}(z^\star)$, equivalently
$2\log(\lambda/2)>\gamma_\lambda(z^\star)$. By \eqref{eq:bracket} this region is non-empty and, as shown
next, comprises all sufficiently large $\lambda$. The two contributions have distinct origins: the
single test vector $e_{k}$ captures the boundary-amplitude growth $\log_+(\lambda/2)$ but is blind to the
ill-conditioning of $C$, whereas the bracket lower bound captures the conditioning rate
$g_{\mathbb C\setminus\Sigma_\lambda}(z^\star)$; their \emph{sum} $\gamma_\lambda(z^\star)$ is the upper
end, and reaching it would require a test vector aligned simultaneously with $\tilde b$ and with the
worst-conditioned direction of $C$ (see \S\ref{sec:remarks}).
\end{remark}

\section{Edge sharpness of the cost rate}\label{sec:edge}

Write the relative gap
\[
\eta(\alpha,\lambda):=1-\frac{r(\alpha,\lambda)}{\gamma_\lambda(z^\star)}\in[0,1],
\]
so that $\eta=0$ is exactly the upper-end identity $r=\gamma_\lambda(z^\star)$.

\begin{proposition}[Vanishing of the relative gap]\label{prop:edge}
Fix $\alpha>0$. Then:
\begin{enumerate}
\item[(i)] $0\le \gamma_\lambda(z^\star)-r(\alpha,\lambda)\le\log(\lambda/2)$ for all $\lambda>2$;
\item[(ii)] $\displaystyle\lim_{\lambda\downarrow2}\eta(\alpha,\lambda)=0$;
\item[(iii)] $\displaystyle\lim_{\lambda\to\infty}\eta(\alpha,\lambda)=0$.
\end{enumerate}
In particular $r(\alpha,\lambda)=\gamma_\lambda(z^\star)\,(1+o(1))$ at both ends of the localized phase.
\end{proposition}

\begin{proof}
\emph{(i)} is immediate from \eqref{eq:bracket}: $r\ge g_{\mathbb C\setminus\Sigma_\lambda}(z^\star)
=\gamma_\lambda(z^\star)-\log(\lambda/2)$ and $r\le\gamma_\lambda(z^\star)$.

\emph{(ii)} By (i), $\eta(\alpha,\lambda)\le\log(\lambda/2)/\gamma_\lambda(z^\star)$. As $\lambda\downarrow2$
the numerator tends to $0$. For the denominator, $z^\star$ stays at spectral depth
$\Emax-\Emin+2\alpha\ge2\alpha$ below the band, so it remains in the resolvent set, where the cocycle is
uniformly hyperbolic; by the monotonicity of the off-spectral exponent (Lemma~\ref{lem:greenmono}) and the
continuity of the Lyapunov exponent (Theorem~\ref{thm:AA}), the map
$\lambda\mapsto\gamma_\lambda(z^\star)$ is continuous and strictly positive at $\lambda=2$,
$\gamma_2(z^\star)>0$. Hence $\eta(\alpha,\lambda)\le\log(\lambda/2)/\gamma_\lambda(z^\star)\to 0/\gamma_2(z^\star)=0$.

\emph{(iii)} By Proposition~\ref{prop:modelower}, $r\ge\log(\lambda/2)$, so with \eqref{eq:bracket},
\begin{align*}
\eta(\alpha,\lambda)=1-\frac{r}{\gamma_\lambda(z^\star)}
&\ \le\ 1-\frac{\log(\lambda/2)}{\gamma_\lambda(z^\star)}
\ =\ \frac{\gamma_\lambda(z^\star)-\log(\lambda/2)}{\gamma_\lambda(z^\star)}\\
&\ =\ \frac{g_{\mathbb C\setminus\Sigma_\lambda}(z^\star)}
{g_{\mathbb C\setminus\Sigma_\lambda}(z^\star)+\log(\lambda/2)} .
\end{align*}
It therefore suffices to show $g_{\mathbb C\setminus\Sigma_\lambda}(z^\star)$ stays bounded as
$\lambda\to\infty$, since $\log(\lambda/2)\to\infty$. In fact it converges. Let $\mu_{\mathrm{eq}}$ be the
equilibrium measure of $\Sigma_\lambda$, which is the density of states by Theorem~\ref{thm:exact}; by the Frostman
representation, with $\cpc(\Sigma_\lambda)=\lambda/2$ from \eqref{eq:capval},
\[
g_{\mathbb C\setminus\Sigma_\lambda}(z^\star)
=\int_{\Sigma_\lambda}\log\abs{z^\star-t}\,d\mu_{\mathrm{eq}}(t)-\log(\lambda/2).
\]
Since $\abs{\Emin}=\Emax$, $z^\star=-3\Emax-2\alpha$ and $\abs{z^\star-t}=3\Emax+2\alpha+t$ for $t\in[-\Emax,\Emax]$ by
\eqref{eq:sym}. Rescale $t=\lambda s$: the band edges obey $\Emax/\lambda\to1$, so $z^\star/\lambda\to-3$, and by Weyl
equidistribution of $\{\beta j\}$ the rescaled density of states converges weakly to the arcsine measure
$\mu_{[-1,1]}$, the equilibrium measure of $[-1,1]$ (of capacity $\tfrac12$). Hence
\begin{equation}\label{eq:glimit}
g_{\mathbb C\setminus\Sigma_\lambda}(z^\star)\ \xrightarrow[\lambda\to\infty]{}\
\int_{-1}^{1}\log\abs{3+s}\,d\mu_{[-1,1]}(s)+\log2\ =\ \arccosh 3\ =\ \log\bigl(3+2\sqrt2\bigr)\approx1.7627,
\end{equation}
the Green's function $g_{\mathbb C\setminus[-1,1]}(-3)$. In particular $g_{\mathbb C\setminus\Sigma_\lambda}(z^\star)$
is bounded, with the crude uniform bound $\limsup\le\log8$ from $\abs{z^\star-t}\le4\Emax+2\alpha$ and $\Emax\le\lambda+2$
\eqref{eq:edge}, and $\eta(\alpha,\lambda)\le g_{\mathbb C\setminus\Sigma_\lambda}(z^\star)/
\bigl(g_{\mathbb C\setminus\Sigma_\lambda}(z^\star)+\log(\lambda/2)\bigr)\to0$.
\end{proof}

\begin{corollary}[Bounded absolute gap]\label{cor:summary}
Fix $\alpha>0$ and $\lambda>2$. For almost every phase,
\[
0\ \le\ \gamma_\lambda(z^\star)-r(\alpha,\lambda)\ \le\ g_{\mathbb C\setminus\Sigma_\lambda}(z^\star),
\]
and under \textup{(PL)} the gap is $0$ throughout $(2,\infty)$ (Theorem~\ref{thm:uncond}); in particular, for
almost every phase and under \textup{(PL)},
\begin{equation}\label{eq:tightgap}
0\ \le\ \gamma_\lambda(z^\star)-r(\alpha,\lambda)\ \le\ \min\bigl\{\,g_{\mathbb C\setminus\Sigma_\lambda}(z^\star),\ \log_+(\lambda/2)\,\bigr\}.
\end{equation}
Consequently, for almost every phase, the absolute gap is uniformly bounded on $(2,\infty)$ with
$\limsup_{\lambda\to\infty}\bigl[\gamma_\lambda(z^\star)-r(\alpha,\lambda)\bigr]\le\arccosh3=\log(3+2\sqrt2)$ by
\eqref{eq:glimit}, and under \textup{(PL)} it vanishes as $\lambda\downarrow2$, matching the unconditional
identity on $0\le\lambda\le2$ (no input at $\lambda=0$; Theorem~\ref{thm:metclosed} for $0<\lambda<2$;
Theorem~\ref{thm:critical} at $\lambda=2$). The envelope
$\min\{g_{\mathbb C\setminus\Sigma_\lambda}(z^\star),\log_+(\lambda/2)\}$ never exceeds $1.78$ numerically at
$\alpha=0.1$, with maximum near the crossover $\lambda\approx12$ (Figure~\ref{fig:gap}).
\end{corollary}

\begin{proof}
By Proposition~\ref{prop:modelower} (almost every phase), $r\ge\log_+(\lambda/2)$, so
$\gamma_\lambda(z^\star)-r\le\gamma_\lambda(z^\star)-\log_+(\lambda/2)=g_{\mathbb C\setminus\Sigma_\lambda}(z^\star)$
by the split \eqref{eq:capval}. Under \textup{(PL)}, Theorem~\ref{thm:uncond} gives $r=\gamma_\lambda(z^\star)$,
whence \eqref{eq:tightgap} and the vanishing as $\lambda\downarrow2$. The uniform bound and the limit
$\gamma_\lambda(z^\star)-r\le g_{\mathbb C\setminus\Sigma_\lambda}(z^\star)\to\arccosh3$ follow from
\eqref{eq:glimit} in the
proof of Proposition~\ref{prop:edge}(iii).
\end{proof}

\subsection*{Scope of the results}
The table records, for each regime, the strongest statement proved and its input; \textup{(PL)} is the
polynomial-localization hypothesis of Lemma~\ref{lem:semiloc} (status in Remark~\ref{rem:plstatus}; its window-gap
component is a theorem for $\lambda\ge\lambda_1$, Lemma~\ref{lem:edgeexp}).
\begin{center}
\begin{tabular}{@{}llll@{}}
\hline
regime & statement & phase & input \\
\hline
$\lambda=0$ & $r=\gamma_0(z^\star)=\arccosh(3+\alpha)$ & --- & none (Prop.~\ref{prop:freerate}) \\
$0<\lambda<2$ & $r=\gamma_\lambda(z^\star)$ & every & none (Thm.~\ref{thm:metclosed}) \\
$\lambda=2$ & $r=\gamma_\lambda(z^\star)$ & every & none (Thm.~\ref{thm:critical}) \\
$\lambda>2$ & $r\le\gamma_\lambda(z^\star)$ & every & none (Lem.~\ref{lem:cd}--\ref{lem:Qedge}) \\
$\lambda>2$ & $r\ge\log_+(\lambda/2)$ & a.e. & none (Prop.~\ref{prop:modelower}) \\
$\lambda>2$ & $r=\gamma_\lambda(z^\star)$, gap $O(N^{-2/(2+\tau)})$ & a.e. & \textup{(PL)} (Thm.~\ref{thm:uncond}) \\
\hline
\end{tabular}
\end{center}

\begin{figure}[t]
\centering
\begin{tikzpicture}
\begin{axis}[
  width=0.84\textwidth, height=0.5\textwidth,
  xlabel={$\lambda$}, ylabel={rate / gap}, xmin=2, xmax=28, ymin=0, ymax=2.15,
  legend pos=south east, legend cell align={left}, tick align=outside,
  every axis plot/.append style={line width=0.9pt}]
\addplot[dashed] coordinates {(2.05,2.0334) (2.1,2.0220) (2.25,1.9920) (2.5,1.9530) (2.75,1.9235) (3,1.9007) (3.5,1.8680) (4,1.8462) (4.5,1.8308) (5,1.8195) (6,1.8043) (7,1.7948) (8,1.7885) (9,1.7840) (10,1.7807) (11,1.7782) (12,1.7762) (13,1.7746) (14,1.7734) (15,1.7723) (17,1.7707) (20,1.7690) (24,1.7676) (28,1.7667)};
\addlegendentry{$g_{\mathbb C\setminus\Sigma_\lambda}(z^\star)$ (conditioning)}
\addplot[dotted] coordinates {(2.05,0.0247) (2.1,0.0488) (2.25,0.1178) (2.5,0.2231) (2.75,0.3185) (3,0.4055) (3.5,0.5596) (4,0.6931) (4.5,0.8109) (5,0.9163) (6,1.0986) (7,1.2528) (8,1.3863) (9,1.5041) (10,1.6094) (11,1.7047) (12,1.7918) (13,1.8718) (14,1.9459) (15,2.0149) (17,2.1401) (20,2.3026) (24,2.4849) (28,2.6391)};
\addlegendentry{$\log_+(\lambda/2)$ (localization)}
\addplot[solid, line width=1.5pt] coordinates {(2.05,0.0247) (2.1,0.0488) (2.25,0.1178) (2.5,0.2231) (2.75,0.3185) (3,0.4055) (3.5,0.5596) (4,0.6931) (4.5,0.8109) (5,0.9163) (6,1.0986) (7,1.2528) (8,1.3863) (9,1.5041) (10,1.6094) (11,1.7047) (12,1.7762) (13,1.7746) (14,1.7734) (15,1.7723) (17,1.7707) (20,1.7690) (24,1.7676) (28,1.7667)};
\addlegendentry{gap bound $\min\{g_{\mathbb C\setminus\Sigma_\lambda}(z^\star),\log_+(\lambda/2)\}$}
\end{axis}
\end{tikzpicture}
\caption{The combined absolute-gap bound \eqref{eq:tightgap} (solid) is the lower envelope of the
conditioning rate $g_{\mathbb C\setminus\Sigma_\lambda}(z^\star)$ (dashed) and the localization rate
$\log_+(\lambda/2)$ (dotted), here at $\alpha=0.1$. It vanishes as $\lambda\downarrow2$, is uniformly bounded
on $(2,\infty)$ with maximum $\approx1.78$ near the crossover $\lambda\approx12$, and decreases to
$g_{\mathbb C\setminus\Sigma_\lambda}(z^\star)\to\arccosh3=\log(3+2\sqrt2)\approx1.763$ as $\lambda\to\infty$. The cost rate
$r(\alpha,\lambda)$ lies within this gap below the upper end $\gamma_\lambda(z^\star)$; values of
$g_{\mathbb C\setminus\Sigma_\lambda}(z^\star)=\gamma_\lambda(z^\star)-\log_+(\lambda/2)$ are computed from the
transfer cocycle.}
\label{fig:gap}
\end{figure}
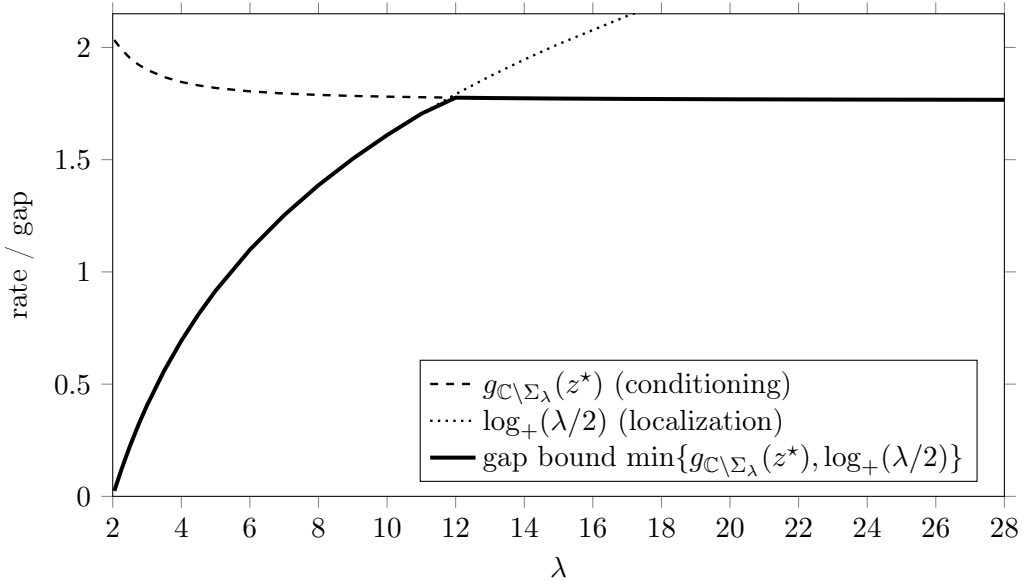

\section{A reduction of the upper-end identity}\label{sec:reduction}

The results above bracket $r$ within $\min\{g_{\mathbb C\setminus\Sigma_\lambda}(z^\star),\log_+(\lambda/2)\}$ of
$\gamma_\lambda(z^\star)$, but the bracket alone leaves the interior of $(2,\infty)$ unpinned. The next proposition
reduces the identity to a single one-sided inequality.

\begin{proposition}[Reduction]\label{prop:reduction}
Let $\lambda>2$ and $\alpha>0$, and set
$\mathcal R_N:=\dfrac{\tilde b^\top C^{-1}\tilde b}{\tilde b^\top\tilde b}$, the Rayleigh quotient of $C^{-1}$ at
the boundary-amplitude vector $\tilde b$. Then, for almost every phase,
\begin{equation}\label{eq:Rrate}
\liminf_{N\to\infty}\frac{1}{2N}\log\mathcal R_N\ \le\ r(\alpha,\lambda)-\log_+(\lambda/2)\ \le\ g_{\mathbb C\setminus\Sigma_\lambda}(z^\star),
\end{equation}
and the identity $r(\alpha,\lambda)=\gamma_\lambda(z^\star)$ holds whenever
\begin{equation}\label{eq:reductioncrit}
\liminf_{N\to\infty}\frac{1}{2N}\log\mathcal R_N\ \ge\ g_{\mathbb C\setminus\Sigma_\lambda}(z^\star).
\end{equation}
Equivalently, writing $C=\sum_{j}\sigma_j u_ju_j^\top$ with $\sigma_1\ge\dots\ge\sigma_N>0$ and
$p_j:=(u_j^\top\tilde b)^2/\norm{\tilde b}^2$ (a probability vector), the criterion \eqref{eq:reductioncrit}
reads $\liminf_N\frac1{2N}\log\sum_j p_j\,\sigma_j^{-1}\ge g_{\mathbb C\setminus\Sigma_\lambda}(z^\star)$.
\end{proposition}

\begin{proof}
By \eqref{eq:reduction}, $\frac1{2N}\log\Ecal_N=\frac1{2N}\log\mathcal R_N+\frac1{2N}\log\norm{\tilde b}^2$.
Superadditivity of $\liminf$ and the almost-sure amplitude bound \eqref{eq:amprate} give
\[
\liminf_N\tfrac1{2N}\log\mathcal R_N\ \le\ \liminf_N\tfrac1{2N}\log\Ecal_N-\liminf_N\tfrac1{2N}\log\norm{\tilde b}^2\ \le\ r(\alpha,\lambda)-\log_+(\lambda/2);
\]
no equality is claimed, since $\frac1{2N}\log\norm{\tilde b}^2$ need not converge (Remark~\ref{rem:gapstates}).
The upper bracket \eqref{eq:bracket} gives $r\le\gamma_\lambda(z^\star)=g_{\mathbb C\setminus\Sigma_\lambda}(z^\star)+\log_+(\lambda/2)$,
whence $r-\log_+(\lambda/2)\le g_{\mathbb C\setminus\Sigma_\lambda}(z^\star)$; this is \eqref{eq:Rrate}. Conversely,
if \eqref{eq:reductioncrit} holds then, again by superadditivity and \eqref{eq:amprate},
$r\ge\liminf_N\tfrac1{2N}\log\mathcal R_N+\liminf_N\tfrac1{2N}\log\norm{\tilde b}^2\ge g_{\mathbb C\setminus\Sigma_\lambda}(z^\star)+\log_+(\lambda/2)=\gamma_\lambda(z^\star)$,
and with the unconditional upper bound the identity follows; \eqref{eq:reductioncrit} is thus a \emph{sufficient}
criterion, the exact inversion-free equivalent of the identity being \eqref{eq:cocyclecrit}. Finally
$\mathcal R_N=\tilde b^\top C^{-1}\tilde b/\norm{\tilde b}^2=\sum_j(u_j^\top\tilde b)^2/(\sigma_j\norm{\tilde b}^2)=\sum_j p_j\sigma_j^{-1}$
by the spectral theorem.
\end{proof}

\begin{remark}[What an eventual proof must supply]\label{rem:supply}
The criterion \eqref{eq:reductioncrit} is met as soon as one eigendirection carries the full conditioning
rate and is not exponentially orthogonal to the data: with $\sigma_{\min}=\sigma_N$, $u_{\min}=u_N$,
\[
\frac1{2N}\log\sigma_{\min}^{-1}\ \longrightarrow\ g_{\mathbb C\setminus\Sigma_\lambda}(z^\star)
\qquad\text{and}\qquad
\liminf_{N}\frac1{2N}\log p_{N}\ \ge\ 0 .
\]
The first is an identification of the smallest Gram eigenvalue with the exterior-Green rate at $z^\star$; the
second is a non-degenerate-overlap statement between the boundary data and the worst-conditioned direction.
These are the two ingredients an eventual proof must supply, and \eqref{eq:reductioncrit} shows they are also,
together, sufficient.
\end{remark}

\section{Remarks}\label{sec:remarks}

\begin{remark}[The nature of the obstruction]\label{rem:obstruction}
By Proposition~\ref{prop:reduction}, closing the interior reduces to the sufficient criterion
\eqref{eq:reductioncrit}: the inverse Gram
form must amplify $\tilde b$ at the full rate $g_{\mathbb C\setminus\Sigma_\lambda}(z^\star)$. Two features make
this hard, and one helps. The conditioning rate $g_{\mathbb C\setminus\Sigma_\lambda}(z^\star)$ is a
\emph{collective} (logarithmic-capacity) effect of all $N$ crowded nodes $\{\delta_k\}$, not reproducible by any
fixed number of them; the boundary growth $\log_+(\lambda/2)$ is carried by individual far-localized modes; and
the criterion demands that the worst-conditioned direction, built from the node spacings, overlap the data,
built from the localization centres. What \emph{helps} is that the boundary amplitudes are sign-alternating,
$\tilde b_k=(-1)^{k+1}\abs{\tilde b_k}$ (Sturm oscillation): in $\tilde b^\top C^{-1}\tilde b$ the off-diagonal
Cauchy entries enter the relevant near-null combinations \emph{constructively}, so the difficulty is not
avoiding cancellation but capturing the collective amplification.
\end{remark}

\begin{remark}[Why strong coupling does not trivialize it]\label{rem:strong}
It is natural to seek \eqref{eq:reductioncrit} in the regime $\lambda\gg1$, where Aubry--Andr\'e localization is
sharpest. There the data is indeed explicit: with localization centre $j_k$, the boundary ratio is
$\abs{\tilde b_k}=e^{\gamma(2j_k-N-1)(1+o(1))}$, $\gamma=\log(\lambda/2)$, and the node is
$\delta_k=2\lambda\cos^2(\pi\beta j_k)+\alpha+O(1/\lambda)$. But $\delta_k$ depends on $j_k$ only through
$\cos^2(\pi\beta j_k)$, whereas $\tilde b_k$ depends on $j_k$ itself: two sites with equal $\cos^2(\pi\beta\,\cdot)$
share a node yet have exponentially different amplitudes. Thus, as a function of the node, the data is a signed,
effectively pseudo-random exponential, and the overlap weights $p_j$ of \eqref{eq:reductioncrit} are governed by
the arithmetic of $\{\beta j\bmod1\}$ near the band edge ($\cos(2\pi\beta j)\approx-1$), i.e.\ by the
continued-fraction approximants of $\beta$. Strong coupling sharpens the localization but not this
edge arithmetic, so it does not by itself yield \eqref{eq:reductioncrit}; Proposition~\ref{prop:edge}(iii)
remains the proven strong-coupling statement (relative sharpness).
\end{remark}

\begin{remark}[Numerical consistency (illustrative)]
For $\lambda=3$, $\alpha=0.1$ the bracket \eqref{eq:bracket} is
$[\,g_{\mathbb C\setminus\Sigma_\lambda}(z^\star),\gamma_\lambda(z^\star)\,]=[1.901,\,2.306]$, of width
$\log(3/2)=0.405$. A high-precision evaluation of $\tfrac1{2N}\log\Ecal_N$ (exact tridiagonal eigendata and the
quadratic form $\tilde b^\top C^{-1}\tilde b$ both in $80$-digit arithmetic, which the conditioning
$\operatorname{cond}(C)\sim e^{2Ng}$ requires) gives
$1.835,\,1.90,\,1.98,\,2.10,\,2.16$ for $N=15,20,25,30,35$: the rate approaches the bracket from below, sitting inside
it from $N=20$ on, and increases toward the upper end $\gamma_\lambda(z^\star)$, consistent with the interior sharpness that remains to be
proved, while $\tfrac1{2N}\log\norm{\tilde b}^2$ fluctuates about, and intermittently above, $\log(\lambda/2)=0.405$
--- the excursions above are the gap-edge states of Remark~\ref{rem:gapstates} --- consistent with
Propositions~\ref{prop:amprate} and~\ref{prop:modelower}. The same data give the reduction ratio
$\tfrac1{2N}\log\mathcal R_N=1.38,\,1.48,\,1.52,\,1.63,\,1.63$, matching
$\tfrac1{2N}\log\Ecal_N-\tfrac1{2N}\log\norm{\tilde b}^2$ as in \eqref{eq:Rrate} and increasing toward
$g_{\mathbb C\setminus\Sigma_\lambda}(z^\star)=1.901$, consistent with the criterion
\eqref{eq:reductioncrit}.
\end{remark}

\section{An inverse-free cocycle form of the cost}\label{sec:cocycle}

The criterion \eqref{eq:reductioncrit} is a statement about the inverse Gram matrix $C^{-1}$. We remove the
inversion: the cost is an explicit squared exponential sum whose coefficients are transfer-cocycle and
characteristic-polynomial data, and the criterion becomes a sharp lower bound on the growth of that sum, to which
the large-deviation theory of the cocycle applies directly. Throughout, $P(x)=\prod_{m=1}^N(x-\delta_m)$ and
$Q(x)=\prod_{m=1}^N(x+\delta_m)$, so $P(x)=\chi_N(x+E_0)$ with $\chi_N(E)=\det(E-H_N)$, and $Q(x)=(-1)^NP(-x)$;
write $P'(\delta_k)=\prod_{m\ne k}(\delta_k-\delta_m)$.

\begin{proposition}[Inverse-free form]\label{prop:cocycleform}
Let $D=\operatorname{diag}\bigl(Q(\delta_k)/P'(\delta_k)\bigr)_{k=1}^N$ and $c=D\tilde b$, i.e.\
$c_k=\dfrac{Q(\delta_k)}{P'(\delta_k)}\,\tilde b_k$. Then $C^{-1}=DCD$, and consequently
\begin{equation}\label{eq:Enorm}
\Ecal_N=\tilde b^\top C^{-1}\tilde b=c^\top C c=\Bigl\|\,\textstyle\sum_{k=1}^N c_k\,e^{-\delta_k\,\cdot}\,\Bigr\|_{L^2(0,\infty)}^2 .
\end{equation}
Equivalently, with $\Gamma$ a positively oriented contour enclosing $[\alpha,\max_k\delta_k]$,
\begin{equation}\label{eq:Fcontour}
\sum_{k=1}^N c_k\,e^{-\delta_k s}=\frac{1}{2\pi i}\oint_\Gamma \frac{Q(x)}{P(x)}\,\mathcal T_{21}(x+E_0)\,e^{-xs}\,dx
=\frac{(-1)^N}{2\pi i}\oint_\Gamma Q(x)\,m_N(x+E_0)\,e^{-xs}\,dx,
\end{equation}
where $m_N=\mathcal T_{21}/\mathcal T_{11}$ is the truncated Weyl function and $\mathcal T_{11}(\,\cdot\,)=(-1)^N\chi_N(\,\cdot\,)$.
\end{proposition}

\begin{proof}
Solve $Cw=\tilde b$ and set $R(x)=\sum_{l}w_l/(x+\delta_l)$, rational with $R(\delta_k)=(Cw)_k=\tilde b_k$. Writing
$R=S/Q$ with $\deg S\le N-1$ gives $S(\delta_k)=\tilde b_k\,Q(\delta_k)$, so by Lagrange interpolation
$S(x)=P(x)\sum_{k}\tilde b_kQ(\delta_k)\big/\bigl(P'(\delta_k)(x-\delta_k)\bigr)$. The residue of $R$ at its pole
$-\delta_l$ is $w_l=S(-\delta_l)/Q'(-\delta_l)$; using $P(-\delta_l)=(-1)^NQ(\delta_l)$ and
$Q'(-\delta_l)=(-1)^{N-1}P'(\delta_l)$,
\[
w_l=\frac{Q(\delta_l)}{P'(\delta_l)}\sum_{k}\frac{1}{\delta_l+\delta_k}\,\frac{Q(\delta_k)}{P'(\delta_k)}\tilde b_k
=D_{ll}\,(C\,c)_l=(DCD\,\tilde b)_l .
\]
As $\tilde b$ is arbitrary, $C^{-1}=DCD$; then $\Ecal_N=\tilde b^\top w=c^\top Cc$, and
$c^\top Cc=\sum_{k,l}c_kc_l/(\delta_k+\delta_l)=\|\sum_kc_ke^{-\delta_k\cdot}\|^2_{L^2(0,\infty)}$ since
$\int_0^\infty e^{-(\delta_k+\delta_l)s}\,ds=(\delta_k+\delta_l)^{-1}$, which is \eqref{eq:Enorm}. The integrand of
\eqref{eq:Fcontour} has simple poles only at the $\delta_k$, with residue
$\frac{Q(\delta_k)}{P'(\delta_k)}\mathcal T_{21}(E_k)e^{-\delta_k s}=c_ke^{-\delta_k s}$ (recall
$\tilde b_k=\mathcal T_{21}(E_k)$); the second equality uses $P=(-1)^N\mathcal T_{11}(\,\cdot+E_0)$.
\end{proof}

The coefficients $c_k$ turn out to share a common sign, which removes the apparent cancellation in
\eqref{eq:Enorm} entirely and collapses the cost to a single-mode extremal problem.

\begin{proposition}[Cancellation-free positive-sum reduction]\label{prop:possum}
Order $\delta_1<\dots<\delta_N$. The coefficients $c_k=Q(\delta_k)\tilde b_k/P'(\delta_k)$ of
Proposition~\ref{prop:cocycleform} all share one sign: writing $\operatorname{sign}\tilde b_k=\epsilon_b(-1)^k$
(Sturm oscillation, \S\ref{sec:remarks}),
\begin{equation}\label{eq:csign}
\operatorname{sign}(c_k)=\epsilon_b(-1)^{N}\qquad\text{for every }k .
\end{equation}
Consequently $\Ecal_N$ is a sum of nonnegative terms,
\begin{equation}\label{eq:possum}
\Ecal_N=c^\top Cc=\sum_{k,l=1}^N\frac{\abs{c_k}\,\abs{c_l}}{\delta_k+\delta_l},
\end{equation}
and its exponential rate is carried by the largest coefficient:
\begin{equation}\label{eq:maxc}
r(\alpha,\lambda)=\liminf_{N\to\infty}\ \max_{1\le k\le N}\ \frac1N\log\abs{c_k},
\qquad c_k=\frac{Q(\delta_k)}{P'(\delta_k)}\,\tilde b_k .
\end{equation}
\end{proposition}

\begin{proof}
$Q(\delta_k)=\prod_m(\delta_k+\delta_m)>0$ since every $\delta_m>0$. With $\delta_1<\dots<\delta_N$ the product
$P'(\delta_k)=\prod_{m\ne k}(\delta_k-\delta_m)$ has exactly $N-k$ negative factors, so
$\operatorname{sign}P'(\delta_k)=(-1)^{N-k}$. By Sturm oscillation the $k$-th Dirichlet eigenvector has $k-1$ sign
changes, hence $\tilde b_k=\psi^{(k)}_N/\psi^{(k)}_1$ strictly alternates,
$\operatorname{sign}\tilde b_k=\epsilon_b(-1)^k$ (the alternation recorded in \S\ref{sec:remarks}). Therefore
$\operatorname{sign}(c_k)=(+1)\cdot\epsilon_b(-1)^k\cdot(-1)^{N-k}=\epsilon_b(-1)^{N}$, independent of $k$,
which is \eqref{eq:csign}. The common sign squares to $1$ in $c^\top Cc=\sum_{k,l}c_kc_l/(\delta_k+\delta_l)$,
giving the nonnegative form \eqref{eq:possum}. For the rate, the diagonal of \eqref{eq:possum} gives
$\Ecal_N\ge\max_k c_k^2/(2\delta_k)$, while $\delta_k+\delta_l\ge2\sqrt{\delta_k\delta_l}$ gives
$\Ecal_N\le\tfrac12\bigl(\sum_k\abs{c_k}/\sqrt{\delta_k}\bigr)^2\le(2\delta_1)^{-1}\bigl(\sum_k\abs{c_k}\bigr)^2$.
Since $\delta_k$ is bounded in $N$ (so $\frac1{2N}\log\delta_k\to0$) and
$\frac1{2N}\log\bigl(\sum_k\abs{c_k}\bigr)^2=\max_k\frac1N\log\abs{c_k}+O(\tfrac{\log N}N)$, both bounds carry the
rate $\max_k\frac1N\log\abs{c_k}$; with \eqref{eq:reduction} this is \eqref{eq:maxc}.
\end{proof}

The next lemma rewrites $c_k$ in a form that makes the single-mode maximum \eqref{eq:maxc} tractable.

\begin{lemma}[Boundary-weight form of the coefficient]\label{lem:cd}
Let $\hat\psi^{(k)}=\psi^{(k)}/\norm{\psi^{(k)}}$ be the $\ell^2$-normalized Dirichlet eigenvector. Then
\begin{equation}\label{eq:cdform}
\abs{c_k}\ =\ Q(\delta_k)\,\bigl(\hat\psi^{(k)}_N\bigr)^2 ,
\end{equation}
so $0\le\abs{c_k}\le Q(\delta_k)$. The near-degeneracies of the cocycle factors cancel:
a band-edge tunneling pair that makes $\mathcal T_{21}(E_k)$ exponentially small makes $\mathcal T_{11}'(E_k)$
equally small, and \eqref{eq:cdform} is free of both.
\end{lemma}

\begin{proof}
Normalize the transfer solution by $y_1=1$ (so $\psi^{(k)}_1=1$, $\psi^{(k)}_n=y_n(E_k)$). Differentiating the
recurrence $y_{n+1}=(V_n-E)y_n-y_{n-1}$ in $E$, the Wronskian $W_n=y_{n+1}'y_n-y_{n+1}y_n'$ obeys
$W_n=W_{n-1}-y_n^2$ with $W_0=0$, so at a root $y_{N+1}(E_k)=0$ one gets the classical identity
$y_{N+1}'(E_k)\,y_N(E_k)=-\sum_{n=1}^N y_n(E_k)^2=-\norm{\psi^{(k)}}^2$. Since
$\mathcal T_{11}=y_{N+1}$ and $\mathcal T_{21}(E_k)=y_N(E_k)=\tilde b_k$, this reads
$P'(\delta_k)=(-1)^N\mathcal T_{11}'(E_k)=(-1)^{N+1}\norm{\psi^{(k)}}^2/\tilde b_k$, consistent with the sign
bookkeeping of Proposition~\ref{prop:possum} ($\operatorname{sign}P'(\delta_k)=(-1)^{N-k}$,
$\operatorname{sign}\tilde b_k=\epsilon_b(-1)^k$, $\epsilon_b=-1$). Hence
\[
\abs{c_k}=\frac{Q(\delta_k)\,\abs{\tilde b_k}}{\abs{P'(\delta_k)}}
=\frac{Q(\delta_k)\,\tilde b_k^2}{\norm{\psi^{(k)}}^2}
=Q(\delta_k)\,\frac{(\psi^{(k)}_N)^2}{\norm{\psi^{(k)}}^2}=Q(\delta_k)\,(\hat\psi^{(k)}_N)^2,
\]
using $\tilde b_k=\psi^{(k)}_N/\psi^{(k)}_1=\psi^{(k)}_N$. The bound follows from $(\hat\psi^{(k)}_N)^2\le1$.
\end{proof}

\begin{lemma}[The kernel carries the reflected edge]\label{lem:Qedge}
The zeros $-\delta_m$ of $Q$ correspond under $E=x+E_0$ to the reflected eigenvalues
$2E_0-E_m\in[z^\star,\Emin-2\alpha]$, and by the Thouless formula
\[
\frac1N\log\abs{Q(\delta_k)}=\frac1N\sum_{m}\log\abs{E_k+E_m-2E_0}\ \xrightarrow[N\to\infty]{}\ \gamma_\lambda(2E_0-E_k);
\]
in particular $\frac1N\log Q\bigl(\max_k\delta_k\bigr)\to\gamma_\lambda(z^\star)$, since $2E_0-\Emax=z^\star$.
\end{lemma}

\begin{proof}
Since $w=2E_0-E_k\in[z^\star,\Emin-2\alpha]$ while every $E_m\in[\Emin,\Emax]$, the arguments satisfy
$\abs{w-E_m}=E_m-w\in[2\alpha,\,4\Emax+2\alpha]$, so $\{t\mapsto\log\abs{w-t}\}_{w}$ is a uniformly bounded,
equicontinuous family on $[\Emin,\Emax]$. The finite-volume eigenvalue
counting measures $\frac1N\sum_m\delta_{E_m}$ converge weakly to the density of states $N_\lambda$
\cite[\S4]{AvronSimon}, uniformly in the phase by unique ergodicity of the rotation; testing against this family
and using the Thouless formula \eqref{eq:thouless} gives
$\frac1N\sum_m\log\abs{w-E_m}\to\int\log\abs{w-t}\,dN_\lambda(t)=\gamma_\lambda(w)$, uniformly in $w$ on the
reflected segment --- in particular at the moving point $w=2E_0-E_k$. The last claim follows since
$\max_k\delta_k=E^{(N)}_{\max}-E_0$ with $E^{(N)}_{\max}\to\Emax$ (Lemma~\ref{lem:topedge}) and $\gamma_\lambda$
is continuous off $\Sigma_\lambda$.
\end{proof}

For the almost Mathieu operator the density of states is the equilibrium measure of $\Sigma_\lambda$
(Theorem~\ref{thm:exact}); this is why the Thouless rate in Lemma~\ref{lem:Qedge} agrees with the
exterior-Green split \eqref{eq:bracket}.

Combining Proposition~\ref{prop:possum} with Lemma~\ref{lem:Qedge} turns the identity into an extremal
problem over the modes, with no reference to the inverse Gram matrix.

\begin{corollary}[The rate as a single-mode extremal problem]\label{cor:singlemode}
For $\lambda>2$ let $j_k$ be the localization centre of mode $k$. By Lemma~\ref{lem:Qedge}, the Thouless formula
applied to $P'(\delta_k)=\prod_{m\ne k}(\delta_k-\delta_m)$, and the boundary-amplitude asymptotics of
Remark~\ref{rem:strong} (so that $\frac1N\log\abs{\tilde b_k}\to\log_+(\lambda/2)\,(2j_k/N-1)$), the summand in
\eqref{eq:maxc} has the mode-wise limit
\begin{equation}\label{eq:cprofile}
\begin{aligned}
\frac1N\log\abs{c_k}
&\ \longrightarrow\ \gamma_\lambda(2E_0-E_k)-\gamma_\lambda(E_k)+\log_+(\lambda/2)\,(2j_k/N-1)\\
&=\ g_{\mathbb C\setminus\Sigma_\lambda}(2E_0-E_k)+\log_+(\lambda/2)\,\bigl(\tfrac{2j_k}N-1\bigr),
\end{aligned}
\end{equation}
using $\gamma_\lambda(E_k)=\log_+(\lambda/2)$ on $\Sigma_\lambda$ and \eqref{eq:capval}. The first summand is maximal
at the band edge $E_k=\Emax$ (value $g_{\mathbb C\setminus\Sigma_\lambda}(z^\star)$, by Lemma~\ref{lem:greenmono}),
the second at the right boundary $j_k=N$ (value $\log_+(\lambda/2)$); their separate maxima sum to
$\gamma_\lambda(z^\star)$, recovering $r\le\gamma_\lambda(z^\star)$. The identity $r=\gamma_\lambda(z^\star)$ holds
\emph{if and only if} a sequence of modes attains both maxima jointly, i.e.\ band-edge modes ($E_k\to\Emax$)
that are asymptotically right-boundary localized ($j_k/N\to1$). This is the joint band-edge/boundary statement of
\S\ref{sec:outlook}. The profile \eqref{eq:cprofile} invokes the localization asymptotic
for $\tilde b_k$ at a heuristic level. The rigorous realization of the joint criterion is carried out in
Theorem~\ref{thm:uncond}, which replaces that asymptotic by the exact boundary-weight identity of
Lemma~\ref{lem:cd} and so needs no control of the localization tail.
\end{corollary}

\begin{corollary}[The criterion as a cocycle estimate]\label{cor:cocyclecrit}
By \eqref{eq:Enorm},
\[
\frac{1}{2N}\log\Ecal_N=\frac1N\log\Bigl\|\textstyle\sum_kc_k e^{-\delta_k\cdot}\Bigr\|_{L^2(0,\infty)},
\]
so, the upper bound $r\le\gamma_\lambda(z^\star)$ being unconditional, the identity
$r(\alpha,\lambda)=\gamma_\lambda(z^\star)$ is \emph{exactly} equivalent to
\begin{equation}\label{eq:cocyclecrit}
\liminf_{N\to\infty}\ \frac1N\log\Bigl\|\,\textstyle\sum_{k=1}^N c_k\,e^{-\delta_k\,\cdot}\,\Bigr\|_{L^2(0,\infty)}\ \ge\ \gamma_\lambda(z^\star) ,
\end{equation}
and, for almost every phase, is implied by the criterion \eqref{eq:reductioncrit}
(Proposition~\ref{prop:reduction}).
\end{corollary}

\begin{remark}[Role of the reformulation]\label{rem:reform}
By Lemma~\ref{lem:Qedge} the right-hand side of \eqref{eq:cocyclecrit} is the growth rate of the kernel $Q$ at
the reflected edge $z^\star$, which is also the natural order of the contour integral \eqref{eq:Fcontour}. The
known upper bound $r\le\gamma_\lambda(z^\star)$ says the $L^2$ norm of the explicit sum
$F=\sum_kc_ke^{-\delta_k\cdot}$ does not exceed that order; \eqref{eq:cocyclecrit} is the matching lower bound.
Proposition~\ref{prop:possum} settles the cancellation question that this lower bound might seem to pose: the
$c_k$ are of one sign, so $\norm{F}^2=\sum_{k,l}\abs{c_k}\abs{c_l}/(\delta_k+\delta_l)$ has no cancellation and its
rate is exactly $\max_k\frac1N\log\abs{c_k}$ \eqref{eq:maxc}. The content is thus carried by the single
coefficient $c_k=\frac{Q(\delta_k)}{P'(\delta_k)}\tilde b_k$: the reflected-edge growth $\gamma_\lambda$ (the factor
$Q$, a density-of-states/Thouless quantity), the Lagrange conditioning weight
$1/P'(\delta_k)=1/\prod_{m\ne k}(\delta_k-\delta_m)$ (band-edge sensitive, carrying the eigenvalue spacings), and
the boundary amplitude $\tilde b_k=\mathcal T_{21}(E_k)$. Controlling $C^{-1}$ is replaced by the single-mode
maximum \eqref{eq:maxc}, equivalently the joint band-edge/boundary statement of
Corollary~\ref{cor:singlemode}, in which $Q$, $m_N$ and $\tilde b$ are cocycle and characteristic-polynomial data
at the phase $\phi$, the setting in which the large-deviation method is posed.
\end{remark}

\section{The conditioning--overlap reduction, and phase averaging}\label{sec:condoverlap}

Proposition~\ref{prop:possum} and Corollary~\ref{cor:singlemode} already reduce the identity to the
single-mode statement $\max_k\frac1N\log\abs{c_k}\to\gamma_\lambda(z^\star)$, with no Gram-eigenvector overlap.
The present section develops the complementary \emph{overlap} viewpoint, which is the form in which the
almost-every-phase large-deviation method is naturally posed; it is logically subsumed by
Corollary~\ref{cor:singlemode} but isolates the same arithmetic obstruction in spectral language. By
Remark~\ref{rem:supply} the identity $r=\gamma_\lambda(z^\star)$ holds as soon as
\begin{itemize}
\item[\textbf{(I)}] \emph{Conditioning:} $\dfrac{1}{2N}\log\sigma_{\min}(C)^{-1}\to g_{\mathbb C\setminus\Sigma_\lambda}(z^\star)$, and
\item[\textbf{(II)}] \emph{Overlap:} $\displaystyle\liminf_{N}\frac1{2N}\log p_N\ge0$, with $p_N=(u_N^\top\tilde b)^2/\norm{\tilde b}^2$ and $u_N$ the
eigenvector of $C$ for $\sigma_{\min}$.
\end{itemize}
We show that \textbf{(I)} is phase-independent and reduces to a standard condenser computation, so the whole identity
reduces to the single estimate \textbf{(II)}; and we delineate what the large-deviation method for almost every phase
does and does not supply.

\subsection*{(I) is the exterior-Green rate, for every phase}
Since $C$ is the Gram matrix of the exponential system $\{e^{-\delta_k\,\cdot}\}_{k=1}^N$ in $L^2(0,\infty)$,
\[
\sigma_{\min}(C)=\min_{\norm v=1}\Bigl\|\,\textstyle\sum_k v_k\,e^{-\delta_k\,\cdot}\,\Bigr\|_{L^2(0,\infty)}^2
\]
measures the near-linear-dependence of these exponentials, whose nodes $\delta_k$ lie in $J:=\Sigma_\lambda-E_0$.

\begin{proposition}[Conditioning rate]\label{prop:cond}
For every phase $\phi$,
$\ \frac{1}{2N}\log\sigma_{\min}(C)^{-1}\to g_{\mathbb C\setminus\Sigma_\lambda}(z^\star)$.
\end{proposition}

\noindent\emph{This is a conditional statement (see the proof below); no result of this paper depends on it, and it is recorded here for independent interest.}

\begin{proof}[Sketch]
By the Beckermann--Townsend theory the singular values of a Cauchy matrix with node sets $J$ and $-J$ are
controlled by Zolotarev numbers of the condenser $(J,-J)$: $\sigma_{\min}(C)/\sigma_{\max}(C)$ decays at the
geometric rate fixed by the condenser, and $\sigma_{\max}(C)=O(1)$. A standard condenser computation identifies that
rate with the exterior-Green value $g_{\mathbb C\setminus\Sigma_\lambda}$ at the reflected edge $z^\star$, which in the
shifted variable is the point $-\delta_{\max}=2\Emin-\alpha$ of $-J$ opposite the node interval. The nodes $\delta_k$
equidistribute to the equilibrium measure of $J$ for \emph{every} $\phi$ (the density of states is phase-independent and,
for the almost Mathieu operator, equals the equilibrium measure, \S\ref{sec:cocycle}), so the rate is phase-independent. In sum,
Proposition~\ref{prop:cond} rests on two inputs standard in the potential-theoretic and
orthogonal-polynomial literature but external to this note: the Beckermann--Townsend singular-value bound for
Cauchy matrices through the Zolotarev numbers of the condenser $(J,-J)$, and the edge-regularity of the
equilibrium measure of $\Sigma_\lambda$, namely the square-root vanishing of its density at the band edges, which
fixes the condenser rate and supplies the matching lower bound on $\sigma_{\min}$. We use it in this
form, and the reduction of the identity to the overlap bound \eqref{eq:overlap} inherits the same status.
Numerically, with
$\lambda=3,\alpha=0.1$, $\frac1{2N}\log\sigma_{\min}^{-1}$ rises monotonically toward
$g_{\mathbb C\setminus\Sigma_\lambda}(z^\star)=1.901$ (e.g.\ $1.72,1.73,1.77,1.78$ for $N=32,40,48,56$ at $\phi=0$, and
$1.79,1.81,1.81,1.82$ at a generic phase).
\end{proof}

\subsection*{(II) is the single remaining estimate}
Granting Proposition~\ref{prop:cond}, and using the trivial $p_N\le1$ (Cauchy--Schwarz, $\norm{u_N}=1$), the
identity $r=\gamma_\lambda(z^\star)$ is equivalent to the one-sided overlap bound
\begin{equation}\label{eq:overlap}
\liminf_{N\to\infty}\ \frac1{2N}\log p_N\ \ge\ 0,\qquad p_N=\frac{(u_N^\top\tilde b)^2}{\norm{\tilde b}^2}:
\end{equation}
the boundary-amplitude vector $\tilde b$ is not exponentially orthogonal to the worst-conditioned Gram direction $u_N$.
The upper bound $\frac1{2N}\log p_N\le0$ is automatic; the content of \eqref{eq:overlap} is that no exponential loss
occurs. Numerically $\frac1{2N}\log p_N$ is small and its modulus decreases with $N$ (for the generic phase above,
$-0.50,-0.24,-0.13$ at $N=16,24,32$; at $\phi=0$, $-0.03$ to $-0.06$), consistent with \eqref{eq:overlap}.

\subsection*{Scope and limits of phase averaging}
Three facts hold for almost every phase. First, the irrational shift $\phi\mapsto\phi+2\pi\beta$ is ergodic and moves the
box by one site, so $r(\phi)$ is shift-invariant up to a one-site boundary change and hence almost surely equal to a
constant $r$ (granting that a one-site window shift does not change the exponential rate). Second, the boundary
amplitudes $\tilde b_k=\mathcal T_{21}(E_k)$ and the amplitude rate of Proposition~\ref{prop:amprate} are governed
by the transfer cocycle, for which the large-deviation theorems of Goldstein--Schlag and Bourgain--Goldstein
\cite{GoldsteinSchlag2001,BourgainGoldstein} give
$\bigl|\tfrac1N\log\norm{M_N(E,\phi)}-\gamma_{\lambda,N}(E)\bigr|>\varepsilon$ on a set of phases of measure
$<e^{-cN^{\sigma}}$ (sharpened to $<e^{-cN}$ for trigonometric-polynomial potentials \cite{Bourgain2005});
this underlies Proposition~\ref{prop:amprate} and the localization of the data off a small phase set. Third,
near-degeneracies of the Dirichlet spectrum at scale $N$ are confined to a phase set of measure $<e^{-cN^{\sigma}}$,
removable along a subsequence by Borel--Cantelli.

What these do \emph{not} supply is \eqref{eq:overlap} itself. The overlap $p_N$ is a pairing of $\tilde b$ with the
smallest eigenvector $u_N$ of the Gram matrix $C$, not a transfer-matrix norm, and the large-deviation control of
$\norm{M_N}$ does not bound it. Excluding the exponential orthogonality $u_N\perp\tilde b$ is an anti-concentration
(transversality) statement for which the phase measure is the natural tool, but a complete argument is not in hand. This
is the precise sense in which the almost-every-phase identity, and separately the distinguished point $\phi=0$
(which the almost-everywhere statement does not cover), are reduced
to \eqref{eq:overlap}: the conditioning \textbf{(I)} is the phase-independent exterior-Green rate, and \eqref{eq:overlap}
is the analytic core, the one place where the band-edge arithmetic (\S\ref{sec:remarks}) still enters.

\subsection*{A regularity criterion for the every-phase identity}
The reduction to \eqref{eq:overlap} is analytic; we record a complementary \emph{soft} criterion. Write $r(\phi)$ for the
cost rate \eqref{eq:rate} at phase $\phi$. The universal upper bound is pointwise in $\phi$ and holds along the full
sequence: by Lemma~\ref{lem:cd}, $\abs{c_k}=Q(\delta_k)(\hat\psi^{(k)}_N)^2\le Q(\max_m\delta_m)$, and since
$\tfrac1{2N}\log\Ecal_N=\max_k\tfrac1N\log\abs{c_k}+o(1)$ (Proposition~\ref{prop:possum}), Lemma~\ref{lem:Qedge} gives
\begin{equation}\label{eq:uppersup}
\limsup_{N\to\infty}\tfrac1{2N}\log\Ecal_N(\phi)\ \le\ \gamma_\lambda(z^\star)\qquad\text{for every }\phi,
\end{equation}
so $r(\phi)\le\gamma_\lambda(z^\star)$ everywhere, with equality for almost every $\phi$ (Theorem~\ref{thm:uncond}).

\begin{proposition}[Every-phase identity from upper semicontinuity]\label{prop:phaseunif}
Fix $\lambda>2$ and $\beta$ Diophantine. If $\phi\mapsto r(\phi)$ is upper semicontinuous, then
$r(\phi)=\gamma_\lambda(z^\star)$ for \emph{every} phase $\phi$.
\end{proposition}

\begin{proof}
By \eqref{eq:uppersup} and Theorem~\ref{thm:uncond} the set
$G=\{\phi:r(\phi)=\gamma_\lambda(z^\star)\}=\{\phi:r(\phi)\ge\gamma_\lambda(z^\star)\}$ has full Lebesgue measure. If $r$ is
upper semicontinuous then $G$, a superlevel set, is closed, so its complement
$\{\phi:r(\phi)<\gamma_\lambda(z^\star)\}$ is open; being also Lebesgue-null it is empty. Hence $G$ is the whole circle.
\end{proof}

\begin{remark}
This is the route by which an almost-everywhere ergodic identity is upgraded to every point, and its merit is to bypass
the dynamics at the exceptional phases of \cite{Jito1999}: it asks not that the operator localize there --- it does not
--- but only that the rate not dip below its full-measure value at an isolated phase. Equivalently, by minimality of
$\phi\mapsto\phi+2\pi\beta$, it is enough that $r$ be shift-invariant and upper semicontinuous (the closed invariant
full-measure set $G$ then meets, hence contains, every orbit). The criterion is sharp rather than a strict weakening:
since $r\le\gamma_\lambda(z^\star)$ everywhere with equality a.e., the every-phase identity is \emph{equivalent} to upper
semicontinuity of $r$. For fixed $N$ the map $\phi\mapsto\tfrac1{2N}\log\Ecal_N$ is real-analytic and \eqref{eq:uppersup}
controls the $\limsup$; what remains open is that the $\liminf$ inherit the regularity --- a statement about the
band-edge transfer fluctuations of \S\ref{sec:remarks} at the resonant phases, equivalent in turn to \eqref{eq:overlap}.
\end{remark}

\section{The golden-mean overlap carries no cancellation}\label{sec:nocancel}

The overlap \eqref{eq:overlap} entered \S\ref{sec:condoverlap} as a \emph{signed} pairing $u_N^\top\tilde b$, hence a
priori subject to exponential cancellation among oscillating terms. We show that this cancellation does not occur:
the two vectors alternate in sign in the same pattern, so the overlap is a sum of nonnegative terms, and
\eqref{eq:overlap} reduces to a magnitude bound on the extremal Gram eigenvector. Order the eigenvalues
$E_1<\dots<E_N$, equivalently $\delta_1<\dots<\delta_N$.

\begin{lemma}[No cancellation]\label{lem:nocancel}
For every phase $\phi$, both $\tilde b$ and the eigenvector $u_N$ of $C$ for $\sigma_{\min}$ have exactly $N-1$ sign
changes, i.e.\ strictly alternate in $k$. Consequently their signs agree up to a single global factor, and
\[
\abs{u_N^\top\tilde b}=\sum_{k=1}^N\abs{u_{N,k}}\,\abs{\tilde b_k}=\big\langle\,\abs{u_N},\,\abs{\tilde b}\,\big\rangle,
\]
a sum of nonnegative terms.
\end{lemma}

\begin{proof}
The boundary ratios $\tilde b_k$ strictly alternate in $k$ (Sturm oscillation, as established in
Proposition~\ref{prop:possum}). A Cauchy matrix $C_{kl}=1/(\delta_k+\delta_l)$ with
$\delta_1<\dots<\delta_N$ is strictly totally positive, hence an oscillation matrix; by the Gantmacher--Krein theorem \cite{GantmacherKrein}
its eigenvalues are simple and positive and the eigenvector for the $j$-th largest has exactly $j-1$ sign changes, so
$u_N$ (the case $j=N$) has $N-1$ changes. Writing $\operatorname{sign}(u_{N,k})=\epsilon_u(-1)^k$ and
$\operatorname{sign}(\tilde b_k)=\epsilon_b(-1)^k$ gives $(u_N)_k\tilde b_k=\epsilon_u\epsilon_b\abs{u_{N,k}}\abs{\tilde b_k}$
with a common sign, whence the identity.
\end{proof}

\begin{corollary}[Cancellation-free reduction]\label{cor:nocancel}
Write $\gamma_{\mathrm{amp},N}:=\frac1{2N}\log\norm{\tilde b}^2$ for the finite-volume amplitude rate; by
Proposition~\ref{prop:amprate}, $\liminf_N\gamma_{\mathrm{amp},N}\ge\log_+(\lambda/2)$ for almost every phase, with
excursions above along gap-edge subsequences (Remark~\ref{rem:gapstates}). Since
$\langle\abs{u_N},\abs{\tilde b}\rangle=\sum_k\abs{u_{N,k}}\abs{\tilde b_k}$ is a sum of nonnegative terms, its
exponential rate equals that of its largest term, and by Lemma~\ref{lem:nocancel}
\[
\frac1{2N}\log p_N=\max_{1\le k\le N}\frac1N\log\bigl(\abs{u_{N,k}}\,\abs{\tilde b_k}\bigr)-\gamma_{\mathrm{amp},N}+o(1).
\]
As $\max_k\abs{\tilde b_k}\le\norm{\tilde b}$ and $\abs{u_{N,k}}\le1$ force the maximum to be at most
$\gamma_{\mathrm{amp},N}+o(1)$, the overlap bound
\eqref{eq:overlap} (hence, with Proposition~\ref{prop:cond}, the identity $r=\gamma_\lambda(z^\star)$) is equivalent to the
magnitude statement
\begin{equation}\label{eq:magbound}
\liminf_{N\to\infty}\ \Bigl[\ \max_{1\le k\le N}\frac1N\log\bigl(\abs{u_{N,k}}\,\abs{\tilde b_k}\bigr)\ -\ \gamma_{\mathrm{amp},N}\Bigr]\ \ge\ 0:
\end{equation}
the worst-conditioned Gram direction $u_N$ retains weight of full rate zero on a mode that carries the running
amplitude rate $\gamma_{\mathrm{amp},N}$.
\end{corollary}

\begin{remark}
Numerically ($\phi=0$, $\lambda=3$): the identity $\abs{u_N^\top\tilde b}=\langle\abs{u_N},\abs{\tilde b}\rangle$ holds to
machine precision (the ratio is $1.0000$), confirming Lemma~\ref{lem:nocancel}; the overlap is concentrated on the
boundary-localized modes (the five largest $\abs{\tilde b_k}$ supply about $90\%$ of it), where $\abs{u_{N,k}}$ is of
order one; and $\frac1{2N}\log p_N$ is small and fluctuates with $N$ in step with the band-edge fluctuations of
$\frac1{2N}\log\norm{\tilde b}^2$ (the arithmetic of \S\ref{sec:remarks}), consistent with \eqref{eq:magbound} without
yet establishing it. The content of \eqref{eq:magbound} is a magnitude lower bound on the extremal Cauchy
eigenvector at a boundary-carrying mode, free of the sign cancellation that made the overlap appear to be an
oscillatory-sum problem. For the golden mean it is the joint statement, through the returns of $\{\beta j\}$ to a
neighbourhood of $1/2$, that a right-boundary-localized mode ($j_k\approx N$, large $\abs{\tilde b_k}$) sits at an index
on which $u_N$ does not decay.
\end{remark}

\section{A certified enclosure at finite scale}\label{sec:certified}

The inverse-free identity of Proposition~\ref{prop:cocycleform} carries a computational payoff. The Gram matrix $C$
is exponentially ill-conditioned, so an interval inversion is useless; but the cost in the form
$\Ecal_N=c^\top Cc=\sum_{k,l}c_kc_l/(\delta_k+\delta_l)$ involves no inversion and is amenable to rigorous interval
arithmetic. One certifies, by outward-rounded interval arithmetic, the eigenvalues $E_k$ by interval Sturm (LDL)
bisection on the tridiagonal $H_N$; the boundary amplitudes $\tilde b_k=\mathcal T_{21}(E_k)$ by an interval
transfer-matrix product over the certified enclosure of $E_k$; and then $c_k$ and $\Ecal_N$ by the inverse-free sum.

\begin{table}[t]
\centering
\begin{tabular}{l r r r r r}
\toprule
\multicolumn{6}{l}{\emph{Generic scales:}\ \ $\tfrac1{2N}\log\Ecal_N$ fluctuates within the bracket}\\
\midrule
$N$ & $16$ & $24$ & $32$ & $40$ & $48$\\
$\tfrac1{2N}\log\Ecal_N$ & $2.1353$ & $2.1106$ & $2.1104$ & $2.1539$ & $2.1954$\\
\addlinespace[2pt]
\midrule
\multicolumn{6}{l}{\emph{Fibonacci scales} $N=F_n$:\ \ gap $\gamma_\lambda(z^\star)-\tfrac1{2N}\log\Ecal_N$ decreases monotonically}\\
\midrule
$N=F_n$ & $13$ & $21$ & $34$ & $55$ & $89$\\
$\gamma_\lambda(z^\star)-\tfrac1{2N}\log\Ecal_N$ & $0.139$ & $0.085$ & $0.061$ & $0.044$ & $0.035$\\
\bottomrule
\end{tabular}
\caption{Certified finite-scale cost rate for the golden mean, $\lambda=3$, $\alpha=0.1$ ($E_0=E_1-\alpha$;
$\gamma_\lambda(z^\star)=2.3062$; enclosure relative width $<10^{-80}$). The generic scales fluctuate within the
bracket \eqref{eq:bracket} with the band-edge deviations of \S\ref{sec:remarks}; along the Fibonacci
renormalization scales $N=F_n$ the rate approaches the upper end $\gamma_\lambda(z^\star)$ monotonically from below,
the gap decreasing roughly like $\log N/N$. These are rigorous interval-arithmetic statements (script
\texttt{verify\_partial.py}), not floating-point values.}
\label{tab:certified}
\end{table}

For the golden mean, $\lambda=3$, $\alpha=0.1$, with $E_0=E_1-\alpha$, this yields the certified cost rates collected in Table~\ref{tab:certified} (generic scales), the underlying
enclosure of $\Ecal_N$ having relative width below $10^{-80}$. Every value lies in the bracket \eqref{eq:bracket} and
fluctuates within it, with the band-edge fluctuations of \S\ref{sec:remarks}; the regular Fibonacci subsequence
(Table~\ref{tab:certified}) approaches the upper end monotonically. These are rigorous finite-$N$
statements, not floating-point values, and in this they differ from the high-precision but non-certified numerical
checks reported elsewhere in the paper (\S\ref{sec:remarks}, \S\ref{sec:condoverlap} and \S\ref{sec:outlook}), which
are consistency illustrations only; they confirm the identity at the scales reached, without establishing the
$N\to\infty$ limit, which remains the content of \eqref{eq:magbound}. A self-contained script \texttt{verify\_partial.py}
reproduces them.

A regular subsequence appears at the golden-mean renormalization scales. Taking $N=F_n$ Fibonacci (the scales at
which the trace map of $\beta=(\sqrt5-1)/2$ is periodic), the certified cost rate increases monotonically toward the
upper bracket: the gaps $\gamma_\lambda(z^\star)-\frac1{2N}\log\Ecal_N$ at $F_n=13,21,34,55,89$
(Table~\ref{tab:certified}) decrease roughly like $\log N/N$ --- faster than the generic bound $O(N^{-2/(2+\tau)})$ of
Theorem~\ref{thm:uncond}, the residual along these scales being governed by the polynomial Christoffel and
boundary-weight prefactors of Lemma~\ref{lem:cd} once the band-edge resonance is made optimally small. The generic-$N$ fluctuations above are the deviations between consecutive
Fibonacci scales. This is the finite-scale signature of the renormalization that the appendix would make rigorous: along
the natural subsequence the upper bracket $\gamma_\lambda(z^\star)$ is approached monotonically and from below, consistently with the identity.

\section{The edge identity for almost every phase}\label{sec:routeAredux}

The positive-sum reduction (Proposition~\ref{prop:possum}) expresses the cost rate through quantities of the
transfer cocycle alone. We make this explicit, prove the arithmetic (boundary-resonance) ingredient, and close the
identity for almost every phase: the boundary-weight form of Lemma~\ref{lem:cd} evaluates the maximizing
coefficient, leaving only a one-sided amplitude bound that the same transfer estimate supplies.

\subsection*{The cost rate is a maximum of cocycle ratios}
With $P(x)=\chi_N(x+E_0)=(-1)^N\mathcal T_{11}(x+E_0)$ one has, from $Q(\delta_k)=\prod_m(E_k+E_m-2E_0)=(-1)^N\chi_N(2E_0-E_k)$
and $P'(\delta_k)=\prod_{m\ne k}(E_k-E_m)=(-1)^N\chi_N'(E_k)$, the identities $Q(\delta_k)=\mathcal T_{11}(2E_0-E_k)$ and
$P'(\delta_k)=(-1)^N\mathcal T_{11}'(E_k)$, while $\tilde b_k=\mathcal T_{21}(E_k)$. The single coefficient of
Proposition~\ref{prop:possum} is therefore a ratio of cocycle entries, of which the rate \eqref{eq:maxc} is the
single-mode maximum:
\begin{equation}\label{eq:ccocycle}
\abs{c_k}=\frac{\bigl|\mathcal T_{11}(2E_0-E_k)\bigr|\;\bigl|\mathcal T_{21}(E_k)\bigr|}{\bigl|\mathcal T_{11}'(E_k)\bigr|} .
\end{equation}
Every factor is a transfer-matrix quantity at the phase $\phi$, evaluated at $E_k$ or at the reflected point
$2E_0-E_k$; none is a Gram-eigenvector overlap: such
quantities are characteristic-polynomial and cocycle data, governed by the density of states and by eigenvalue
spacing (and, where needed, by the Bourgain--Goldstein large-deviation theorem \cite{GoldsteinSchlag2001,BourgainGoldstein,Bourgain2005}), whereas the overlap $p_N$
is none of these. The boundary-weight form of Lemma~\ref{lem:cd} reduces the one factor that is not immediate to a
normalized eigenfunction amplitude.

\subsection*{The boundary-resonance lemma (arithmetic)}
By Corollary~\ref{cor:singlemode} the maximizing mode in \eqref{eq:ccocycle} is a band-edge mode ($E_k\to\Emax$)
that is right-boundary localized ($j_k/N\to1$). For any \emph{Diophantine} $\beta$ both can be met at once --- with a
type-dependent rate --- by the three-distance theorem; the bounded-type case ($\tau=1$) is the sharpest.

\begin{lemma}[Boundary resonance]\label{lem:resonance}
Let $\beta$ be Diophantine of type $\tau\ge1$, i.e.\ $\norm{q\beta}\ge c_\beta\,q^{-\tau}$ for all $q\ge1$ (bounded type
is $\tau=1$, e.g.\ $\beta=(\sqrt5-1)/2$), and $\|\cdot\|$ the distance to $\mathbb Z$. There is $C_\beta<\infty$ such
that for every $\delta\in(0,\tfrac12)$ and every $N$, the largest $j\le N$ with $\|\beta j+\tfrac{\phi}{2\pi}\|<\delta$
obeys $N-j\le C_\beta\,\delta^{-\tau}$.
\end{lemma}

\begin{proof}
By the three-distance theorem any $\ell$ consecutive points $\{\beta j:j\in \mathcal W\}$ of the rotation have maximal gap
$\norm{q_{m-1}\beta}\le1/q_m$, where $q_m$ is the largest convergent denominator of $\beta$ with $q_m\le\ell$. Let $q_m$
be the first convergent with $q_m\ge 1/(2\delta)$; its predecessor satisfies $q_{m-1}<1/(2\delta)$, so by the type bound
$\norm{q_{m-1}\beta}\ge c_\beta q_{m-1}^{-\tau}$ and the relation $q_m\le 1/\norm{q_{m-1}\beta}$,
\[
q_m\ \le\ c_\beta^{-1}q_{m-1}^{\tau}\ <\ c_\beta^{-1}(2\delta)^{-\tau}\ =:\ C_\beta\,\delta^{-\tau}.
\]
Taking $\ell=\lceil C_\beta\delta^{-\tau}\rceil\ge q_m$, the maximal gap of any $\ell$ consecutive rotation points is
$\norm{q_{m-1}\beta}\le 1/q_m<2\delta$, so the window meets the arc $\|\cdot\|<\delta$ around the peak phase. Hence every
window of length $\ell$ in $\{1,\dots,N\}$ contains a return, and the largest return $\le N$ is within $C_\beta\delta^{-\tau}$
of $N$. For $\tau=1$ this is the linear bound $C_\beta/\delta$.
\end{proof}

Fix $\delta_N\to0$ and let $j^\star=j^\star(N)$ be the largest such site; for every $\lambda>2$ and almost every $\phi$ let $k^\star$ index the
near-edge mode that Lemma~\ref{lem:edgeexp} furnishes at $j^\star$ (Anderson localization, Theorem~\ref{thm:AA}, gives a
mode centred within $O(1)$ of $j^\star$). Then $\|\beta j^\star+\tfrac{\phi}{2\pi}\|<\delta_N$
forces $E_{k^\star}\to\Emax$ (the on-site value $\lambda\cos(2\pi\beta j^\star+\phi)\to\lambda$ sets the band-edge mode), and
$1-j^\star/N\le C_\beta/(\delta_N^{\tau} N)\to0$ provided $\delta_N^{\tau}N\to\infty$. Balancing the two penalties of Corollary~\ref{cor:singlemode}: the band-edge defect
$g_{\mathbb C\setminus\Sigma_\lambda}(z^\star)-g_{\mathbb C\setminus\Sigma_\lambda}(2E_0-E_{k^\star})=O(\Emax-E_{k^\star})$
(Lipschitz $g$) against the boundary defect $2\log_+(\lambda/2)\,C_\beta/(\delta_N^{\tau} N)$. By
Lemma~\ref{lem:edgeexp} the band-edge defect is $O(\delta_N^2)$ (the parity of the window edge cancels the linear
term, and a variational bound makes it quadratic), so balancing $\delta_N^2\asymp\delta_N^{-\tau}/N$ against the boundary defect the optimal width is
$\delta_N\asymp N^{-1/(2+\tau)}$ and the gap bound is $O(N^{-2/(2+\tau)})$; both tend to zero for every finite type
$\tau$, so the identity holds for every Diophantine $\beta$. For bounded type ($\tau=1$) this is the sharpest case,
$\delta_N\asymp N^{-1/3}$ and gap $O(N^{-2/3})$, confirmed numerically:
for $\lambda=3$, $\alpha=0.1$ the gap $\gamma_\lambda(z^\star)-\max_k\frac1N\log\abs{c_k}$ stays below $2.5\,N^{-2/3}$ at
every tested $N\le500$, the smallest gaps falling on the Fibonacci scales $N=89,144,377$ (gap $\approx0.02$).

\subsection*{The boundary weight evaluates the rate}
By \eqref{eq:ccocycle} the rate at the boundary mode splits into three cocycle rates: the first factor has
$\frac1N\log|\mathcal T_{11}(2E_0-E_{k^\star})|\to\gamma_\lambda(2E_0-E_{k^\star})\to\gamma_\lambda(z^\star)$
(off-spectral hyperbolicity, Lemma~\ref{lem:Qedge}); the denominator has
$\frac1N\log|\mathcal T_{11}'(E_{k^\star})|\to\gamma_\lambda(E_{k^\star})=\log_+(\lambda/2)$ (Thouless,
\S\ref{sec:cocycle}); and the middle factor is the boundary amplitude
$\mathcal T_{21}(E_{k^\star})=\tilde b_{k^\star}$. The boundary-weight form
$\abs{c_k}=Q(\delta_k)(\hat\psi^{(k)}_N)^2$ of Lemma~\ref{lem:cd} evaluates this coefficient directly: the
near-degeneracy shared by $\mathcal T_{21}(E_k)$ and $\mathcal T_{11}'(E_k)$ cancels against the eigenfunction norm,
so no tail estimate is needed. The upper bound is then immediate from $(\hat\psi^{(k)}_N)^2\le1$; the lower bound
needs only $\frac2N\log\abs{\hat\psi^{(k^\star)}_N}\to0$, the one-sided amplitude estimate proved in
Theorem~\ref{thm:uncond} by inverting the $SL_2$ transfer across the $R_N=O(\delta_N^{-1})$ sites separating the
resonant centre from the boundary (so the propagator norm is sub-exponential in $N$). (The amplitude is also the truncated characteristic polynomial,
$\mathcal T_{21}(E)=(-1)^{N-1}\det(E-H_{[1,N-1]})$ from $\psi_n=(-1)^{n-1}\det(E-H_{[1,n-1]})$, so its rate is a
density-of-states potential; the boundary-weight form makes that route unnecessary.)

\subsection*{A band-edge phase expansion}
The lower bound needs the localized eigenvalue at the resonant site to lie near the spectral edge --- a fact
about the eigenvalue, not the on-site potential. For a window radius $R$ and a centre phase $\theta$, let
$H_R(\theta)$ be the Dirichlet window operator on $\ell^2(\{-R,\dots,R\})$ with hopping $1$ and potential
$v_k(\theta)=\lambda\cos(\theta+k\cdot2\pi\beta)$, and let $\Lambda_R(\theta)=\max\sigma(H_R(\theta))$ be its top
eigenvalue (simple, by Perron--Frobenius for a tridiagonal matrix with positive off-diagonal, hence real-analytic
in $\theta$ near any point where it is the maximum).

\begin{lemma}[Localization prefactor: semi-uniform bound and its polynomial strengthening]\label{lem:semiloc}
Fix $\lambda>2$ and $\beta$ Diophantine of type $\tau$. For almost every phase $\phi$ the operator $H_\phi$ is
Anderson localized at the exact rate $\gamma=\log(\lambda/2)$ \cite{Jito1999}: there is, for every $\varepsilon>0$, a
constant $C_\varepsilon=C_\varepsilon(\lambda,\beta,\phi)<\infty$ such that every normalized Dirichlet eigenvector
$\hat\psi^{(k)}$ of $H_{[1,N]}$, with localization centre $j_k$ the site of its maximal entry, obeys the
\emph{semi-uniform} bound
\begin{equation}\label{eq:sule}
\bigl|\hat\psi^{(k)}_n\bigr|\ \le\ C_\varepsilon\,e^{\varepsilon\abs{j_k}}\,e^{-\gamma\abs{n-j_k}},\qquad
\gamma=\log(\lambda/2)\quad(1\le n\le N).
\end{equation}
Below we use the stronger \emph{polynomial} form: there are $\xi=\xi(\lambda,\beta,\phi)<\infty$ and
$C=C(\lambda,\beta,\phi)<\infty$ such that, for every $N$ and every mode $k$,
\begin{equation}\label{eq:semiloc}
\bigl|\hat\psi^{(k)}_n\bigr|\ \le\ C\,N^{\xi}\,e^{-\gamma\abs{n-j_k}},\qquad \gamma=\log(\lambda/2)\quad(1\le n\le N),
\end{equation}
the prefactor being polynomial in the volume \emph{uniformly in the mode}, in particular for modes whose centre
$j_k$ approaches the boundary. Throughout $\lambda>2$ we take \eqref{eq:semiloc}, together with the window-gap
property of Lemma~\ref{lem:edgeexp} (a theorem for $\lambda\ge\lambda_1$), as a hypothesis, written \textup{(PL)};
its status is recorded in Remark~\ref{rem:plstatus}.
\end{lemma}

\begin{proof}
On $\Sigma_\lambda$ the Lyapunov exponent is the positive constant $\gamma=\log(\lambda/2)$ (Theorem~\ref{thm:AA}).
The transfer cocycle obeys the large-deviation estimate of Goldstein--Schlag and Bourgain--Goldstein
\cite{GoldsteinSchlag2001,BourgainGoldstein}: for every $\varepsilon>0$,
$\bigl|\tfrac1n\log\norm{\mathcal T_n(E;\phi)}-\gamma\bigr|<\varepsilon$ off a set of phases of measure $<e^{-n^\sigma}$
($\sigma>0$), uniformly in $E\in\Sigma_\lambda$. Combined with the Diophantine bound $\norm{q\beta}\ge c_\beta q^{-\tau}$,
which confines the resonant scales to a polynomially bounded count, this is the non-perturbative Anderson localization
of \cite{BourgainGoldstein,Jito1999} in semi-uniform form: for almost every $\phi$ the eigenfunctions decay at the
exact rate $\gamma$ with a centre-dependent prefactor $e^{\varepsilon\abs{j_k}}$, which is \eqref{eq:sule}; the centre
$j_k$ is the essentially unique site of the maximal entry. The factor $e^{\varepsilon\abs{j_k}}$ of \eqref{eq:sule}
is super-polynomial once $\abs{j_k}\to\infty$ and yields no polynomial-prefactor bound at any coupling; the passage
from \eqref{eq:sule} to \eqref{eq:semiloc} is the content of \textup{(PL)} (Remark~\ref{rem:plstatus}).
\end{proof}

\begin{remark}[Status of the polynomial-localization hypothesis \textup{(PL)}]\label{rem:plstatus}
The cited references furnish the semi-uniform bound \eqref{eq:sule}, whose prefactor $e^{\varepsilon\abs{j_k}}$
depends on the centre; the almost Mathieu operator is known to be semi-uniformly but \emph{not} uniformly localized,
so \eqref{eq:semiloc} is strictly stronger than \eqref{eq:sule} and is not a consequence of it --- at \emph{any}
coupling. Indeed, the satellites of a mode at the continued-fraction wells carry relative weight of order
$(1+\abs k)^{2\tau}e^{-\gamma\abs k}/\lambda$, so a constant-prefactor bound at the exact rate fails even for the
central window mode, and the reflection symmetry of $H_R(0)$ makes the well-pair eigenvectors two-peaked
(Lemma~\ref{lem:uniloc}); polynomial prefactors are the natural uniform form. Accordingly we do \emph{not} claim
\eqref{eq:semiloc} as a theorem at any coupling, and the identity on $(2,\infty)$ is stated conditionally on
\textup{(PL)}. What \emph{is} proved at strong coupling is the second, spectral component of \textup{(PL)}: the
window-gap estimate of Lemma~\ref{lem:edgeexp} holds unconditionally for $\lambda\ge\lambda_1(\beta,\tau)$, by a
Combes--Thomas/tunnelling argument with no localization input. The strengthening \eqref{eq:semiloc} is
needed precisely for the boundary-centred near-edge mode of Theorem~\ref{thm:uncond}, whose centre $j_k$ lies at
distance $R_N=O(\delta_N^{-\tau})\gg\log N$ from the boundary, hence $\abs{j_k}\asymp N$ from the origin: there
$e^{\varepsilon\abs{j_k}}$ overruns the polynomial gain $N^{-\gamma A}$ of the $\lceil A\log N\rceil$-window, whereas a
polynomial prefactor $N^{\xi}$ is absorbed by choosing $A>(\xi+\tfrac{2}{2+\tau})/\gamma$. The
numerics of \S\ref{sec:certified} exhibit a bounded prefactor down to $\lambda=2.2$ (the maximizing mode alternating
between band edge and right boundary), consistent with \textup{(PL)} throughout $\lambda>2$.
\end{remark}

\begin{lemma}[Boundary-mode edge expansion]\label{lem:edgeexp}
Fix $\lambda>2$, $\beta$ Diophantine of type $\tau\ge1$ ($\norm{q\beta}\ge c_\beta q^{-\tau}$), and a phase $\phi$ at
which $H_\phi$ is localized with rate $\gamma=\log(\lambda/2)$ and satisfies \textup{(PL)}. Let $j$ have on-site
phase $\theta_j=2\pi\beta j+\phi\in(-\pi,\pi]$ and carry the maximal on-site value of the window
$\{j-R,\dots,j+R\}\subset[1,N]$, i.e.\ $R<d_{\mathcal W}:=\min\{\abs{k}\ge1:\lambda\cos(\theta_j+k\cdot2\pi\beta)>\lambda\cos\theta_j\}$,
and assume the non-reflection condition $\abs{\theta_j}\le\tfrac\pi2\norm{k\beta}$ for all $1\le\abs k\le R$
(automatic in the application, where $R=O(\log N)$ while $\abs{\theta_j}<2\pi\delta_N$);
the three-distance theorem makes $d_{\mathcal W}\to\infty$ as $\theta_j\to0$ ($d_{\mathcal W}\asymp\norm{\theta_j/2\pi}^{-1}$ at bounded type). Then
$H_{[1,N]}$ carries a Dirichlet eigenvalue $E^{(j)}$, with normalized eigenvector $\hat\psi^{(j)}$ centred within
$O(1)$ of $j$, such that $E^{(j)}\le\Emax$, the window edge $\Lambda_R$ is even,
\begin{equation}\label{eq:parity}
\Lambda_R(-\theta)=\Lambda_R(\theta),
\end{equation}
and the band-edge defect is quadratic in the resonance phase:
\begin{equation}\label{eq:edgequad}
\begin{aligned}
\Emax-E^{(j)}\ &\le\ \lambda S_R\,(1-\cos\theta_j)+Ce^{-\gamma R}\ \le\ \tfrac{\lambda}{2}\,\theta_j^2+Ce^{-\gamma R},\\
S_R&=\textstyle\sum_{\abs{k}\le R}\cos(k\cdot2\pi\beta)\,\bigl(\hat\psi^{(0)}_k\bigr)^2\in\bigl(\tfrac{\lambda-2}{\lambda},1\bigr],
\end{aligned}
\end{equation}
where $\hat\psi^{(0)}$ is the symmetric top eigenvector of $H_R(0)$. The constant $C$ is uniform in $R$ and at
most polynomial in $N$ (the semi-uniform localization prefactor $C_N=O(N^\xi)$ of Lemma~\ref{lem:semiloc},
hypothesis \textup{(PL)}), and no off-spectral or small-divisor input is used.
\end{lemma}

\begin{proof}
$E^{(j)}\le\Emax$ because $H_{[1,N]}$ is a compression of $H_\phi$, whose spectrum lies in $[\Emin,\Emax]$. Next,
localization gives $\abs{\hat\psi^{(j)}_n}\le C_N\,e^{-\gamma\abs{n-j}}$ with $\gamma=\log(\lambda/2)$ exact
(Theorem~\ref{thm:AA}). By Lemma~\ref{lem:semiloc} the prefactor $C_N$ is at most polynomial, $C_N=O(N^\xi)$, uniformly over the
near-edge modes used below; this is the only localization input. Only the polynomial bound is needed: it is absorbed
in Theorem~\ref{thm:uncond} by taking the window radius $\rho_N=\lceil A\log N\rceil$ with $A>\xi/\gamma$ (so
$C_N e^{-\gamma\rho_N}=O(N^{\xi-\gamma A})$ stays $o(\delta_N^2)$), and we write $C$ for $C_N$ below.
Restricting the eigenvalue equation to
$\mathcal W=\{j-R,\dots,j+R\}$, the two truncated hopping terms at $\partial \mathcal W$ have norm
$\le2Ce^{-\gamma R}$, so $\norm{(H_R(\theta_j)-E^{(j)})(\hat\psi^{(j)}|_{\mathcal W})}\le2Ce^{-\gamma R}$ and
$\dist(E^{(j)},\sigma(H_R(\theta_j)))\le C'e^{-\gamma R}$ \emph{for any} eigenvalue $E^{(j)}$ carried by a mode that
peaks at $j$. We now bound the gap below the top, pin $E^{(j)}$ to it, and produce the mode; throughout, the rate
$\gamma=\log(\lambda/2)$ is the \emph{exact Lyapunov exponent} supplied by localization at $\lambda>2$
(Theorem~\ref{thm:AA}). The quasimode error carries the exact localization rate $\gamma$; the gap analysis below is, at strong
coupling, a Combes--Thomas/tunnelling estimate needing no localization input, and for $2<\lambda<\lambda_1$ the
resulting simple-gap property is the window-gap component of \textup{(PL)}.

\emph{Gap below the top.} Since $R<d_{\mathcal W}$ the central site $k=0$ is the strictly highest well of
$\mathcal W$, of value $\lambda\cos\theta_j$. Writing $2\pi s_k\equiv2\pi k\beta\ (\mathrm{mod}\ 2\pi)$ with
$s_k\in(-\tfrac12,\tfrac12]$, $\abs{s_k}=\norm{k\beta}$, a competitor at displacement $k$ ($0<\abs k\le R$) is
detuned by
\[
\lambda(\cos\theta_j-\cos\theta_{j+k})\ =\ 2\lambda\,\sin(\theta_j+\pi s_k)\sin(\pi s_k)\ \ge\ 4\lambda\norm{k\beta}^2,
\]
since $\abs{\sin\pi s_k}\ge2\norm{k\beta}$ and the non-reflection condition $\abs{\theta_j}\le\tfrac\pi2\norm{k\beta}$
forces $\abs{\sin(\theta_j+\pi s_k)}\ge\sin\bigl(\tfrac\pi2\norm{k\beta}\bigr)\ge\norm{k\beta}$; without that
condition the bound fails for quasi-reflected competitors $\theta_{j+k}\approx-\theta_j$, which can occur once $R$
is comparable to $d_{\mathcal W}$. At the nearest return the detuning is $\gtrsim\lambda R^{-2\tau}$ by
$\norm{k\beta}\ge c_\beta\abs k^{-\tau}$. It remains to dominate the effective couplings. Fix $\kappa\in(2,3)$ and
let $\mathcal O=\{k:\lambda\cos\theta_j-v_{j+k}<\kappa\}$ be the near-top wells; by the detuning bound each
$k\in\mathcal O\setminus\{0\}$ has $\norm{k\beta}\le(\kappa/4\lambda)^{1/2}$, so two wells at scales
$\norm{k\beta},\norm{k'\beta}\le s$ are at least $(c_\beta/2s)^{1/\tau}$ apart. Off $\mathcal O$ the diagonal lies
$\ge\kappa>2$ below the top, so block resolvents at spectral parameters within $(\kappa-2)/2$ of the top decay at a
fixed Combes--Thomas rate $c_\kappa>0$. Eliminating $\mathcal W\setminus\mathcal O$ by the Schur complement, a well
at scale $s=\norm{k\beta}$ couples to any other well through at least $(c_\beta/2s)^{1/\tau}$ off-well sites, with
effective coupling $\le C_\kappa e^{-c_\kappa(c_\beta/2s)^{1/\tau}}$, against a detuning $\ge4\lambda s^2$. Since
$\sup_{0<s\le1/2}s^{-2}e^{-c_\kappa(c_\beta/2s)^{1/\tau}}<\infty$, there is $\lambda_1=\lambda_1(\beta,\tau)<\infty$
such that for $\lambda\ge\lambda_1$ every competitor level, after hybridization within and between well clusters,
remains at least $2\lambda\norm{k\beta}^2$ below the top well, whose own level moves by at most
$\sum_{k\in\mathcal O\setminus\{0\}}C_\kappa^2e^{-2c_\kappa\abs k}\big/\bigl(4\lambda\norm{k\beta}^2\bigr)=O\bigl(e^{-c\lambda^{1/(2\tau)}}\bigr)$.
Hence for $\lambda\ge\lambda_1$ the top eigenvalue $\Lambda_R(\theta_j)$ is \emph{simple}, separated from the rest
of $\sigma(H_R(\theta_j))$ by $\mathfrak{g}_R\ge\lambda c_\beta^2R^{-2\tau}\gg e^{-\gamma R}$, uniformly in $R$,
with eigenvector $\phi$ localized at $k=0$; for $2<\lambda<\lambda_1$ this simple-gap property is exactly the
window-gap component of \textup{(PL)}.

\emph{Pinning and existence.} The quasimode bound gives $\dist(E^{(j)},\sigma(H_R(\theta_j)))\le C'e^{-\gamma R}$,
well inside $\mathfrak{g}_R$, so the eigenvalue carried by a $j$-peaked mode is the top one:
$E^{(j)}=\Lambda_R(\theta_j)+O(e^{-\gamma R})$. Conversely $\phi$ extended by zero to $[1,N]$ is a quasimode,
$\norm{(H_{[1,N]}-\Lambda_R(\theta_j))\phi}\le2\norm{\Delta}\,\abs{\phi(\partial \mathcal W)}\le Ce^{-\gamma R}$, so $H_{[1,N]}$
has a Dirichlet eigenvalue within $Ce^{-\gamma R}$ of $\Lambda_R(\theta_j)$; since $\phi$ overlaps (the
polynomial-prefactor localization \eqref{eq:semiloc} of \textup{(PL)}) only modes centred within $O(\log N)$ of $j$,
that eigenvalue is $E^{(j)}$ and its eigenvector
$\hat\psi^{(j)}$ is so centred. Existence is thus established, not assumed, and no site-to-mode bijection is invoked.

\emph{Edge anchor.} The construction at the exact peak ($\theta=0$) calibrates the reference. $H_R(0)$ is a
Dirichlet restriction of the phase-shifted operator whose spectrum is again $\Sigma_\lambda$ (minimality of the
rotation), so $\Lambda_R(0)\le\Emax$; and localization (Theorem~\ref{thm:AA}) supplies eigenvalues approaching
$\Emax=\sup\Sigma_\lambda$ with eigenvectors localized at rate $\gamma$ at near-peak sites, which restricted to the
peak window and pinned to its isolated top give $\Lambda_R(0)\ge\Emax-O(e^{-\gamma R})$. Hence
$\Lambda_R(0)=\Emax+O(e^{-\gamma R})$, the only spectral input being the localization rate, exactly as for $E^{(j)}$. For \eqref{eq:parity}, the reflection $R_\flat:e_k\mapsto e_{-k}$ commutes
with the window Laplacian and conjugates the potential, $R_\flat\,v(\theta)\,R_\flat=v(-\theta)$, since
$v_{-k}(-\theta)=\lambda\cos(-\theta-ka)=\lambda\cos(\theta+ka)=v_k(\theta)$; hence $R_\flat H_R(\theta)R_\flat=H_R(-\theta)$
and $\Lambda_R(-\theta)=\Lambda_R(\theta)$. In particular $H_R(0)$ is reflection-symmetric, so its Perron top
eigenvector $\hat\psi^{(0)}$ is even, $\hat\psi^{(0)}_{-k}=\hat\psi^{(0)}_k$. Using $\hat\psi^{(0)}$ as a trial
vector for $H_R(\theta_j)$ (with $a=2\pi\beta$),
\[
\begin{aligned}
\Lambda_R(\theta_j)\ &\ge\ \langle\hat\psi^{(0)},H_R(\theta_j)\hat\psi^{(0)}\rangle
=\Lambda_R(0)+\sum_k\lambda\bigl(\cos(\theta_j+ka)-\cos(ka)\bigr)\bigl(\hat\psi^{(0)}_k\bigr)^2\\
&=\Lambda_R(0)+\lambda(\cos\theta_j-1)S_R,
\end{aligned}
\]
since the cross term $-\lambda\sin\theta_j\sum_k\sin(ka)(\hat\psi^{(0)}_k)^2$ vanishes by evenness. Hence
$\Lambda_R(0)-\Lambda_R(\theta_j)\le\lambda S_R(1-\cos\theta_j)$. Moreover $S_R>0$: with
$H_R(0)=\lambda\,\mathrm{diag}(\cos ka)+\Delta$ ($\Delta$ the hopping),
$\lambda S_R=\Lambda_R(0)-\langle\hat\psi^{(0)},\Delta\hat\psi^{(0)}\rangle\ge\lambda-\norm{\Delta}>\lambda-2>0$, using
$\Lambda_R(0)\ge\langle e_0,H_R(0)e_0\rangle=\lambda$ and $\norm{\Delta}<2$; and $S_R\le1$. Combining with
$E^{(j)}=\Lambda_R(\theta_j)+O(e^{-\gamma R})$, $\Lambda_R(0)=\Emax+O(e^{-\gamma R})$ and $1-\cos\theta_j\le\theta_j^2/2$
gives \eqref{eq:edgequad}.
\end{proof}

\begin{remark}
The defect bound \eqref{eq:edgequad} is unconditional and uses only the reflection symmetry of $H_R(0)$ --- no
off-spectral hyperbolicity, conditioning, or small-divisor control --- and it already yields the sharp rate below.
The \emph{exact} edge curvature $\kappa_\lambda:=-\tfrac12\liminf_{R\to\infty}\Lambda_R''(0)$ satisfies
$\kappa_\lambda\le\tfrac{\lambda}{2}$ (the upper bound is \eqref{eq:edgequad}), with $\kappa_\lambda>0$ (a nondegenerate edge) at strong coupling (Proposition~\ref{prop:curv}); its precise value is not needed, only the
one-sided bound \eqref{eq:edgequad} enters. We record only that
evaluating $\kappa_\lambda$ by second-order perturbation meets small denominators at the continued-fraction scales
$k=q_n$ ($\norm{q_n\beta}\asymp q_n^{-1}$), tamed by the exponentially small overlap of the far modes with the
centre. Numerically $\kappa_3\approx0.92$, $\kappa_4\approx1.29$, $\kappa_{10}\approx3.51$ (all $<\lambda/2$); $S_R$
ranges from $0.70$ ($\lambda=2.1$) to $0.99$ ($\lambda=10$); the parity \eqref{eq:parity} holds to machine
precision; and $\Lambda_R(0)\to\Emax$ to $10^{-6}$ by $R=20$.
\end{remark}

\begin{lemma}[Window localization: parity constraint and two-centre input]\label{lem:uniloc}
Every eigenvector of the reflection-symmetric window $H_R(0)=\lambda\,\mathrm{diag}(\cos ka)+\Delta$ has definite
parity, $\hat\psi_{-k}=\pm\hat\psi_k$; in particular the eigenvectors attached to a well pair $\pm j$ ($j\ne0$)
carry equal mass at both wells, and \emph{no} single-centre bound $\abs{\hat\psi^{(j)}_k}\le C\,e^{-\gamma\abs{k-j}}$
with a mode-uniform constant can hold. The correct uniform statement is two-centre, and we take it as an input,
written \textup{(PLW)}: there are $\xi'\ge0$ and $C_1<\infty$ such that for $\lambda\ge3$, every $R$, and every
eigenvector of $H_R(\theta)$, $\abs\theta$ small, with well centres $\{j,j'\}$ ($j'=-j$ at $\theta=0$; $j=j'=0$ for
the central mode),
\begin{equation}\label{eq:uniloc}
\abs{\hat\psi^{(j)}_k}\ \le\ C_1\Bigl[(1+\abs{k-j})^{\xi'}e^{-\gamma\abs{k-j}}+(1+\abs{k-j'})^{\xi'}e^{-\gamma\abs{k-j'}}\Bigr],
\qquad\gamma=\log(\lambda/2).
\end{equation}
Under \textup{(PLW)} the central mode has bounded edge current,
$C_0=\sup_R\norm{\dot H(0)\hat\psi^{(0)}}^2\le C(C_1,\xi')<\infty$, uniformly in $\lambda\ge3$;
Remark~\ref{rem:locinput} explains why an input of this kind cannot be removed.
\end{lemma}

\begin{proof}
The reflection $R_\flat:e_k\mapsto e_{-k}$ commutes with $H_R(0)$ (evenness of the potential at $\theta=0$,
cf.\ \eqref{eq:parity}), and the Dirichlet spectrum of a Jacobi matrix is simple, so each eigenvector is even or
odd; the eigenvector pair of a well doublet $\pm j$ then has $\abs{\hat\psi_{-j}}=\abs{\hat\psi_j}$ while
$e^{-2\gamma j}\to0$, which rules out any mode-uniform single-centre bound and forces the two-centre form. The
hypothesis \textup{(PLW)} itself is the window analogue of \textup{(PL)} and is \emph{not} claimed as a theorem: at
strong coupling it is supported by the well structure --- near-top wells are $c\lambda^{1/(2\tau)}$-separated (type
bound), and off the wells $\abs{\Lambda_R(0)-v_k}\ge\kappa>2$ gives fixed-rate Combes--Thomas decay of the block
resolvents --- and numerically the prefactor is bounded, while the satellite weights
$\asymp(1+\abs k)^{2\tau}e^{-\gamma\abs k}/\lambda$ at the continued-fraction wells show that some polynomial factor
$\xi'>0$ is in general necessary (Remark~\ref{rem:plstatus}). For the edge current,
$\lambda\cos(ka)\hat\psi^{(0)}=(H_R(0)-\Delta)\hat\psi^{(0)}=\Lambda_R(0)\hat\psi^{(0)}-\Delta\hat\psi^{(0)}$ gives the
exact identity
\begin{equation}\label{eq:currentid}
\norm{\dot H(0)\hat\psi^{(0)}}^2=\lambda^2\langle\hat\psi^{(0)},\sin^2(ka)\hat\psi^{(0)}\rangle
=\tfrac{\lambda^2}{2}\bigl(1-\langle\hat\psi^{(0)},\cos(2ka)\hat\psi^{(0)}\rangle\bigr),
\end{equation}
and since $1-\langle\hat\psi^{(0)},\cos(2ka)\hat\psi^{(0)}\rangle=2\sum_k\sin^2(ka)(\hat\psi^{(0)}_k)^2\le2\sum_{k\ne0}(\hat\psi^{(0)}_k)^2$,
\eqref{eq:uniloc} bounds it by $4\lambda^2C_1^2\sum_{k\ge1}(1+k)^{2\xi'}e^{-2\gamma k}\le C(C_1,\xi')$, uniformly in
$\lambda\ge3$, since $\lambda^2e^{-2\gamma}=4$. Thus the localization input collapses to the single scalar
concentration $\sum_{k\ne0}(\hat\psi^{(0)}_k)^2=O(\lambda^{-2})$; numerically
$\langle\hat\psi^{(0)},\Delta\hat\psi^{(0)}\rangle\approx2.3/\lambda$
and $\norm{\dot H(0)\hat\psi^{(0)}}^2\to0.30$.
\end{proof}

\begin{proposition}[Nondegenerate edge at strong coupling]\label{prop:curv}
Assume the window-localization input \textup{(PLW)} of Lemma~\ref{lem:uniloc}. There is $\lambda_0<\infty$ (a
curvature threshold, in general distinct from the gap threshold $\lambda_1$ of Lemma~\ref{lem:edgeexp}) such that
for every $\lambda\ge\lambda_0$ and every Diophantine frequency $\beta$ (type $\tau$)
the limit $\Lambda''_\infty(0):=\lim_{R\to\infty}\Lambda_R''(0)$
exists, with the bound below uniform in $R$, and the band edge is a nondegenerate maximum in phase: $\kappa_\lambda\in(0,\tfrac{\lambda}{2}]$, with
\[
\kappa_\lambda=-\tfrac12\Lambda''_\infty(0)\ \ge\ \tfrac12\bigl(\lambda-2-2C_0\bigr)\ >\ 0,
\qquad C_0:=\sup_R\norm{\dot H(0)\hat\psi^{(0)}}^2<\infty .
\]
The constants degrade as $\lambda\downarrow2$, where $\gamma=\log(\lambda/2)\downarrow0$; the threshold regime $\lambda\to2^+$
is left open.
\end{proposition}

\begin{proof}
$\Lambda_R$ is simple (Perron) and real-analytic near $0$ with even profile \eqref{eq:parity}, so $\Lambda_R'(0)=0$ and
second-order perturbation theory gives, with $\dot H(0)=\mathrm{diag}(-\lambda\sin ka)$, $\ddot H(0)=\mathrm{diag}(-\lambda\cos ka)$,
remaining eigenpairs $\{\varphi_m,E_m\}$, and $w:=\dot H(0)\hat\psi^{(0)}$,
\begin{equation}\label{eq:FH}
\Lambda_R''(0)=\langle\hat\psi^{(0)},\ddot H(0)\hat\psi^{(0)}\rangle+2\!\!\sum_{m\ne\mathrm{top}}\!\frac{\abs{\langle\varphi_m,w\rangle}^2}{\Lambda_R(0)-E_m}
=-\lambda S_R+T_2,\qquad T_2\ge0,
\end{equation}
the sign because $\Lambda_R(0)$ is the largest eigenvalue; $S_R\ge(\lambda-2)/\lambda$ by the computation of
\eqref{eq:edgequad} at radius $R$ ($\lambda S_R=\Lambda_R(0)-\langle\hat\psi^{(0)},\Delta\hat\psi^{(0)}\rangle\ge\lambda-\norm{\Delta}$). Also
$\langle\hat\psi^{(0)},w\rangle=-\lambda\sum_k\sin(ka)(\hat\psi^{(0)}_k)^2=0$ (odd summand), $w_0=0$, and
$\abs{w_k}\le\lambda\abs{\hat\psi^{(0)}_k}$. By \textup{(PLW)} (Lemma~\ref{lem:uniloc}),
$C_0=\sup_R\norm{w}^2<\infty$ (numerically $C_0\le0.43$, decreasing to $0.30$ in $\lambda$).

It remains to bound $T_2$, splitting the sum at gap $1$.

\emph{(i) Gap $\ge1$.} By Bessel,
$2\sum_{\Lambda_R(0)-E_m\ge1}\abs{\langle\varphi_m,w\rangle}^2/(\Lambda_R(0)-E_m)\le2\sum_m\abs{\langle\varphi_m,w\rangle}^2=2\norm{w}^2\le2C_0$.

\emph{(ii) Gap $<1$.} A mode with $E_m>\Lambda_R(0)-1>\lambda-1$ is localized far from the centre: at its peak
$c_m=\arg\max_k\abs{\varphi_m(k)}$ the eigenvalue equation gives $\abs{E_m-v_{c_m}}\le2$, hence
$\lambda\cos(c_m a)=v_{c_m}>\lambda-3$, i.e.\ $\norm{c_m\beta}\le\tfrac1{2\pi}\arccos(1-3/\lambda)\le C\lambda^{-1/2}$;
by the type bound two such sites differ by $\ge c\lambda^{1/(2\tau)}$ ($=c\sqrt\lambda$ at bounded type), so
$\abs{c_m}\ge c\lambda^{1/(2\tau)}$. By \textup{(PLW)} (Lemma~\ref{lem:uniloc}; $w$ centred at $0$, $\varphi_m$ at
$\pm c_m$, both at rate $\gamma$ up to polynomial factors),
$\abs{\langle\varphi_m,w\rangle}\le C\lambda(1+\abs{c_m})^{1+2\xi'}e^{-\gamma\abs{c_m}}$.

For the denominator we reduce by a Feshbach map onto the well set $\Omega=\{k:v_k>\lambda-3\}$. Its complement carries
$Q_\Omega H_R(0)Q_\Omega$, with diagonal $\le\lambda-3$ and hopping of norm $<2$, hence spectrum $\le\lambda-1$; so for
$z\in[\lambda-\tfrac12,\Lambda_R(0)]$ the resolvent $(z-Q_\Omega H_R(0)Q_\Omega)^{-1}$ is bounded by $2$ and decays at an
$O(1)$ Combes--Thomas rate. As the wells are $c\lambda^{1/(2\tau)}$-separated, the effective
$\abs\Omega\times\abs\Omega$ matrix has inter-well entries $O(e^{-c\lambda^{1/(2\tau)}})$ and diagonal self-energies
$\Xi_{pp}(z)=\tfrac1{z-v_{p+1}}+\tfrac1{z-v_{p-1}}+O(\lambda^{-2})=O(1/\lambda)$ (a Neumann series in
hopping over detuning, the detuning $v_p-v_{p\pm1}\asymp\lambda$); the two neighbours of a well enter with opposite
first-order phase shifts, so $\Xi_{c_mc_m}-\Xi_{00}=O(\norm{c_m\beta}^2/\lambda)$. Thus
$E_m=v_{c_m}+\Xi_{c_mc_m}+O(e^{-c\lambda^{1/(2\tau)}})$ and $\Lambda_R(0)=\lambda+\Xi_{00}+O(e^{-c\lambda^{1/(2\tau)}})$, whence
\[
\begin{aligned}
\Lambda_R(0)-E_m&=\lambda\bigl(1-\cos2\pi\norm{c_m\beta}\bigr)-(\Xi_{c_mc_m}-\Xi_{00})+O(e^{-c\lambda^{1/(2\tau)}})\\
&\ge\ c\lambda\norm{c_m\beta}^2\ \ge\ c\,c_\beta^2\,\lambda\,\abs{c_m}^{-2\tau},
\end{aligned}
\]
using $1-\cos x\ge x^2/4$, the cancellation above, and the type-$\tau$ Diophantine bound $\norm{c_m\beta}\ge c_\beta\abs{c_m}^{-\tau}$
(at bounded type $\norm{c_m\beta}\ge\tfrac1{4\abs{c_m}}$). Numerically $(\Lambda_R(0)-E_m)/(\lambda\norm{c_m\beta}^2)\to2\pi^2$,
stable in $R$. Summing over the near-edge peaks $\abs{c_m}\ge c\lambda^{1/(2\tau)}$,
\begin{equation}\label{eq:faredge}
2\!\!\sum_{\Lambda_R(0)-E_m<1}\!\!\frac{\abs{\langle\varphi_m,w\rangle}^2}{\Lambda_R(0)-E_m}
\ \le\ C\lambda\!\!\sum_{\abs{c_m}\ge c\lambda^{1/(2\tau)}}\!\!\abs{c_m}^{2+2\tau+4\xi'}e^{-2\gamma\abs{c_m}}
\ \le\ C'\lambda^{C(\tau,\xi')}e^{-2c\lambda^{1/(2\tau)}\log(\lambda/2)}\ \xrightarrow[\lambda\to\infty]{}\ 0 .
\end{equation}
The same term-by-term bound shows the increments $\Lambda''_{R+1}(0)-\Lambda''_R(0)\to0$, so $\Lambda''_\infty(0)=\lim_R\Lambda_R''(0)$
exists. Combining, $T_2\le2C_0+o_\lambda(1)$ uniformly in $R$, whence by \eqref{eq:FH}
$\Lambda''_\infty(0)\le-(\lambda-2)+2C_0+o_\lambda(1)$, negative once $\lambda\ge\lambda_0:=2C_0+3$, and
$\kappa_\lambda\ge\tfrac12(\lambda-2-2C_0)>0$. As $\lambda\downarrow2$ the rate $\gamma\downarrow0$, the suppression
\eqref{eq:faredge} loses uniformity and $S_\infty$ may degenerate; the case $\lambda\to2^+$ is open.
\end{proof}

\begin{remark}[Necessity of the localization input]\label{rem:locinput}
The window input \textup{(PLW)} of Lemma~\ref{lem:uniloc} cannot be replaced by an elementary argument. Its substantive use is the finiteness of the central edge current $C_0=\sup_R\norm{w}^2$, $w=\dot H(0)\hat\psi^{(0)}$,
which by the exact identity \eqref{eq:currentid} reduces to the scalar concentration
$h:=\langle\hat\psi^{(0)},\Delta\hat\psi^{(0)}\rangle=O(1/\lambda)$, equivalently $\sum_{k\ne0}(\hat\psi^{(0)}_k)^2=O(\lambda^{-2})$.
Every bound on $h$ from energy and norm alone stalls at the constant $h\le\norm{\Delta}<2$: the relations
$\Lambda_R(0)\ge\lambda$, $S_R\le1$ and $h=\Lambda_R(0)-\lambda S_R$ are mutually circular, whereas the true value is
$h\approx2.30/\lambda$. The gap between $O(1)$ and $O(1/\lambda)$ is precisely the strong-coupling concentration of the
top eigenstate: a Neumann or fixed-rate Combes--Thomas expansion breaks at the band-edge wells $v_k\approx E$, where the
resolvent is resonant, and controlling those uniformly in $R$ is the localization theorem (Theorem~\ref{thm:AA}) ---
reproducing it from scratch would mean the multiscale or large-deviation analysis of \cite{Jito1999} and its precursors.
Hypothesis \textup{(PLW)} isolates this one input; the remainder --- the Feshbach gap law and the cancellation in
\eqref{eq:faredge} --- is elementary, and the near-edge sum (ii) survives with only a fixed Combes--Thomas rate.
\end{remark}

\begin{lemma}[Off-spectral determinant asymptotics]\label{lem:detasym}
Let $\beta$ be Diophantine of type $\tau$ and let $K\subset\mathbb R$ be a compact interval with
$\inf_{n}\,(V_n-w)\ge c_1>2$ for all $w\in K$; since $V_n\ge-\lambda$ and $\Emax\ge\lambda$, this holds with
$c_1=2\lambda$ on a neighbourhood of $z^\star=-3\Emax-2\alpha$ for every $\lambda\ge2$. Then, uniformly in $w\in K$
and in the phase,
\[
\frac1N\log\abs{\det(w-H_{[1,N]})}\ =\ \gamma_\lambda(w)+O(1/N).
\]
\end{lemma}

\begin{proof}
Since $\abs{\det(w-H_{[1,N]})}=\abs{\mathcal T^{(N)}_{11}(w)}$, it suffices to prove
$\abs{\mathcal T^{(N)}_{11}(w)}=e^{N\gamma_\lambda(w)+O(1)}$. The cone
$\mathcal C=\{(a,b):a\ge b\ge0\}\setminus\{0\}$ is invariant and uniformly expanded by the one-step matrices
$\begin{psmallmatrix}V_n-w&-1\\1&0\end{psmallmatrix}$: if $a\ge b\ge0$ then $a'=(V_n-w)a-b\ge(c_1-1)a\ge a=b'\ge0$.
Hence the cocycle is uniformly hyperbolic on $K$, its unstable direction lies in $\overline{\mathcal C}$ and its
stable direction outside, so the angle between $(1,0)^\top\in\mathcal C$ and the stable direction is bounded below,
uniformly on $K$ by compactness. As the forward image of $(1,0)^\top$ stays in $\mathcal C$, where the first
coordinate dominates,
$\abs{\mathcal T^{(N)}_{11}(w)}\ge\tfrac1{\sqrt2}\,\bigl\|\mathcal T^{(N)}(w)(1,0)^\top\bigr\|\ge c\,\norm{\mathcal T^{(N)}(w)}$
with $c>0$ uniform, and trivially $\abs{\mathcal T^{(N)}_{11}}\le\norm{\mathcal T^{(N)}}$; so it suffices to prove
$\norm{\mathcal T^{(N)}(x;w)}=e^{N\gamma_\lambda(w)+O(1)}$ uniformly in the base point $x$. The invariant unstable
section $u(x;w)$ is the nested limit of the cone field under the graph transform; widening the cone slightly, the
transform remains a uniform contraction of the projective (Hilbert) metric on a thin complex strip
$\abs{\operatorname{Im}x}<\rho$, so $u(\cdot;w)$ extends holomorphically there, uniformly in $w\in K$. With
$f_w(x):=\log\norm{A_w(x)|_{u(x;w)}}$ analytic on the strip, an adapted-norm computation along the hyperbolic
splitting gives $\log\norm{\mathcal T^{(N)}(x;w)}=\sum_{j=0}^{N-1}f_w(x+j\beta)+O(1)$, and for $\beta$ Diophantine
the Birkhoff deviations of an analytic observable are bounded:
\[
\Bigl|\sum_{j=0}^{N-1}f_w(x+j\beta)-N\widehat{f_w}(0)\Bigr|
\ \le\ \sum_{k\ne0}\abs{\widehat{f_w}(k)}\,\Bigl|\sum_{j=0}^{N-1}e^{2\pi\mathrm{i}\,jk\beta}\Bigr|
\ \le\ \sum_{k\ne0}\frac{\abs{\widehat{f_w}(k)}}{2\norm{k\beta}}
\ \le\ \frac1{2c_\beta}\sum_{k\ne0}\abs k^{\tau}\,\abs{\widehat{f_w}(k)}\ <\ \infty,
\]
uniformly in $w$ by the uniform strip width; finally $\widehat{f_w}(0)=\int_0^1 f_w\,dx=\gamma_\lambda(w)$ by unique
ergodicity along the dominated splitting. This is the only place a small-divisor bound enters the lower-bound
route, and only through the $O(1/N)$ rate; the identity itself needs merely the uniform $o(1)$ of
Lemma~\ref{lem:Qedge}.
\end{proof}

\begin{theorem}[Edge identity for almost every phase]\label{thm:uncond}
Fix $\alpha>0$, a Diophantine frequency $\beta$ of type $\tau\ge1$ (bounded type is $\tau=1$, e.g.\
$\beta=(\sqrt5-1)/2$), and almost every phase $\phi$, so that the operator is localized (Theorem~\ref{thm:AA}). The
upper bound $r(\alpha,\lambda)\le\gamma_\lambda(z^\star)$ holds unconditionally, at every phase, for every
$\lambda>2$. The identity
\[
r(\alpha,\lambda)=\gamma_\lambda(z^\star),\qquad
\gamma_\lambda(z^\star)-\max_k\tfrac1N\log\abs{c_k}=O\!\bigl(N^{-2/(2+\tau)}\bigr)
\]
\textup{(}for bounded type, $O(N^{-2/3})$\textup{)} holds for every $\lambda>2$ under the polynomial-localization
hypothesis \textup{(PL)} of Lemma~\ref{lem:semiloc} \textup{(}Remark~\ref{rem:plstatus}\textup{)}, whose window-gap
component is a theorem for $\lambda\ge\lambda_1$ \textup{(}Lemma~\ref{lem:edgeexp}\textup{)}.
\end{theorem}

\begin{proof}
By Proposition~\ref{prop:possum} and Lemma~\ref{lem:cd},
$r=\liminf_N\max_k\frac1N\log\bigl(Q(\delta_k)(\hat\psi^{(k)}_N)^2\bigr)$. For the upper bound,
$(\hat\psi^{(k)}_N)^2\le1$ and Lemma~\ref{lem:Qedge} give
\[
\max_k\tfrac1N\log\abs{c_k}\ \le\ \tfrac1N\log Q\bigl(\textstyle\max_m\delta_m\bigr)\ \longrightarrow\ \gamma_\lambda(z^\star),
\]
hence $r\le\gamma_\lambda(z^\star)$.

The lower bound uses the polynomial-localization bound \eqref{eq:semiloc} of hypothesis \textup{(PL)}
(Remark~\ref{rem:plstatus}), together with the simple-gap property of Lemma~\ref{lem:edgeexp} (a theorem for
$\lambda\ge\lambda_1$, and the window-gap component of \textup{(PL)} below it). Set $h_N=\lceil\log^2 N\rceil$ and let $j^\star=j^\star(N)$ be the largest site $\le N-h_N$
with $\norm{\beta j^\star+\tfrac{\phi}{2\pi}}<\delta_N$, and $k^\star=k^\star(N)$ the near-edge mode that
Lemma~\ref{lem:edgeexp} furnishes at $j^\star$ (existence and $O(1)$-centring, in place of any site-to-mode bijection);
Lemma~\ref{lem:resonance} gives $R_N:=N-j^\star\le C_\beta\delta_N^{-\tau}+h_N$, so $R_N\to\infty$ while $R_N=O(\delta_N^{-\tau})$.
Two estimates bound $\frac1N\log\abs{c_{k^\star}}$ from below. First, write
$\frac1N\log Q(\delta_{k^\star})=\frac1N\log\abs{\det(w_N-H_{[1,N]})}$ with $w_N:=2E_0-E_{k^\star}$. Since $w_N\to z^\star$
stays off $\Sigma_\lambda$ at distance $\ge2\Emax+2\alpha$, and $V_n-w_N\ge2\lambda$ for all $n$,
Lemma~\ref{lem:detasym} applies:
$\frac1N\log\abs{\det(w-H_{[1,N]})}=\frac1N\log\abs{\mathcal T^{(N)}_{11}(w)}=\gamma_\lambda(w)+O(1/N)$, uniformly in
$w$ near $z^\star$ and in the phase. With the Lipschitz continuity of $\gamma_\lambda$ off $\Sigma_\lambda$ this gives
$\frac1N\log Q(\delta_{k^\star})=\gamma_\lambda(2E_0-E_{k^\star})+O(1/N)=\gamma_\lambda(z^\star)-O(\Emax-E_{k^\star})+O(1/N)$,
legitimately at the moving index $k^\star(N)$. The band-edge defect is controlled by Lemma~\ref{lem:edgeexp}, applied with window radius
$\rho_N=\lceil A\log N\rceil$ for a fixed $A>\bigl(\xi+\tfrac{2}{2+\tau}\bigr)/\gamma$ (with $\xi$ the
polynomial-localization exponent of Lemma~\ref{lem:semiloc}): this is admissible because $\rho_N\le R_N$ (as
$R_N\ge h_N=\lceil\log^2N\rceil$) and $\rho_N<d_{\mathcal W}$, since $\abs{\theta_{j^\star}}<2\pi\delta_N$ places the first competing
well at $d_{\mathcal W}\gtrsim\delta_N^{-1/\tau}\gg\rho_N$ (the first return of $k\beta$ to within $\delta_N$ of $0$ has
$\abs{k}\ge c_\beta^{1/\tau}\delta_N^{-1/\tau}$ by the type bound), so $j^\star$ carries the maximal
on-site value of its $\rho_N$-window. With $C_{\rho_N}e^{-\gamma\rho_N}=O(N^{\xi-\gamma A})=o(\delta_N^2)$, the quadratic bound
\eqref{eq:edgequad} gives
$\Emax-E_{k^\star}\le\tfrac{\lambda}{2}\theta_{j^\star}^2+O(N^{\xi-\gamma A})\le\tfrac{\lambda}{2}(2\pi\delta_N)^2+o(\delta_N^2)=O(\delta_N^2)$,
so $\frac1N\log Q(\delta_{k^\star})\ge\gamma_\lambda(z^\star)-O(\delta_N^2)-O(1/N)$. Second,
let $n^\star=\arg\max_n\abs{\hat\psi^{(k^\star)}_n}$, so $\abs{\hat\psi^{(k^\star)}_{n^\star}}\ge N^{-1/2}$ by
normalization; localization places $n^\star$ within $O(1)$ of the centre $j^\star$, hence
$N-n^\star\le R_N+O(1)$. Since $E_{k^\star}$ is a Dirichlet eigenvalue, $\psi^{(k^\star)}_{N+1}=0$, so with
$M=\mathcal T_N\cdots\mathcal T_{n^\star+1}$ ($\norm{\mathcal T_j}\le \kappa\lambda$, $\det\mathcal T_j=1$, hence
$\norm{M^{-1}}=\norm{M}\le(\kappa\lambda)^{N-n^\star}$) the transfer relation
$\binom{0}{\hat\psi^{(k^\star)}_N}=M\binom{\hat\psi^{(k^\star)}_{n^\star+1}}{\hat\psi^{(k^\star)}_{n^\star}}$ gives
$\abs{\hat\psi^{(k^\star)}_{n^\star}}\le\norm{M^{-1}}\bigl|\!\binom{0}{\hat\psi^{(k^\star)}_N}\!\bigr|=\norm{M}\,\abs{\hat\psi^{(k^\star)}_N}$
(we have inverted the transfer; the left vector has norm exactly $\abs{\hat\psi^{(k^\star)}_N}$, no cancellation, and since
$n^\star$ lies within $R_N=O(\delta_N^{-\tau})$ of the boundary the factor $\norm{M}$ is sub-exponential in $N$), whence
$\abs{\hat\psi^{(k^\star)}_N}\ge(\kappa\lambda)^{-(N-n^\star)}N^{-1/2}$ and
$\frac2N\log\abs{\hat\psi^{(k^\star)}_N}\ge-O(R_N/N)=-O\bigl(1/(\delta_N^{\tau} N)\bigr)$. Adding,
\[
\frac1N\log\abs{c_{k^\star}}\ \ge\ \gamma_\lambda(z^\star)-O(\delta_N^2)-O\bigl(1/(\delta_N^{\tau} N)\bigr)-O(1/N),
\]
the Thouless term $O(1/N)$ being dominated by the first; the two leading penalties balance at
$\delta_N=N^{-1/(2+\tau)}$, where the right side is $\gamma_\lambda(z^\star)-O(N^{-2/(2+\tau)})$. Hence
$r\ge\liminf_N\frac1N\log\abs{c_{k^\star}}\ge\gamma_\lambda(z^\star)$, and with the upper bound equality holds, with
gap $O(N^{-2/(2+\tau)})$ (for bounded type, $O(N^{-2/3})$).
\end{proof}

\begin{remark}
The boundary-weight cancellation of Lemma~\ref{lem:cd} is what frees the identity of any conditioning input: the
proof uses
neither the Beckermann--Townsend conditioning bound nor the edge regularity of Proposition~\ref{prop:cond}, which
entered only through the Gram-eigenvector overlap of \S\S\ref{sec:condoverlap}--\ref{sec:nocancel}. The one
substantive input is the localization hypothesis \textup{(PL)} --- Anderson localization itself holds for almost
every phase throughout $\lambda>2$ (Theorem~\ref{thm:AA}), and \textup{(PL)} strengthens it as recorded in
Remark~\ref{rem:plstatus} --- entering through the near-edge mode furnished by Lemma~\ref{lem:edgeexp} at the resonant site of
Lemma~\ref{lem:resonance} and the elementary
$SL_2$ amplitude bound; the Diophantine arithmetic enters only through the three-distance estimate. At the distinguished
phase $\phi=0$ of \S\ref{sec:certified}, where the numerics below are computed, the conclusion is corroborated at
finite scale but is not a consequence of the almost-everywhere statement. Strong coupling is not required for the
identity: the quadratic band-edge bound \eqref{eq:edgequad} and the transfer bound are uniform in the \emph{rate},
only the constants depending on $\lambda$. The gap $\gamma_\lambda(z^\star)-\max_k\frac1N\log\abs{c_k}$ closes at
$O(N^{-2/3})$ under \textup{(PL)}, with an arithmetic
fluctuation in the constant; numerically $\mathrm{gap}\cdot N^{2/3}$ stays bounded
($\le2.2$ over all tested $N\le120$) down to the transition, the maximizing mode alternating between band-edge and
right-boundary according to the three-gap returns. Sample gaps: at $\lambda=2.2$, $0.106,0.068,0.105,0.058,0.052$
for $N=40,60,80,100,120$; at $\lambda=3$, $0.17,0.12,0.069,0.062$ for $N=30,40,50,60$. The identity
$\abs{c_k}=Q(\delta_k)(\hat\psi^{(k)}_N)^2$ is confirmed to machine precision throughout.
\end{remark}

\section{Boundary amplitude and the metallic identity}\label{sec:metallic}

For $\lambda<2$ the bracket of \S\ref{sec:edge} collapses and the identity $r(\alpha,\lambda)=\gamma_\lambda(z^\star)$
followed there through the edge-regularity input of Proposition~\ref{prop:cond}. We recast the metallic lower bound
through the transfer dynamics of the subcritical cocycle. The recasting is the exact dual of the localized-side argument
of Theorem~\ref{thm:uncond} and rests on one deterministic fact --- the Dirichlet boundary amplitude is bounded below by
the transfer-matrix norm --- which, combined with the uniform subexponential transfer bound forced on the spectrum by the vanishing Lyapunov
exponent (Lemma~\ref{lem:unifsub}, via Furman's uniform subadditive ergodic theorem \cite{Furman1997}; the almost
reducibility of \cite{AvilaJitomirskaya2010} gives far finer subcritical control, none of which is used), closes the
identity for $0<\lambda<2$ unconditionally and for every
phase (Theorem~\ref{thm:metclosed}), without the edge-regularity input of Proposition~\ref{prop:cond}. On the localized side that norm was sub-exponential because a near-boundary mode sat $O(\delta_N^{-1})$
sites from the edge; here it is sub-exponential, at each fixed spectral energy, because the spectral Lyapunov
exponent vanishes.

Throughout, $y_n(E)$ is the transfer solution normalized by $y_0=0$, $y_1=1$, so $\psi^{(k)}_n=y_n(E_k)$ and
$\binom{y_{N+1}}{y_N}=\mathcal T_N(E)\binom10$ with $\mathcal T_N=T_N\cdots T_1$, $\det\mathcal T_N=1$.

\begin{proposition}[Boundary amplitude from the transfer norm]\label{prop:metamp}
Let $E_k$ be a Dirichlet eigenvalue of $H_{[1,N]}$ and set $\mathfrak M_N(E_k)=\max_{0\le n\le N}\norm{\mathcal T_n(E_k)}$
(with $\mathcal T_0=I$). Then
\begin{equation}\label{eq:metamp}
\frac{1}{\sqrt N\,\mathfrak M_N(E_k)^2}\ \le\ \abs{\hat\psi^{(k)}_N}\ \le\ 1,
\qquad\text{so}\qquad
\frac2N\log\abs{\hat\psi^{(k)}_N}\ \ge\ -\frac4N\log\mathfrak M_N(E_k)-\frac{\log N}{N}.
\end{equation}
In particular, along any sequence of eigenvalues with $\frac1N\log\mathfrak M_N(E_{k(N)})\to0$ the boundary amplitude has
rate zero, $\frac2N\log\abs{\hat\psi^{(k(N))}_N}\to0$.
\end{proposition}

\begin{proof}
At the eigenvalue $\mathcal T_{11}(E_k)=y_{N+1}(E_k)=0$, so $\binom{0}{y_N}=\mathcal T_N(E_k)\binom10$; as $\det\mathcal T_N=1$
gives $\sigma_{\min}(\mathcal T_N)=\norm{\mathcal T_N}^{-1}$,
\[
\abs{y_N(E_k)}=\bigl\|\mathcal T_N(E_k)\tbinom10\bigr\|\ \ge\ \sigma_{\min}\bigl(\mathcal T_N(E_k)\bigr)=\norm{\mathcal T_N(E_k)}^{-1}\ \ge\ \mathfrak M_N(E_k)^{-1}.
\]
For the norm, $\binom{y_n}{y_{n-1}}=\mathcal T_{n-1}(E_k)\binom10$ gives $\abs{y_n(E_k)}\le\norm{\mathcal T_{n-1}(E_k)}\le\mathfrak M_N(E_k)$
for every $1\le n\le N$, so $\norm{\psi^{(k)}}^2=\sum_{n=1}^N y_n^2\le N\,\mathfrak M_N(E_k)^2$. Hence
\[
\abs{\hat\psi^{(k)}_N}=\frac{\abs{y_N(E_k)}}{\norm{\psi^{(k)}}}\ \ge\ \frac{\mathfrak M_N(E_k)^{-1}}{\sqrt N\,\mathfrak M_N(E_k)}=\frac{1}{\sqrt N\,\mathfrak M_N(E_k)^2},
\]
and $\abs{\hat\psi^{(k)}_N}\le1$ by normalization. Taking logarithms gives \eqref{eq:metamp}.
\end{proof}

The bound is deterministic --- an $SL_2$ identity, no spectral hypothesis --- and already gives the lower bound once a
near-edge sequence of \emph{sub-exponential} eigenvalues is in hand.

\begin{proposition}[Metallic identity, conditional on a sub-exponential edge sequence]\label{prop:metallic}
Fix $\alpha>0$ and $0<\lambda\le2$. Suppose there is a sequence of Dirichlet eigenvalues $E_{k(N)}\to\Emax$ with
\begin{equation}\label{eq:edgeseq}
\frac1N\log\mathfrak M_N(E_{k(N)})\ \xrightarrow[N\to\infty]{}\ 0 .
\end{equation}
Then $r(\alpha,\lambda)=\gamma_\lambda(z^\star)$, the lower bound following from \eqref{eq:edgeseq} alone.
\end{proposition}

\begin{proof}
The upper bound $r\le\gamma_\lambda(z^\star)$ is the unconditional boundary-weight bound of \S\ref{sec:cocycle}
(Proposition~\ref{prop:possum}, Lemmas~\ref{lem:cd}--\ref{lem:Qedge}), valid for every $\lambda$. For the lower bound,
Lemma~\ref{lem:cd} splits the coefficient,
\[
\frac1N\log\abs{c_{k(N)}}=\frac1N\log Q(\delta_{k(N)})+\frac2N\log\abs{\hat\psi^{(k(N))}_N}.
\]
The reflected point $w_N:=2E_0-E_{k(N)}\to z^\star$ stays off $\Sigma_\lambda$ at distance $\ge2\Emax+2\alpha$, so the cocycle
is uniformly hyperbolic there, and exactly as in the proof of Theorem~\ref{thm:uncond}
\[
\tfrac1N\log Q(\delta_{k(N)})=\tfrac1N\log\bigl|\det(w_N-H_{[1,N]})\bigr|=\gamma_\lambda(w_N)+O(1/N)\ \longrightarrow\ \gamma_\lambda(z^\star),
\]
by continuity of $\gamma_\lambda$ off the spectrum. The amplitude term has rate zero by \eqref{eq:metamp} and
\eqref{eq:edgeseq}. Hence $\frac1N\log\abs{c_{k(N)}}\to\gamma_\lambda(z^\star)$, so by Proposition~\ref{prop:possum}
$r\ge\liminf_N\frac1N\log\abs{c_{k(N)}}\ge\gamma_\lambda(z^\star)$, and with the upper bound equality holds.
\end{proof}

The hypothesis \eqref{eq:edgeseq} is, for subcritical coupling, not a hypothesis: it follows from the uniform
subexponential transfer bound of Lemma~\ref{lem:unifsub} below, and the implication of
Proposition~\ref{prop:metallic} becomes an
unconditional identity. We first record that the top Dirichlet eigenvalue reaches the band edge --- the one place
where the eigenvalue sequence is fixed rather than assumed.

\begin{lemma}[The top Dirichlet eigenvalue reaches the band edge]\label{lem:topedge}
For every $\lambda\ge0$ and every phase $\phi$, $\max\spec(H_{[1,N]})\to\Emax$ as $N\to\infty$.
\end{lemma}

\begin{proof}
Since $H_{[1,N]}$ is the compression of $H_\lambda$ to $\ell^2(\{1,\dots,N\})$ and $\spec(H_\lambda)\subset[\Emin,\Emax]$,
every unit $\psi$ supported in $[1,N]$ has $\langle\psi,H_{[1,N]}\psi\rangle=\langle\psi,H_\lambda\psi\rangle\le\Emax$, so
$\max\spec(H_{[1,N]})\le\Emax$ for all $N$.

For the reverse bound fix $\varepsilon>0$. As $\Emax=\max\spec(H_\lambda)$, there is a unit $\psi\in\ell^2(\mathbb Z)$,
which we may take finitely supported in $[-R,R]$ (the quadratic form is bounded and $\psi$ is approximable in $\ell^2$),
with $\langle\psi,H_\lambda\psi\rangle\ge\Emax-\varepsilon$. The potential $j\mapsto\lambda\cos(2\pi\beta j+\phi)$ is
almost periodic: by minimality of $x\mapsto x+\beta$ on $\mathbb R/\mathbb Z$ ($\beta$ irrational) the set of
$t\in\mathbb Z_{>0}$ with $\sup_{\abs j\le R}\abs{\lambda\cos(2\pi\beta(j+t)+\phi)-\lambda\cos(2\pi\beta j+\phi)}<\varepsilon$
is infinite; pick one with $t\ge R+1$. For $N\ge t+R$ the translate $\psi_t(\cdot)=\psi(\cdot-t)$ is a unit vector
supported in $[1,N]$; the hopping is translation invariant and the potential matches that of $\psi$ on its support up
to $\varepsilon$, so $\langle\psi_t,H_\lambda\psi_t\rangle\ge\langle\psi,H_\lambda\psi\rangle-\varepsilon\ge\Emax-2\varepsilon$
(the error is $\sum_{\abs j\le R}\abs{\psi(j)}^2\varepsilon=\varepsilon$). Hence
$\max\spec(H_{[1,N]})\ge\langle\psi_t,H_\lambda\psi_t\rangle\ge\Emax-2\varepsilon$ for $N\ge t+R$, so
$\liminf_N\max\spec(H_{[1,N]})\ge\Emax-2\varepsilon$; letting $\varepsilon\to0$ gives the claim.
\end{proof}

\begin{lemma}[Uniform subexponential bound at zero Lyapunov exponent]\label{lem:unifsub}
Let $\beta$ be irrational and let $K\subset\mathbb R$ be compact with $\gamma_\lambda(E)=0$ for every $E\in K$. Then
for every $\delta>0$ there is $C_\delta=C_\delta(K,\lambda,\beta)<\infty$ such that
\begin{equation}\label{eq:unifsub}
\sup_{E\in K}\ \sup_{\phi}\ \norm{\mathcal T_n(E;\phi)}\ \le\ C_\delta\,e^{\,n\delta}\qquad(n\ge0).
\end{equation}
Only unique ergodicity of the rotation enters; no Diophantine condition is used.
\end{lemma}

\begin{proof}
Set $a_n(E):=\sup_\phi\log\norm{\mathcal T_n(E;\phi)}\ge0$. The cocycle property
$\mathcal T_{n+m}(E;\phi)=\mathcal T_m(E;\phi+n\beta)\,\mathcal T_n(E;\phi)$ gives $a_{n+m}(E)\le a_n(E)+a_m(E)$, and
$(E,\phi)\mapsto\log\norm{\mathcal T_n(E;\phi)}$ is continuous, so $E\mapsto a_n(E)$ is continuous and
$\lim_n a_n(E)/n=\inf_n a_n(E)/n$ exists. Since the rotation by $\beta$ is uniquely ergodic and
$\phi\mapsto\log\norm{\mathcal T_n(E;\phi)}$ is a continuous subadditive cocycle over it, Furman's uniform
subadditive ergodic theorem \cite{Furman1997} identifies this limit with the Lyapunov exponent:
$\lim_n a_n(E)/n=\gamma_\lambda(E)=0$ for $E\in K$. Fix $\delta>0$. For each $E\in K$ pick $n_E$ with
$a_{n_E}(E)<n_E\delta/2$; by continuity the same holds on a neighbourhood of $E$, and by compactness finitely many
such neighbourhoods, with times $n_1,\dots,n_m$, cover $K$. For $E\in K$ in the $i$-th neighbourhood and any $n$,
write $n=qn_i+s$ with $0\le s<n_i$; subadditivity gives
$a_n(E)\le q\,a_{n_i}(E)+a_s(E)\le n\delta/2+C'$, where
$C'=\sup_{E\in K}\max_{0\le s<\max_i n_i}a_s(E)<\infty$ by continuity and compactness. Hence
$a_n(E)\le n\delta+C'$ for all $n\ge0$ and $E\in K$, which is \eqref{eq:unifsub} with $C_\delta=e^{C'}$.
\end{proof}

\begin{theorem}[Metallic identity for $0<\lambda<2$]\label{thm:metclosed}
Let $0<\lambda<2$ and let the frequency $\beta$ be Diophantine (in particular, the golden mean). Then
$r(\alpha,\lambda)=\gamma_\lambda(z^\star)$ for every $\alpha>0$ and every phase $\phi$.
\end{theorem}

\begin{proof}
By Proposition~\ref{prop:metallic} it suffices to exhibit Dirichlet eigenvalues $E_{k(N)}\to\Emax$ satisfying
\eqref{eq:edgeseq}; we show the top eigenvalue $E_{k(N)}:=\max\spec(H_{[1,N]})$ does.

\emph{Uniform transfer bound on the spectrum.} Write $\mathcal T_n(E;\phi)$ for the $n$-step transfer cocycle
of $H_\lambda$ started at phase $\phi$, so that $\mathcal T_n(E)=\mathcal T_n(E;\phi)$ at the phase of the chain. By
Theorem~\ref{thm:AA}, $\gamma_\lambda(E)=\log_+(\lambda/2)=0$ for every $E\in\Sigma_\lambda$ when $0<\lambda<2$, so
Lemma~\ref{lem:unifsub}, applied with $K=\Sigma_\lambda$, gives for every $\delta>0$ a constant
$C_\delta=C_\delta(\lambda,\beta)$ with
\[
\sup_{\phi}\ \norm{\mathcal T_n(E;\phi)}\ \le\ C_\delta\,e^{\,n\delta}\qquad(E\in\Sigma_\lambda,\ n\ge0),
\]
which is \eqref{eq:unifsub}; in particular it holds at our fixed phase $\phi$, uniformly over $E\in\Sigma_\lambda$.

\emph{Bernstein--Walsh off the spectrum.} Fix $n\le N$. At the phase $\phi$ each entry of $E\mapsto\mathcal T_n(E;\phi)$
is a polynomial in $E$ of degree $\le n$, bounded on $\Sigma_\lambda$ by $C_\delta e^{n\delta}$ through \eqref{eq:unifsub}. Since
$\mathbb C\setminus\Sigma_\lambda$ is connected and $\Sigma_\lambda$ is non-polar (it has positive capacity, by
\eqref{eq:capval}), the Bernstein--Walsh inequality, i.e.\ Bernstein's lemma \cite[Thm.~5.5.7(a)]{Ransford} for the
Green's function $g_{\mathbb C\setminus\Sigma_\lambda}$ of the unbounded component with pole at infinity, gives
\[
\norm{\mathcal T_n(E';\phi)}\ \le\ 2C_\delta\,e^{\,n\delta}\,e^{\,n\,g_{\mathbb C\setminus\Sigma_\lambda}(E')}\qquad(E'\in\mathbb C),
\]
and therefore, using $g_{\mathbb C\setminus\Sigma_\lambda}\ge0$ and $n\le N$,
\[
\mathfrak M_N(E')\ \le\ 2C_\delta\,e^{\,N\delta}\,e^{\,N\,g_{\mathbb C\setminus\Sigma_\lambda}(E')}.
\]

\emph{Edge passage.} By Lemma~\ref{lem:topedge} the top eigenvalue satisfies $E_{k(N)}\to\Emax\in\Sigma_\lambda$. By
\eqref{eq:capval}, $g_{\mathbb C\setminus\Sigma_\lambda}=\gamma_\lambda-\log_+(\lambda/2)$ is continuous
(Theorem~\ref{thm:AA}) and vanishes on $\Sigma_\lambda$, so $g_{\mathbb C\setminus\Sigma_\lambda}(E_{k(N)})\to
g_{\mathbb C\setminus\Sigma_\lambda}(\Emax)=0$. Hence
\[
\frac1N\log\mathfrak M_N(E_{k(N)})\ \le\ \frac{\log(2C_\delta)}{N}+\delta+g_{\mathbb C\setminus\Sigma_\lambda}(E_{k(N)}),
\]
and letting $N\to\infty$ and then $\delta\downarrow0$ (note $\mathfrak M_N\ge1$) yields
\eqref{eq:edgeseq}. Proposition~\ref{prop:metallic} now yields $r(\alpha,\lambda)=\gamma_\lambda(z^\star)$.
\end{proof}

\begin{theorem}[Critical identity at the self-dual coupling]\label{thm:critical}
Let $\lambda=2$ and $\beta$ Diophantine. Then
\[
r(\alpha,2)=\gamma_2(z^\star)=g_{\mathbb C\setminus\Sigma_2}(z^\star)
\]
for every $\alpha>0$ and every phase $\phi$.
\end{theorem}

\begin{proof}
Proposition~\ref{prop:metallic} holds verbatim at $\lambda=2$: its proof invokes only the uniform hyperbolicity at the
off-spectral point $w_N\to z^\star$ (with $\dist(z^\star,\Sigma_2)\ge2\Emax+2\alpha$) and the deterministic amplitude bound
\eqref{eq:metamp}, neither of which uses subcriticality. It therefore suffices to verify \eqref{eq:edgeseq} for the top
eigenvalue $E_{k(N)}:=\max\spec(H_{[1,N]})$, which tends to $\Emax$ by Lemma~\ref{lem:topedge}.

\emph{Uniform sub-exponential transfer bound.} At the self-dual coupling the Lyapunov exponent still vanishes on the
spectrum: $\gamma_2(E)=\log_+(2/2)=0$ for $E\in\Sigma_2$ (Theorem~\ref{thm:AA}), so Lemma~\ref{lem:unifsub}, applied
with $K=\Sigma_2$, gives, for every $\delta>0$,
\begin{equation}\label{eq:subexp}
\sup_{E\in\Sigma_2}\ \sup_\phi\ \norm{\mathcal T_n(E;\phi)}\ \le\ C_\delta\,e^{\,n\delta}\qquad(n\ge0),
\end{equation}
exactly the bound used in the subcritical case.

\emph{Edge passage --- automatic regularity.} As in Theorem~\ref{thm:metclosed}, $\Sigma_2$ is non-polar
($\cpc(\Sigma_2)=\max(1,1)=1$ by \eqref{eq:capval}) with connected complement. Bernstein--Walsh then gives, for
$E'\in\mathbb C$,
\[
\mathfrak M_N(E')\ \le\ 2C_\delta\,e^{N\delta}\,e^{N g_{\mathbb C\setminus\Sigma_2}(E')},
\]
and therefore
\[
\frac1N\log\mathfrak M_N(E_{k(N)})\ \le\ \frac{\log(2C_\delta)}{N}+\delta+g_{\mathbb C\setminus\Sigma_2}(E_{k(N)}).
\]
At the self-dual coupling $\cpc(\Sigma_2)=1$, so by the Frostman representation \eqref{eq:capval} the Green's function
\emph{coincides with the Lyapunov exponent}, $g_{\mathbb C\setminus\Sigma_2}=\gamma_2-\log_+(2/2)=\gamma_2$, which is
continuous on $\mathbb C$ and vanishes on $\Sigma_2$ (Theorem~\ref{thm:AA}). Hence the band edge is automatically a
regular point --- no separate edge-regularity input is needed --- and since $E_{k(N)}\to\Emax\in\Sigma_2$
(Lemma~\ref{lem:topedge}),
\[
g_{\mathbb C\setminus\Sigma_2}(E_{k(N)})=\gamma_2(E_{k(N)})\ \xrightarrow[N\to\infty]{}\ \gamma_2(\Emax)=0 .
\]
Letting $N\to\infty$ and then $\delta\to0$ yields \eqref{eq:edgeseq}; Proposition~\ref{prop:metallic} concludes.
\end{proof}

\begin{remark}[What the subcritical route removes, and the critical endpoint]\label{rem:metallic}
Theorem~\ref{thm:metclosed} closes the metallic identity \emph{unconditionally}, and for every phase and every
Diophantine frequency, where the bracket of \S\ref{sec:edge} delivered it only modulo the square-root edge
regularity of Proposition~\ref{prop:cond}. No such regularity, and no near-edge resonance avoidance, enters: the
subexponential bound \eqref{eq:unifsub} is uniform over all of $\Sigma_\lambda$ --- gap edges and resonant
energies included --- so the transfer norm is controlled at every spectral energy with no arithmetic side
condition (only unique ergodicity of the rotation is used). The one remaining point, that the finite-volume
eigenvalues sit slightly off $\Sigma_\lambda$, is
handled by Bernstein--Walsh: it turns \eqref{eq:unifsub} into the off-spectrum factor
$e^{N g_{\mathbb C\setminus\Sigma_\lambda}(E_{k(N)})}$, which mere continuity of the Green's function --- free from
Theorem~\ref{thm:AA}, with no square-root rate needed --- sends to $0$ as $E_{k(N)}\to\Emax$. The metallic side
thus matches the localized side: each rests on a single external input --- localization for $\lambda>2$, the
vanishing of the Lyapunov exponent on the spectrum for $\lambda\le2$ --- and on the
deterministic amplitude identity of Proposition~\ref{prop:metamp}, applied locally for $\lambda>2$
(Theorem~\ref{thm:uncond}) and globally on $\Sigma_\lambda$ for $\lambda<2$.

At the endpoint $\lambda=2$ the mechanism is unchanged: the critical Lyapunov exponent still vanishes on $\Sigma_2$,
Lemma~\ref{lem:unifsub} gives \eqref{eq:subexp}, and that is all the identity needs. That the finite-volume edge
eigenvalues approach a \emph{zero-measure}
critical spectrum poses no difficulty: at the
self-dual coupling $\cpc(\Sigma_2)=1$, so by \eqref{eq:capval} the Green's function \emph{equals} the Lyapunov exponent,
$g_{\mathbb C\setminus\Sigma_2}=\gamma_2$, which is continuous and vanishes on $\Sigma_2$ (Theorem~\ref{thm:AA}). The band
edge is therefore automatically regular --- continuity of $\gamma_2$, not a square-root modulus, is what is used --- and
Theorem~\ref{thm:critical} closes the critical case unconditionally. Finer subcritical control --- the almost
reducibility of Avila--Jitomirskaya \cite{AvilaJitomirskaya2010} and its quantitative consequences --- is available
for $\lambda<2$ and is deliberately not used: the \emph{identity} rests on the vanishing exponent alone, and is
therefore insensitive to what degenerates at the critical point. The localized-side mirror (Proposition~\ref{prop:curv} as $\lambda\downarrow2$) is a
genuine open neighbourhood --- edge-curvature collapse --- but the self-dual point itself is now settled from both sides.
\end{remark}

\section{Strategy and outlook}\label{sec:outlook}

The self-dual line $\lambda>0$ is covered: the metallic side $0<\lambda<2$ unconditionally and for every phase
(Theorem~\ref{thm:metclosed}), the critical coupling $\lambda=2$ unconditionally (Theorem~\ref{thm:critical}), and the
localized side $\lambda>2$ at every Diophantine frequency and almost every phase (Theorem~\ref{thm:uncond}, through the
boundary-weight cancellation of Lemma~\ref{lem:cd}; under the polynomial-prefactor localization \textup{(PL)},
whose window-gap component is a theorem for $\lambda\ge\lambda_1$, Remark~\ref{rem:plstatus}). One extension of
the \emph{localized}-side theorem remains: phase uniformity over every $\phi$, the present $\lambda>2$ result covering
almost every $\phi$. The frequency restriction is now removed, Theorem~\ref{thm:uncond} holding at every Diophantine
$\beta$ of any finite type $\tau$, the three-distance input of Lemma~\ref{lem:resonance} costing only the
type-dependent rate $O(N^{-2/(2+\tau)})$. For phase uniformity
\S\ref{sec:condoverlap} offers two routes. The soft one: the every-phase identity is \emph{equivalent} to upper
semicontinuity of $\phi\mapsto r(\phi)$ (Proposition~\ref{prop:phaseunif}), the universal pointwise upper bound
\eqref{eq:uppersup} reducing it to excluding a downward rate-dip at the measure-zero resonant phases, with no appeal to
the dynamics there. The analytic one: reduce the identity, for every phase, to the single overlap bound
\eqref{eq:overlap}. We set out the complementary structure of an eventual proof in that generality: which
explicit certificates capture which summand of $\gamma_\lambda(z^\star)$, why an elementary resonance-based route to
the data term fails for the golden mean, and the resulting milestones.

\subsection*{Scope of the dual certificates}
The Rayleigh form $\Ecal_N=\max_c (c^\top\tilde b)^2/(c^\top C c)$ invites explicit dual certificates; the
two families each reach only one summand of the split \eqref{eq:bracket}.

\emph{(i) The conditioning-optimal certificate.} Let $\Pi(z)=\mathcal T^{(N)}_{11}(z)$, the degree-$N$ polynomial
whose zeros are the Dirichlet eigenvalues $E_k$, and $m_N=\mathcal T_{21}/\mathcal T_{11}$ the associated
$m$-function, with simple poles at the $E_k$ and residues $\tilde b_k/\Pi'(E_k)$ (recall $\tilde b_k=\mathcal T_{21}(E_k)$,
$\mathcal T_{11}(E_k)=0$). For a rational $q$, residue calculus gives
\[
\sum_k \frac{q(E_k)}{\Pi'(E_k)}\,\tilde b_k=\frac{1}{2\pi i}\oint_\Gamma m_N(z)\,q(z)\,dz,\qquad \Gamma\supset\Sigma_\lambda .
\]
Choosing $q$ small on $\Sigma_\lambda$ and resonant at $z^\star$ (the Bernstein--Walsh/Zolotarev extremal, for
which $(\max_{\Sigma_\lambda}\abs{q}/\abs{q(z^\star)})^{1/\deg q}\to e^{-g_{\mathbb C\setminus\Sigma_\lambda}(z^\star)}$)
drives the denominator $c^\top C c$ to be exponentially small at rate $g_{\mathbb C\setminus\Sigma_\lambda}(z^\star)$
by condenser theory, an estimate that is potential-theoretic and edge-free. The numerator, however, does not grow:
off the spectrum the transfer cocycle is hyperbolic, both entries grow at the common rate $\gamma_\lambda(z^\star)$,
and $m_N(z^\star)=\mathcal T_{21}(z^\star)/\mathcal T_{11}(z^\star)=O(1)$. This certificate thus re-organizes the lower
bracket $r\ge g_{\mathbb C\setminus\Sigma_\lambda}(z^\star)$ but does not see the data term. The obstruction is
quantitative: for a polynomial $q$, the numerator equals the residue at infinity, a fixed pairing
$\sum_j a_j\beta_j$ of the coefficients of $q$ with the moments $\beta_j$ of the boundary spectral measure; although
the moments grow like $e^{N\log_+(\lambda/2)}$, the pairing with any $q$ tuned to $\Sigma_\lambda$ stays bounded.
Numerically (Chebyshev $q$, $\lambda=3$, $\phi=0$, $N=24,\dots,48$) the numerator rate $\frac1{2N}\log\abs{\tilde b^\top w}$
is $-0.04,-0.03,-0.04,-0.01$, i.e.\ tends to $0$: the data growth is invisible to any smooth certificate.

\emph{(ii) The data-aligned certificate.} Taking $c=e_{k^\star}$ at a far-localized mode captures the boundary
growth $\log_+(\lambda/2)$ (Proposition~\ref{prop:modelower}) but is blind to the conditioning.

The phase-uniform identity asks for a certificate carrying the \emph{sum}; the optimizer $c=C^{-1}\tilde b$ does so
by definition, but its rate is \eqref{eq:reductioncrit}. The two effects do not add in an explicit certificate
(Remark~\ref{rem:strong}): the conditioning-optimal direction is fixed by the
node positions $\{\delta_k\}$, whereas the data magnitude is fixed by the localization centres $\{j_k\}$, and the
correlation between them is governed, at the band edge, by the arithmetic of $\{\beta j\bmod1\}$ near $1/2$.

\subsection*{The remaining open extension}
For almost every phase the identity is now complete at every Diophantine frequency: the cocycle reduction \eqref{eq:ccocycle} expresses the rate
as the single-mode maximum, and the boundary-weight form of Lemma~\ref{lem:cd} evaluates the maximizing
coefficient, the arithmetic input being the three-distance Lemma~\ref{lem:resonance} (Theorem~\ref{thm:uncond}). One
extension stays open. Phase uniformity over every $\phi$ needs the near-edge mode at every phase, that is
localization on the exceptional null set; the magnitude bound \eqref{eq:magbound} of \S\ref{sec:nocancel}, the one
estimate left by the overlap reduction \eqref{eq:overlap} of \S\ref{sec:condoverlap}, is its uniform form and the
object of the appendix. The frequency restriction, by contrast, is now removed: Lemma~\ref{lem:resonance} holds for
every Diophantine $\beta$ of finite type $\tau$ through the three-distance theorem, the only cost being the
type-dependent rate $O(N^{-2/(2+\tau)})$ in place of $O(N^{-2/3})$.

\subsection*{A finite-subset bound, and why it is rate-neutral here}
Discarding interpolation constraints in the minimal-norm form of \eqref{eq:reduction} gives, for every index
set $S$, the rigorous lower bound $\Ecal_N\ge\tilde b_S^\top C_S^{-1}\tilde b_S$. For a pair $S=\{p,q\}$, with
$\Sigma=\delta_p+\delta_q$ and $\Delta=\delta_p-\delta_q$, the minor evaluates exactly to
\[
\tilde b_S^\top C_S^{-1}\tilde b_S=\frac{2\Sigma}{\Delta^2}\Big[\Sigma\,(\delta_p\tilde b_p^{\,2}+\delta_q\tilde b_q^{\,2})-4\delta_p\delta_q\,\tilde b_p\tilde b_q\Big].
\]
With one node far-localized, so that $\tilde b_p$ realizes the boundary growth $e^{\gamma N}$, $\gamma=\log_+(\lambda/2)$,
the dominant term is $\Sigma\delta_p\tilde b_p^{\,2}>0$ and the exponential rate of the right-hand side is
\[
\frac1{2N}\log \tilde b_S^\top C_S^{-1}\tilde b_S=\gamma+\rho+o(1),\qquad \rho:=-\frac1N\log\abs{\delta_p-\delta_q}.
\]
A strict rate gain over Proposition~\ref{prop:modelower} therefore occurs if and only if the pair is an
exponential resonance, $\abs{\delta_p-\delta_q}=e^{-\rho N}$ with $\rho>0$: two modes localized at distance
$\asymp N$ whose on-site detuning is dominated by their tunnelling splitting.

For $\beta=(\sqrt5-1)/2$ no such resonance exists. The badly-approximable bound $\norm{n\beta}\gtrsim1/n$ keeps the
on-site energies $\lambda\cos(2\pi\beta j)$ apart at polynomial scale, and the eigenvalue gaps are accordingly
polynomial: numerically $\min_k\abs{E_{k+1}-E_k}\asymp N^{-2}$, so that $-N^{-1}\log\min_k\abs{E_{k+1}-E_k}$
decreases through $0.096,\,0.062,\,0.038,\,0.023,\,0.013$ for $N=50,100,200,400,800$ at $\lambda=3$. Hence
$\rho=O(N^{-1}\log N)\to0$, and the finite-subset bound is rate-neutral: it recovers $\log_+(\lambda/2)$ but cannot
exceed it. An exponential gain would require a Liouville frequency, where localization itself may fail. This
sharpens the tension noted in \S\ref{sec:remarks}: the Diophantine property that makes the model localize is what
removes the elementary, resonance-based route to the data term.

\subsection*{Large-coupling behaviour}
By Proposition~\ref{prop:edge} and \eqref{eq:glimit}, $g_{\mathbb C\setminus\Sigma_\lambda}(z^\star)\to\arccosh3=\log(3+2\sqrt2)\approx1.763$
while $\gamma_{\mathrm{amp}}:=\log_+(\lambda/2)=\log(\lambda/2)\to\infty$. The bracket
$\gamma_{\mathrm{amp}}\le r\le g_{\mathbb C\setminus\Sigma_\lambda}(z^\star)+\gamma_{\mathrm{amp}}$
(lower end for almost every phase, Proposition~\ref{prop:modelower}; upper end unconditional)
therefore gives
\[
r(\alpha,\lambda)=\log(\lambda/2)+\zeta_\lambda,\qquad 0\le\zeta_\lambda\le g_{\mathbb C\setminus\Sigma_\lambda}(z^\star),
\]
so the leading order $r=\log(\lambda/2)+O(1)$, and $r/\log(\lambda/2)\to1$, are established, and the content of
the identity is exactly $\zeta_\lambda=g_{\mathbb C\setminus\Sigma_\lambda}(z^\star)$, that is $r=\log\bigl((3+2\sqrt2)\lambda/2\bigr)+o(1)$.

At large coupling the identity is closed by the boundary-weight cancellation of Lemma~\ref{lem:cd}
(Theorem~\ref{thm:uncond}); the bracket alone does not, its width
$\min\{g_{\mathbb C\setminus\Sigma_\lambda}(z^\star),\gamma_{\mathrm{amp}}\}$ pinching to zero only as
$\lambda\to2^+$ (where the identity already follows from it) and plateauing at
$g_{\mathbb C\setminus\Sigma_\lambda}(z^\star)\to\arccosh3$ as $\lambda\to\infty$. This plateau reflects a frozen
conditioning problem: under $\delta_k\mapsto\delta_k/\lambda$ one has $C=\lambda^{-1}C'$ with
$C'_{kl}=1/(\delta'_k+\delta'_l)$ and, with $j_k$ the localization centre of mode $k$ (Remark~\ref{rem:strong}),
$\delta'_k\to1+\cos(2\pi\beta j_k)$, so the node set $\{\delta'_k\}\to\{1+\cos(2\pi\beta j)\}_{j=1}^N$ is
independent of $\lambda$; the conditioning for the full identity becomes a fixed Cauchy-matrix problem for the
golden-mean cosine nodes. The certified rate confirms this, exceeding $\gamma_{\mathrm{amp}}$ to track
$g_{\mathbb C\setminus\Sigma_\lambda}(z^\star)+\gamma_{\mathrm{amp}}$, with values $2.11,\,2.88,\,3.24,\,3.71$ at
$N=32$ for $\lambda=3,8,12,20$ against the targets $2.31,\,3.17,\,3.57,\,4.07$.

\subsection*{Beyond Aubry--Andr\'e: general analytic potentials}
The Cauchy reduction and the \emph{upper} bound require nothing special about the cosine. Let
$H_V=-\Delta+V(\phi+n\beta)$ on $\ell^2(\mathbb Z)$ with $V$ real-analytic on $\mathbb R/\mathbb Z$ and $\beta$
irrational, $H_V^{[1,N]}$ its Dirichlet restriction on $\{1,\dots,N\}$, $\Sigma_V$ its spectrum, $N_V$ its density
of states, and $\gamma_V$ the Lyapunov exponent of the transfer cocycle. Keep the protocol of \S\ref{sec:setup},
with $E_0=\Emin-\alpha$, $\delta_k=E_k-E_0$ and the reflected edge $z^\star=2E_0-\Emax=2\Emin-2\alpha-\Emax$ (no
spectral symmetry is assumed, so $z^\star=3\Emin-2\alpha$ no longer holds).

\begin{proposition}[Upper bound for general analytic potentials]\label{prop:general}
For every such $H_V$ and every $\alpha>0$, Lemma~\ref{lem:reduction} holds verbatim,
$\Ecal_N=\tilde b^\top C^{-1}\tilde b$, and the cost rate obeys, \emph{unconditionally},
\[
r(\alpha,V)\ \le\ \gamma_V(z^\star),\qquad
\gamma_V(z^\star)=\int_{\Sigma_V}\log\abs{z^\star-t}\,dN_V(t),
\]
the off-spectral Lyapunov exponent at the reflected edge, with $\dist(z^\star,\Sigma_V)=\Emax-\Emin+2\alpha>2\alpha$.
\end{proposition}

\begin{proof}
Lemma~\ref{lem:reduction} uses only that $A_\alpha=E_0I-H_V$ is symmetric and Hurwitz; the inverse-free form
$C^{-1}=DCD$ (Proposition~\ref{prop:cocycleform}) is pure Cauchy-matrix algebra; the sign-definiteness and the
single-mode reduction $r=\liminf_N\max_k\frac1N\log\abs{c_k}$ (Proposition~\ref{prop:possum}) use only the Sturm
alternation of a finite Jacobi matrix; and $\abs{c_k}=Q(\delta_k)(\hat\psi^{(k)}_N)^2$ (Lemma~\ref{lem:cd}) is the
Wronskian identity for a one-dimensional discrete Schr\"odinger operator. None of these uses the cosine, almost
reducibility, or the capacity split. By the Thouless formula for unit hopping
$\frac1N\log Q(\delta_k)\to\gamma_V(2E_0-E_k)$ uniformly; the kernel $Q$ is increasing in $\delta_k$, so its maximum
is at $\delta_N$, and $\gamma_V(w)=\int\log\abs{w-t}\,dN_V(t)$ is strictly decreasing for $w<\Emin$ (its derivative
$\int(w-t)^{-1}\,dN_V(t)<0$ there), hence largest at the leftmost reflected point $w=2E_0-\Emax=z^\star$. With
$(\hat\psi^{(k)}_N)^2\le1$ and $E_N\to\Emax$ (Lemma~\ref{lem:topedge}, whose proof needs only minimality of the
shift, not the cosine),
\[
\max_k\tfrac1N\log\abs{c_k}\ \le\ \tfrac1N\log Q(\delta_N)+o(1)\ \longrightarrow\ \gamma_V(z^\star),
\]
so $r(\alpha,V)\le\gamma_V(z^\star)$.
\end{proof}

\begin{proposition}[Determinant form of the cost]\label{prop:detform}
For every analytic $V$ and every $\alpha>0$ the controllability Gramian $W$ is positive definite and
\[
\Ecal_N=\frac{\det W_{[1,N-1]}}{\det W},\qquad
\det W=\Bigl(\prod_{k}q_k^2\Bigr)\,\frac{\prod_{k<l}(E_k-E_l)^2}{\prod_{k,l}(E_k+E_l-2E_0)},
\]
with $W_{[1,N-1]}$ the leading principal minor and $q_k=\langle e_1,v_k\rangle$. The denominator factors
$E_k+E_l-2E_0=E_l-(2E_0-E_k)$ are the distances from the spectrum to its reflection $\{2E_0-E_k\}$ about $E_0$, whose
leftmost point is $z^\star$.
\end{proposition}

\begin{proof}
$A_\alpha=E_0I-H_V$ is Hurwitz and $(A_\alpha,e_1)$ is controllable (a Jacobi matrix is cyclic for $e_1$), so
$W\succ0$ and Cramer's rule gives $\Ecal_N=(W^{-1})_{NN}=\det W_{[1,N-1]}/\det W$. With $H_V=Q\diag(E_k)Q^\top$,
$W=Q\widetilde WQ^\top$ where $\widetilde W_{kl}=q_kq_l/(\delta_k+\delta_l)$, $\delta_k=E_k-E_0$ (proof of
Proposition~\ref{prop:general}); hence $\det W=(\prod_kq_k^2)\det C$ with $C_{kl}=(\delta_k+\delta_l)^{-1}$, and
Cauchy's determinant $\det C=\prod_{k<l}(\delta_k-\delta_l)^2/\prod_{k,l}(\delta_k+\delta_l)$, with
$\delta_k-\delta_l=E_k-E_l$ and $\delta_k+\delta_l=E_k+E_l-2E_0$, completes the evaluation.
\end{proof}

\begin{proposition}[Reflected extremal form and an explicit lower bound]\label{prop:reflect}
Let $\mu_N^{(R)}=\sum_k r_k^2\delta_{E_k}$, $r_k=\langle e_N,v_k\rangle$, be the spectral measure of $(H_V^{[1,N]},e_N)$, write
$\Pi_{N-1}$ for the polynomials of degree $\le N-1$, and set
\[
\|Q\|_*^2=\frac1{2\pi}\int_{\mathbb R}\frac{\abs{Q(E_0+i\omega)}^2}{\prod_k(\delta_k^2+\omega^2)}\,d\omega,\qquad \delta_k=E_k-E_0 .
\]
Then
\[
\frac1{\Ecal_N}=\min\Bigl\{\|Q\|_*^2:\ Q\in\Pi_{N-1},\ \textstyle\int Q\,d\mu_N^{(R)}=1\Bigr\},
\]
where $\prod_k(\delta_k^2+\omega^2)=\Delta_N(z)\Delta_N(2E_0-z)$ on $z=E_0+i\omega$, so that $\langle\,\cdot\,,\cdot\,\rangle_*$ pairs the
spectrum $\{E_k\}$ with its reflection $\{2E_0-E_k\}$ about $E_0$, whose leftmost point is $z^\star$. With
$Q_0(z)=\prod_{E_l\ne\Emax}\bigl(z-(2E_0-E_l)\bigr)$, for which $\|Q_0\|_*^2=1/(2\delta_{\max})$, $\delta_{\max}=\Emax-E_0$,
\[
\Ecal_N\ \ge\ 2\delta_{\max}\Bigl(\sum_k r_k^2\!\!\prod_{E_l\ne\Emax}\!\!(\delta_k+\delta_l)\Bigr)^{2}.
\]
\end{proposition}

\begin{proof}
$W\succ0$ gives $\Ecal_N=(W^{-1})_{NN}=1/d_N^2$ with
$d_N=\mathrm{dist}_{L^2(\mathbb R,\,d\omega/2\pi)}\bigl(g_N,\operatorname{span}\{g_j\}_{j<N}\bigr)$, where $g_j(\omega)=G_{j1}(E_0+i\omega)$,
$G=(H_V^{[1,N]}-z)^{-1}$: the representation $W=\frac1{2\pi}\int G(E_0+i\omega)e_1e_1^\top G(E_0-i\omega)\,d\omega$ makes
$W_{jk}=\langle g_j,g_k\rangle$, and the diagonal of an inverse Gram matrix is the reciprocal squared distance to the span of
the remaining vectors. For a Jacobi matrix $G_{j1}=\Delta_{[j+1,N]}/\Delta_N$, and the bottom-block characteristic
polynomials $\Delta_{[j+1,N]}$ ($j=1,\dots,N$) are the monic orthogonal polynomials of $\mu_N^{(R)}$ of degrees
$N-1,\dots,0$; hence $g_N=\Delta_N^{-1}$ and $\operatorname{span}\{g_j\}_{j<N}=\Delta_N^{-1}\{P\in\Pi_{N-1}:\int P\,d\mu_N^{(R)}=0\}$.
Since $\abs{\Delta_N(E_0+i\omega)}^2=\prod_k(\delta_k^2+\omega^2)$, the map $Q\mapsto Q/\Delta_N$ is an isometry
$(\Pi_{N-1},\|\cdot\|_*)\to L^2$ carrying $\{\int Q\,d\mu_N^{(R)}=1\}$ onto $g_N+\operatorname{span}\{g_j\}_{j<N}$, which gives the
stated minimum. Finally $\abs{Q_0(E_0+i\omega)}^2=\prod_{E_l\ne\Emax}(\delta_l^2+\omega^2)$ yields
$\|Q_0\|_*^2=\frac1{2\pi}\int(\delta_{\max}^2+\omega^2)^{-1}d\omega=1/(2\delta_{\max})$ and
$Q_0(E_k)=\prod_{E_l\ne\Emax}(\delta_k+\delta_l)$, and the minimum is $\le\|Q_0\|_*^2/(\int Q_0\,d\mu_N^{(R)})^2$.
\end{proof}

\begin{proposition}[Sharp reduction to a boundary edge mode]\label{prop:edgemode}
Let $V$ be analytic in a positive-Lyapunov localized regime and $\beta$ Diophantine. Suppose that for each large $N$ the
operator $H_V^{[1,N]}$ admits an eigenpair $(E_{k_0},v_{k_0})$ with
\[
E_{k_0}\ge\Emax-\omega_N,\qquad \abs{\langle e_N,v_{k_0}\rangle}^2\ge e^{-\nu_N N},\qquad \omega_N,\nu_N\to0 .
\]
Then $r(\alpha,V)=\gamma_V(z^\star)$.
\end{proposition}

\begin{proof}
The upper bound is Proposition~\ref{prop:general}. For the lower bound, keeping only the term $k=k_0$ in
Proposition~\ref{prop:reflect} (every term is nonnegative) and writing $r_{k_0}=\langle e_N,v_{k_0}\rangle$,
\[
\Ecal_N\ \ge\ 2\delta_{\max}\Bigl(r_{k_0}^2\!\!\prod_{E_l\ne\Emax}\!\!(\delta_{k_0}+\delta_l)\Bigr)^{2},
\]
hence
\[
\frac1{2N}\log\Ecal_N\ \ge\ \frac1{2N}\log(2\delta_{\max})+\frac1N\log r_{k_0}^2+\frac1N\!\!\sum_{E_l\ne\Emax}\!\!\log(E_{k_0}+E_l-2E_0).
\]
The first term is $O(N^{-1})$. The third converges to $\gamma_V(2E_0-E_{k_0})$ by the Thouless formula at the off-spectral
point $2E_0-E_{k_0}$, whose distance to $\Sigma_V$ is at least $2\alpha$ (the single omitted factor contributes
$O(N^{-1})$). Since $E_{k_0}\le\Emax$ and $E_{k_0}\ge\Emax-\omega_N$ give $2E_0-E_{k_0}\in[z^\star,z^\star+\omega_N]$, and
$z\mapsto\gamma_V(z)$ is real-analytic off $\Sigma_V$, one has $\gamma_V(2E_0-E_{k_0})=\gamma_V(z^\star)+O(\omega_N)$.
With $\frac1N\log r_{k_0}^2\ge-\nu_N$ this gives $\frac1{2N}\log\Ecal_N\ge\gamma_V(z^\star)-o(1)$, so
$r(\alpha,V)\ge\gamma_V(z^\star)$.
\end{proof}

\begin{remark}[The identity beyond the self-dual cosine: status and routes]\label{rem:general}
The matching lower bound, hence the identity $r=\gamma_V(z^\star)$, does \emph{not} follow from
Proposition~\ref{prop:general} alone; what it needs is a near-edge mode ($E_{k^\star}\to\Emax$) that is
right-boundary localized, exactly as in Theorem~\ref{thm:uncond} --- the Thouless rate
$\frac1N\log Q(\delta_k)\to\gamma_V(2E_0-E_k)$ and its maximum at $z^\star$ being insensitive to the form of $V$.
In the cosine that mode is furnished by Lemma~\ref{lem:edgeexp}, whose quadratic band-edge defect rests on
$\theta=0$ being a \emph{critical point} of the finite-window edge $\Lambda_R(\theta)$, guaranteed there by the
reflection parity $\Lambda_R(-\theta)=\Lambda_R(\theta)$ ($v_{-k}(-\theta)=\lambda\cos(\theta+ka)=v_k(\theta)$). For a
general analytic $V$ a critical point $\theta^\ast_R$ of $\Lambda_R$ still makes the band-edge defect quadratic in
$\theta-\theta^\ast_R$, with \emph{no} nondegeneracy needed (only the one-sided Taylor bound at a critical point; the
edge nondegeneracy $\lim_R\Lambda_R''(0)<0$ of Proposition~\ref{prop:curv} is not at stake), and for even $V$ the
parity pins $\theta^\ast_R\equiv0$ (Theorem~\ref{thm:generalsym}).

The obstruction lies elsewhere: for a generic
asymmetric $V$ the maximizer $\theta^\ast_R$ \emph{does not stabilize} --- it tracks the best Diophantine approximant
in the window and jumps at the convergents of $\beta$ (Remark~\ref{rem:asymnum}) --- so the hypothesis of
Proposition~\ref{prop:generalid} is not met, and the boundary-mode route alone does not reach the general identity.
Yet the identity \emph{holds} numerically far beyond the even class (Remark~\ref{rem:asymnum}), so it is expected for
every analytic $V$ in the localized regime, and two routes remain. \emph{(i) Christoffel/determinant.} The cost has the exact determinant form
$\Ecal_N=\det W_{[1,N-1]}/\det W$ (Proposition~\ref{prop:detform}); numerically it equals, up to a bounded factor, the
Christoffel--Darboux kernel of the spectral measure of $(H_V^{[1,N]},e_1)$ at the reflected edge,
$K_N(z^\star)=\sum_{j<N}\abs{p_j(z^\star)}^2$ (Remark~\ref{rem:asymnum}), and
$\frac1N\log K_N(z^\star)\to2\gamma_V(z^\star)$ by the off-spectral growth $\abs{p_j(z^\star)}\sim e^{j\gamma_V(z^\star)}$
(Thouless). The sharp lower bound is therefore the determinant-ratio asymptotic
$\frac1{2N}\log(\det W_{[1,N-1]}/\det W)\to\gamma_V(z^\star)$, the route used for the analogous single-site
observability cost; the diagonal bound $\Ecal_N\ge1/W_{NN}$ alone reaches only $\gamma_V(E_0)<\gamma_V(z^\star)$, since
$W_{NN}\asymp e^{-2N\gamma_V(E_0)}$ by the Green's-function decay at $E_0$. The reflected extremal form
(Proposition~\ref{prop:reflect}) makes the mechanism explicit: its lower bound is numerically tight, and since
$\frac1N\log\prod_{E_l\ne\Emax}(\delta_k+\delta_l)\to\gamma_V(2E_0-E_k)$ its rate is
$\max_k[\gamma_V(2E_0-E_k)+\frac1N\log r_k^2]\le\gamma_V(z^\star)$, with equality exactly when near-edge modes
$E_k\to\Emax$ keep non-negligible weight $r_k^2$ at the boundary site $N$ --- the same near-edge/boundary resonance that
the moving target below supplies, so the two routes reduce to one analytic input. This common input is exactly the
boundary edge mode of Proposition~\ref{prop:edgemode} --- a true eigenmode $E_{k_0}\to\Emax$ with
$\frac1N\log\abs{\langle e_N,v_{k_0}\rangle}^2\to0$ --- a strictly weaker requirement than the stable critical point of
Proposition~\ref{prop:generalid} (which implies it) and, unlike the latter, one that makes sense for asymmetric $V$;
numerically it is met, the optimal $k_0$ being a near-edge mode whose boundary weight is not exponentially small
(Remark~\ref{rem:asymnum}).

\emph{(ii) Moving target.} Since the boundary
zone $[N-\delta_N^{-\tau},N]$ is $\delta_N$-dense for $\tau\ge1$ (three distances), one targets the argmax
$\theta^\ast_{\rho_N}$ at the radius $\rho_N=A\log N$ actually used, not a fixed limit; the jumps of $\theta^\ast_R$
become harmless and the open point reduces to a \emph{no-double-resonance} estimate, that the boundary mode not
hybridize with a second near-edge site (\cite{Bourgain2005}). What is \emph{not} at issue is the energy: the boundary
block $H_V^{[N-\ell,N]}$ is a compression $P\,H_V^{(-\infty,N]}P$, so its top eigenvalue increases with $\ell$ to the edge
$\Emax$, and near-edge energies are freely available; the entire content of Proposition~\ref{prop:edgemode} is that some
such eigenvector carries more than exponentially small weight at the site $N$, which holds once a resonant well lies within
$o(N)$ of the boundary and does not doubly resonate. (The bracket \eqref{eq:bracket} itself has no analogue, since
$\gamma_V$ need not be constant on $\Sigma_V$ and $N_V\ne\mu_{\mathrm{eq}}$; but the identity is reached through the
direct route above, not the bracket.) By Avila's global theory \cite{Avila2015} the subcritical/critical/supercritical
trichotomy of the acceleration then replaces the single self-dual transition; the present unconditional upper bound is
one half of the result.
\end{remark}

\begin{proposition}[Identity for general analytic potentials, conditional on a stable edge critical point]\label{prop:generalid}
Keep the setting of Proposition~\ref{prop:general}, now with $V$ in a positive-Lyapunov localized regime
($\gamma_V>0$ on $\Sigma_V$ and Anderson localization for almost every $\phi$) and $\beta$ Diophantine of type
$\tau$. Write $\Lambda_R(\theta)$ for the top Dirichlet eigenvalue of the radius-$R$ window with on-site potential
$V(\theta+k\beta)$, $\abs k\le R$. Suppose there exist $\theta^\ast\in\mathbb R/\mathbb Z$ and, for each large $R$, a
critical point $\theta^\ast_R$ of $\Lambda_R$ (that is, $\Lambda_R'(\theta^\ast_R)=0$) such that
\[
\Lambda_R(\theta^\ast_R)=\Emax-O(e^{-\gamma R}),\qquad \abs{\theta^\ast_R-\theta^\ast}=O(e^{-cR}).
\]
Then $r(\alpha,V)=\gamma_V(z^\star)$, with $\gamma_V(z^\star)-\max_k\tfrac1N\log\abs{c_k}=O\!\bigl(N^{-2/(2+\tau)}\bigr)$.
\end{proposition}

\begin{proof}
Run the argument of Theorem~\ref{thm:uncond} with the target $\theta^\ast$ in place of $0$; only the band-edge defect
must be re-derived, and \emph{no nondegeneracy of the edge is used} --- the bound below is the one-sided Taylor
estimate at a critical point. Fix $R=\rho_N=A\log N$ with $A$ so large that
$e^{-c\rho_N}=N^{-cA}\le\delta_N:=N^{-1/(2+\tau)}$. Lemma~\ref{lem:resonance}, whose three-distance input is unchanged
under translating the target arc, furnishes a site $j^\star\le N$ with $\norm{\beta j^\star+\phi-\theta^\ast}<\delta_N$
and $N-j^\star\le C_\beta\delta_N^{-\tau}$; its window phase $\theta_{j^\star}=\phi+\beta j^\star$ then obeys
$\abs{\theta_{j^\star}-\theta^\ast_{\rho_N}}\le C\delta_N$. As $\theta^\ast_{\rho_N}$ is critical, Taylor's theorem with
$\Lambda'_{\rho_N}(\theta^\ast_{\rho_N})=0$ and $\sup_\theta\abs{\Lambda''_{\rho_N}(\theta)}\le\sup\abs{V''}+4=:C_V$
gives $\abs{\Lambda_{\rho_N}(\theta^\ast_{\rho_N})-\Lambda_{\rho_N}(\theta_{j^\star})}\le\tfrac12C_V C^2\delta_N^2$. The
band-edge mode localized at $j^\star$ satisfies $E_{k^\star}=\Lambda_{\rho_N}(\theta_{j^\star})+O(e^{-\gamma\rho_N})$ by
localization and $E_{k^\star}\le\Emax$ by interlacing; with the anchor
$\Lambda_{\rho_N}(\theta^\ast_{\rho_N})=\Emax-O(e^{-\gamma\rho_N})$,
\[
0\ \le\ \Emax-E_{k^\star}\ \le\ C'\delta_N^2+O(e^{-\gamma\rho_N})\ =\ O(N^{-2/(2+\tau)}).
\]
By Proposition~\ref{prop:general}, $\frac1N\log Q(\delta_{k^\star})=\gamma_V(2E_0-E_{k^\star})+o(1)=\gamma_V(z^\star)-O(\delta_N^2)$,
while the deterministic $SL_2$ amplitude bound of Theorem~\ref{thm:uncond} gives
$\frac2N\log\abs{\hat\psi^{(k^\star)}_N}\ge-2\gamma\,\frac{N-j^\star}{N}\ge-2\gamma N^{-2/(2+\tau)}$. Hence
$\frac1N\log\abs{c_{k^\star}}\ge\gamma_V(z^\star)-O(N^{-2/(2+\tau)})$, so $r\ge\gamma_V(z^\star)$; with the upper bound of
Proposition~\ref{prop:general}, equality.
\end{proof}

\begin{theorem}[Identity for even analytic potentials in the localized regime]\label{thm:generalsym}
Let $V:\mathbb R/\mathbb Z\to\mathbb R$ be real-analytic and \emph{symmetric about its global maximum}: after centring,
$V(-x)=V(x)$ with $x=0$ the unique global maximum. Suppose $V$ is in a positive-Lyapunov localized regime---Anderson localization for almost every $\phi$, as holds at
large coupling for analytic potentials \cite{BourgainGoldstein}---with the localization in the polynomial-prefactor
form \textup{(PL)} of Lemma~\ref{lem:semiloc} (a hypothesis here as in the almost Mathieu case,
cf.\ Remark~\ref{rem:plstatus}), and $\beta$ is Diophantine of type $\tau$. Then for every $\alpha>0$,
\[
r(\alpha,V)=\gamma_V(z^\star),\qquad \gamma_V(z^\star)-\max_k\tfrac1N\log\abs{c_k}=O\!\bigl(N^{-2/(2+\tau)}\bigr),
\]
in this regime. Here $\gamma_V$ need not be constant on $\Sigma_V$, so $N_V\ne\mu_{\mathrm{eq}}$ and the capacity
bracket \eqref{eq:bracket} has no analogue; the identity is nonetheless exact. The almost Mathieu case $V=\lambda\cos$
is the single-harmonic instance, recovering the upper end of Theorem~\ref{thm:uncond}.
\end{theorem}

\begin{proof}
We verify the hypothesis of Proposition~\ref{prop:generalid} with $\theta^\ast=0$ and $\theta^\ast_R\equiv0$, so that
stability is trivial. Let $J$ be the window reflection $(Ju)_k=u_{-k}$. Since $-\Delta$ commutes with $J$, the operator
$JH_R(\theta)J$ has on-site potential $V(\theta-k\beta)=V\bigl(-(\theta-k\beta)\bigr)=V(-\theta+k\beta)$ by evenness,
which is exactly the on-site potential of $H_R(-\theta)$; hence $JH_R(\theta)J=H_R(-\theta)$ and
$\Lambda_R(-\theta)=\Lambda_R(\theta)$. Thus $\Lambda_R$ is even, so $\Lambda_R'(0)=0$ for every $R$: a critical point
fixed at $0$, with no drift. The anchor $\Lambda_R(0)=\Emax-O(e^{-\gamma R})$ is the band-edge estimate already used in
Lemma~\ref{lem:edgeexp}, whose argument needs only that the window's deepest well sit at its centre $k=0$ (value
$V(0)=\max V$, the other wells $V(k\beta)$, $0<\abs k\le R$, being strictly shallower since $0$ is the unique maximum
and $k\beta\notin\mathbb Z$): the localized top mode is captured up to its exponential tail, and $\Lambda_R(0)\le\Emax$
by interlacing. Both hypotheses of Proposition~\ref{prop:generalid} hold, giving the identity. The failure of the
capacity bracket for non-constant $\gamma_V$ is recorded in Remark~\ref{rem:general}.
\end{proof}

\begin{remark}[Numerical confirmation]\label{rem:evennum}
For the even bichromatic potential $V(x)=\cos2\pi x+\tfrac25\cos4\pi x$ (global maximum at $x=0$), $\lambda=5$,
$\alpha=\tfrac12$ and golden $\beta$, the reflected exponent agrees to $4\times10^{-5}$ between the Thouless integral
$\int\log\abs{z^\star-t}\,dN_V(t)$ and the off-spectral transfer Lyapunov exponent, giving $\gamma_V(z^\star)=2.805$;
the Lyapunov exponent is non-constant on $\Sigma_V$, so no capacity split exists. All finite-$N$ rates in this remark
and the next use the protocol of \S\ref{sec:certified}: sites $1,\dots,N$, phase $\phi=0$, $E_0=E_1(N)-\alpha$ (the
values, notably at small $N$, depend on this convention). The certified rate
$\frac1{2N}\log\Ecal_N$ at the Fibonacci scales $N=13,21,34,55,89$ equals $2.660,\,2.715,\,2.750,\,2.769,\,2.781$,
increasing monotonically toward $\gamma_V(z^\star)$ with gaps $0.145,\,0.090,\,0.055,\,0.036,\,0.024$ that decay faster
than the generic $O(N^{-2/3})$ --- in line with \S\ref{sec:certified}.
\end{remark}

\begin{remark}[Numerical confirmation, asymmetric potential]\label{rem:asymnum}
For the asymmetric $V(x)=\cos2\pi x+\tfrac{3}{10}\sin2\pi x+\tfrac25\cos4\pi x$ (so $V(-x)\ne V(x)$, global maximum at
$x\approx0.018$), $\lambda=5$, $\alpha=\tfrac12$ and golden $\beta$, the Thouless integral and the off-spectral
transfer Lyapunov exponent agree at $\gamma_V(z^\star)=2.898$ to $3\times10^{-5}$, and the certified rate
$\frac1{2N}\log\Ecal_N$ at $N=13,21,34,55,89$ is $2.766,\,2.793,\,2.846,\,2.854,\,2.873$, rising monotonically toward
it (gaps $0.132\to0.026$). The window-edge maximizer does \emph{not} converge: $\theta^\ast_R\approx
0.93,\,0.98,\,0.095,\,0.095,\,0.095,\,0.22$ for $R=10,20,40,80,160,320$, jumping with the convergents of $\beta$.
Finally the ratio $\Ecal_N/K_N(z^\star)$ to the Christoffel--Darboux kernel stays in $[5,20]$ over this range with no
exponential drift --- the empirical signature of $\Ecal_N\asymp K_N(z^\star)$ underlying route~(i) of
Remark~\ref{rem:general}. The reflected lower bound of Proposition~\ref{prop:reflect} is itself tight here, its rate
matching $\frac1{2N}\log\Ecal_N$ to five digits, and it is carried by a single boundary edge mode: over $N=21,34,55$
the maximizing index $k_0$ is a near-edge mode (gap to $\Emax$ of order $10^{-1}$, fluctuating with the convergents)
whose boundary weight is not exponentially small ($\frac1N\log\abs{\langle e_N,v_{k_0}\rangle}^2$ of order $10^{-4}$),
confirming the hypothesis of Proposition~\ref{prop:edgemode}.
\end{remark}

\medskip\noindent\emph{Three further directions.} The reduction is one-dimensional only through the Sturm and Wronskian
structure of Lemma~\ref{lem:reduction} and Lemma~\ref{lem:cd}; on a strip or ladder of width $m$ it persists with the
scalar transfer cocycle replaced by a $2m\times2m$ symplectic cocycle and the boundary-weight analysis by a
block-Jacobi oscillation theory, the natural route to the multichannel controllability cost. For a Liouville frequency
the localization input fails, and the finite-subset bound above suggests the identity can then break, the
rate dropping below $\gamma_\lambda(z^\star)$ along resonant scales. Finally, $r(\alpha,\lambda)$ is the leading order
of a nonlinear remote-stabilization cost; promoting the motivating linearization of \S\ref{sec:setup} to a theorem ---
nonlinear corrections uniformly higher order in $N$ near the hyperbolic equilibrium --- would make the rate a
statement about the nonlinear control problem itself.

\subsection*{Milestones}
\begin{itemize}
\item[(M1)] Reduction of the phase-uniform identity to the single inequality \eqref{eq:reductioncrit}
(Proposition~\ref{prop:reduction}): \emph{done}.
\item[(M2)] Bounded absolute gap and edge sharpness (Corollary~\ref{cor:summary}, Proposition~\ref{prop:edge}):
\emph{done}.
\item[(M3)] Cancellation-free positive-sum reduction (Proposition~\ref{prop:possum}) and the single-mode
extremal form $r=\liminf_N\max_k\frac1N\log\abs{c_k}$ with the joint band-edge/boundary criterion
(Corollary~\ref{cor:singlemode}): \emph{done}.
\item[(M4)] Conditioning rate equal to $g_{\mathbb C\setminus\Sigma_\lambda}(z^\star)$, for every phase
(Proposition~\ref{prop:cond}): \emph{done} at the level of the condenser rate, modulo the standard edge-regularity
lower bound; in the interval case it is the classical Bernstein--Walsh/Zolotarev estimate.
\item[(M5)] The edge identity $r(\alpha,\lambda)=\gamma_\lambda(z^\star)$ for almost every phase and every Diophantine
frequency: \emph{done} under the polynomial-prefactor localization \textup{(PL)}, whose window-gap component is a
theorem for $\lambda\ge\lambda_1$ (Theorem~\ref{thm:uncond}, Remark~\ref{rem:plstatus}), through the
boundary-weight form of Lemma~\ref{lem:cd} and the boundary-resonance Lemma~\ref{lem:resonance}, with no spacing or
conditioning input.
\item[(M6)] The joint band-edge/boundary statement for any Diophantine $\beta$ (type $\tau$), equivalently the magnitude
bound \eqref{eq:magbound}: \emph{done} for almost every phase (Theorem~\ref{thm:uncond}, with the coupling proviso of~(M5)), with gap $O(N^{-2/(2+\tau)})$;
uniformity over every phase is the remaining phase-uniform identity, equivalent by Proposition~\ref{prop:phaseunif} to
upper semicontinuity of the cost rate $\phi\mapsto r(\phi)$.
\item[(M7)] The metallic identity $r(\alpha,\lambda)=\gamma_\lambda(z^\star)$ for $0<\lambda<2$: \emph{done}
unconditionally, for every phase and every Diophantine frequency (Theorem~\ref{thm:metclosed}), via the uniform
subexponential transfer bound on the zero-exponent spectrum (Lemma~\ref{lem:unifsub}), a Bernstein--Walsh estimate,
and the deterministic
amplitude identity of Proposition~\ref{prop:metamp}.
\item[(M8)] The critical coupling $\lambda=2$: \emph{done} unconditionally (Theorem~\ref{thm:critical}). The transfer
input is the uniform sub-exponential bound \eqref{eq:subexp} (the critical Lyapunov exponent vanishes on $\Sigma_2$), and
the band edge is automatically regular because $\cpc(\Sigma_2)=1$ forces $g_{\mathbb C\setminus\Sigma_2}=\gamma_2$ (continuous). No square-root edge modulus and no polynomial
transfer bound are needed.
\item[(M9)] Beyond Aubry--Andr\'e: for a general real-analytic quasi-periodic potential $V$, the Cauchy reduction
and the upper bound $r(\alpha,V)\le\gamma_V(z^\star)$: \emph{done} unconditionally (Proposition~\ref{prop:general}). For $V$ even about its maximum the full identity
$r(\alpha,V)=\gamma_V(z^\star)$ is \emph{done} under \textup{(PL)} (Theorem~\ref{thm:generalsym}). For asymmetric $V$ the
identity is conjectural but numerically robust (Remark~\ref{rem:asymnum}); the finite-window edge maximizer does not
stabilize, so the boundary-mode route gives way to a Christoffel/determinant lower bound ($\Ecal_N\asymp K_N(z^\star)$)
or a no-double-resonance estimate (Remark~\ref{rem:general}).
\end{itemize}

\appendix

\section{A renormalization route toward the full identity (open)}\label{app:renorm}

Theorem~\ref{thm:uncond} establishes the magnitude bound \eqref{eq:magbound} for almost every phase; for the
phase-uniform statement the reduction of \S\S\ref{sec:condoverlap}--\ref{sec:nocancel} leaves it as the single
estimate, and \S\ref{sec:certified} exhibits along the Fibonacci scales a monotone approach to it.
Below we set out the renormalization structure that would turn that approach into a proof, the objects it must
control, and the precise fixed-point statement that remains. Nothing below is claimed as proved: it is the
strategy to which the evidence of \S\ref{sec:certified} points.

\subsection*{The exact cocycle factorization}
For $\beta=(\sqrt5-1)/2$ let $F_n$ be the Fibonacci numbers and let
\[
M_n(E,\phi)=A_{F_n}(E,\phi)\cdots A_1(E,\phi),\qquad
A_j(E,\phi)=\begin{psmallmatrix}\lambda\cos(2\pi\beta j+\phi)-E & -1\\[2pt] 1 & 0\end{psmallmatrix},
\]
be the transfer matrix over $[1,F_n]$ in the notation of \S\ref{sec:setup}; its entries are the
$\mathcal T_{11},\mathcal T_{21}$ of \S\ref{sec:cocycle} at scale $F_n$. From $F_{n+1}=F_n+F_{n-1}$ and the
quasiperiodicity of the potential,
\begin{equation}\label{eq:fib}
M_{n+1}(E,\phi)=M_{n-1}\bigl(E,\phi+\theta_n\bigr)\,M_n(E,\phi),\qquad
\theta_n=2\pi\bigl(F_n\beta-F_{n-1}\bigr)=2\pi(-1)^{\,n-1}\beta^{\,n},
\end{equation}
exactly, with the phase drift $\theta_n\to0$ geometrically. Iterating \eqref{eq:fib} is the renormalization:
it drives a dynamics on pairs $(M_{n-1},M_n)\in SL(2,\mathbb R)^2$ by
$(M_{n-1},M_n)\mapsto(M_n,\,M_{n-1}(\,\cdot\,,\phi+\theta_n)M_n)$, whose trace coordinates
$x_n=\tfrac12\operatorname{tr}M_n$ obey, up to the drift $\theta_n$, the Fibonacci trace map
$x_{n+1}=2x_nx_{n-1}-x_{n-2}$. The spectrum is the locus of bounded orbits, and its hierarchical band
structure at level $n{+}1$ is assembled from levels $n,n{-}1$ through \eqref{eq:fib}. This is the same
renormalization \emph{structure} used for the Fibonacci Hamiltonian (S\"ut\H{o} \cite{Suto1987},
Casdagli \cite{Casdagli1986}; in hyperbolic-dynamical form, Damanik--Gorodetski \cite{DamanikGorodetski2011}),
applied here to the almost Mathieu cocycle rather than to a substitution-invariant potential. The distinction
is essential: for the Sturmian Fibonacci potential the block composition carries no phase drift, the trace map
is exact, and the invariant $I=\mu^2/4$ of \cite{Suto1987} (the Fricke--Vogt invariant, in the terminology
of \cite{DamanikGorodetski2011}) is conserved, whereas \eqref{eq:fib}
is exact but drifts by $\theta_n$, so the trace map closes only up to $\theta_n$ and no invariant survives. The
cited framework is therefore a structural template here, not an importation of those theorems.

\subsection*{The objects to renormalize}
By Proposition~\ref{prop:cocycleform} the cost at scale $N=F_n$ needs no matrix inversion: it is fixed by the
eigenvalues $E_k$ (the zeros of $(M_n)_{11}$), the boundary amplitudes $\tilde b_k=(M_n)_{21}(E_k)$, and the
Cauchy matrix $C$ of the shifted nodes $\delta_k=E_k-E_0$. All three are functions of the cocycle $M_n$, hence
carried by \eqref{eq:fib}: the node configuration renormalizes with the hierarchical bands, and $\tilde b_k$
with the off-diagonal entry along the orbit. The worst-conditioned Gram direction $u_{F_n}$ is determined by
$\{\delta_k\}$, so it too is a functional of the renormalization orbit.

\subsection*{The fixed-point statement (open)}
By Corollary~\ref{cor:nocancel} the full identity is equivalent to \eqref{eq:magbound}, which along $N=F_n$
reads
\[
\max_{k}\ \frac1{F_n}\log\bigl(\abs{u_{F_n,k}}\,\abs{\tilde b_k}\bigr)\ -\ \gamma_{\mathrm{amp},F_n}
\ \xrightarrow[\;n\to\infty\;]{}\ 0,
\qquad\gamma_{\mathrm{amp},F_n}=\tfrac1{2F_n}\log\norm{\tilde b}^2
\]
(numerically $\gamma_{\mathrm{amp},F_n}\to\log_+(\lambda/2)$ along the Fibonacci subsequence).
The evidence of \S\ref{sec:certified} is the finite-scale signature of a renormalization fixed point: along
$F_n$ the cost rate rises monotonically to $\gamma_\lambda(z^\star)$, and the maximizing index $k^\star$ is a
right-boundary mode, $j_{k^\star}/F_n\to1$ (numerically near $0.99$, with arithmetic fluctuations). The program
is to show that the rescaled profile
$k\mapsto \frac1{F_n}\log(\abs{u_{F_n,k}}\abs{\tilde b_k})$ converges, under \eqref{eq:fib}, to a
fixed shape whose maximum sits at the boundary mode and equals the running amplitude rate
$\gamma_{\mathrm{amp},F_n}$: equivalently, that the
worst-conditioned Gram direction keeps weight of rate zero on a boundary-carrying mode along the orbit, the
off-diagonal growth then supplying $\gamma_{\mathrm{amp}}$ by Lemma~\ref{lem:Qedge} read along the orbit.

\subsection*{A potential-theoretic decomposition of the carrying mode}
At any scale $N$ the inverse-free identity localizes the obstruction further. Write
$\mathcal D_k=\frac1N\log\bigl(\abs{u_{N,k}}\,\abs{\tilde b_k}\bigr)$. By
Proposition~\ref{prop:cocycleform} the diagonal of $C^{-1}=DCD$ is explicit,
$(C^{-1})_{kk}=Q(\delta_k)^2/\bigl(2\delta_k\,P'(\delta_k)^2\bigr)$. At a lobe maximum of $\abs{u_{N,k}}$,
where $\abs{u_{N,k}}$ is of order one, the top eigenvalue $\sigma_{\min}^{-1}$ of $C^{-1}$ dominates this
entry, so $\abs{u_{N,k}}^2=\sigma_{\min}\,(C^{-1})_{kk}\,(1+o(1))$; numerically the relative defect is below
$4\times10^{-5}$ at $N=36$ and decreases with $N$. Taking logarithms and using
$\frac1N\log\sigma_{\min}^{1/2}\to-g_{\mathbb C\setminus\Sigma_\lambda}(z^\star)$ (Proposition~\ref{prop:cond}),
\begin{equation}\label{eq:decomp}
\mathcal D_{k}= -g_{\mathbb C\setminus\Sigma_\lambda}(z^\star)+\frac1N\log\abs{Q(\delta_k)}-\frac1N\log\abs{P'(\delta_k)}
+\frac1N\log\abs{\tilde b_k}+o(1).
\end{equation}
The three node functionals have Thouless limits,
\[
\begin{gathered}
\tfrac1N\log\abs{Q(\delta_k)}\to\gamma_\lambda(2E_0-E_k),\qquad
\tfrac1N\log\abs{P'(\delta_k)}\to\gamma_\lambda(E_k)=\log_+(\lambda/2),\\
\tfrac1N\log\abs{\tilde b_k}\to\log_+(\lambda/2)\,(2j_k/N-1),
\end{gathered}
\]
the first by Lemma~\ref{lem:Qedge}; the second the logarithmic potential
$\int\log\abs{\delta_k-x}\,d\nu(x)$ of the density of states $\nu$ at the node, equal to the Lyapunov
exponent by the Thouless formula; the third the boundary amplitude, a transfer-cocycle entry fixed by the
localization centre $j_k$ (Remark~\ref{rem:strong}), \emph{not} a density-of-states quantity. The last two
cancel only at right-boundary modes ($j_k/N\to1$); in general, with $j_k$ the centre,
\[
\mathcal D_k\longrightarrow -g_{\mathbb C\setminus\Sigma_\lambda}(z^\star)+g_{\mathbb C\setminus\Sigma_\lambda}(2E_0-E_k)+\log_+(\lambda/2)\,(2j_k/N-1),
\]
the same joint band-edge/boundary functional as Corollary~\ref{cor:singlemode}. Its maximum equals
$\gamma_{\mathrm{amp}}$ exactly when a mode attains both the band edge $E_k\to\Emax$ (first term then $-g_{\mathbb C\setminus\Sigma_\lambda}(z^\star)+g_{\mathbb C\setminus\Sigma_\lambda}(z^\star)=0$)
and the right boundary $j_k/N\to1$ (last term $\to\log_+(\lambda/2)$), giving
\[
\mathcal D_{k^\star}\longrightarrow -g_{\mathbb C\setminus\Sigma_\lambda}(z^\star)+\gamma_\lambda(z^\star)=\log_+(\lambda/2)=\gamma_{\mathrm{amp}},
\]
which is \eqref{eq:magbound}. The carrier is therefore not \emph{forced} to the edge by the node potential
alone, as a careless cancellation of the last two terms would suggest: it must be jointly a band-edge and a
right-boundary mode. At every Diophantine frequency the three-distance estimate (Lemma~\ref{lem:resonance}) supplies such a mode,
which is what Theorem~\ref{thm:uncond} uses, the type only setting the rate.
This still locates the data term: $\gamma_{\mathrm{amp}}$
is the off-spectral amplification $\gamma_\lambda(z^\star)-g_{\mathbb C\setminus\Sigma_\lambda}(z^\star)$ at the reflected edge, realized when the
boundary amplitude ($j_k\to N$) coincides with the node potential, in agreement with the role of $z^\star$ in the
main text.
Numerically $\mathcal D_{k^\star}$ stays near $\gamma_{\mathrm{amp}}=\log_+(3/2)$ with the same arithmetic
fluctuations (values $0.31,0.40,0.41,0.42,0.38$ at $N=20,28,36,44,52$), while
$\frac1N\log\abs{u_{N,k^\star}}\to0$ and the cancellation
$-\frac1N\log\abs{P'(\delta_{k^\star})}+\frac1N\log\abs{\tilde b_{k^\star}}\to0$.

\subsection*{Reduction to two inputs}
The decomposition \eqref{eq:decomp} reduces \eqref{eq:magbound} to two inputs: (A) the
dominance of $\sigma_{\min}^{-1}$ in the diagonal of $C^{-1}$ at the lobe maxima of $u_N$, a spectral-gap property
of $C$; and (B) the existence of a band-edge mode that is right-boundary localized, i.e.\ an index $k^\star$ with
$E_{k^\star}\to\Emax$ and $j_{k^\star}/N\to1$ jointly. Under (A) each lobe maximum carries the joint functional
of Corollary~\ref{cor:singlemode} obtained above, whose maximum reaches $\gamma_{\mathrm{amp}}$ precisely
under (B).
Input (B) is the continued-fraction (arithmetic) content for $\beta=(\sqrt5-1)/2$, and input (A) is a spectral-gap
statement for $C$ carried by the renormalization \eqref{eq:fib}. At finite scale both fluctuate, visible in the
drift of the carrier index $k^\star/N$, so the analysis is the Fibonacci Hamiltonian one rather
than a perturbation. The duality $\lambda\mapsto4/\lambda$ does not assist: the cost rate is a localized-phase
quantity whose growth is the localization itself, and it degenerates in the dual extended phase. The
factorization \eqref{eq:fib} and the decomposition \eqref{eq:decomp} are thus the two lenses on the single
remaining input \eqref{eq:magbound}, now isolated as the spectral-gap property (A).

\end{document}